\documentclass[a4paper,11pt]{article}

\usepackage{amsmath}
\usepackage{color}
\usepackage[latin1]{inputenc}
\usepackage[T1]{fontenc}
\usepackage[normalem]{ulem}
\usepackage[english]{babel}
\usepackage{verbatim}
\usepackage{alphabeta}
\usepackage{mathtools}
\usepackage{graphicx}
\usepackage{enumerate,enumitem}
\usepackage{amsmath,amssymb,amsfonts,amsthm,mathrsfs}
\usepackage{rotating,relsize}
\usepackage{float}
\usepackage{array}
\usepackage{MnSymbol}
\usepackage[all,cmtip]{xy}
 \usepackage[hidelinks]{hyperref}
\xyoption{all}
\setlist[enumerate]{leftmargin=.7cm,label=\roman*)}
\usepackage{appendix}

\newtheorem{theorem}{Theorem}[section]
\newtheorem{theoremA}{Theorem}

\newtheorem{lemma}[theorem]{Lemma}
\newtheorem{prop}[theorem]{Proposition}
\newtheorem{cor}[theorem]{Corollary}

\theoremstyle{definition}
\newtheorem*{remark*}{Remark}
\newtheorem{defn}[theorem]{Definition}
\newtheorem{rem}[theorem]{Remark}
\newtheorem{example}[theorem]{Example}

\DeclareMathOperator{\pr}{pr}

\DeclareMathOperator{\id}{id}
\DeclareMathOperator{\im}{Im}

\DeclareMathOperator{\coker}{coker}

\DeclareMathOperator{\map}{Map}

\DeclareMathOperator*{\colim}{colim}

\DeclareMathOperator*{\holim}{holim}

\DeclareMathOperator{\THH}{THH}
\DeclareMathOperator{\THR}{THR}
\DeclareMathOperator{\TCR}{TCR}
\DeclareMathOperator{\TC}{TC}
\DeclareMathOperator{\TRR}{TRR}
\DeclareMathOperator{\TFR}{TFR}
\DeclareMathOperator{\TR}{TR}

\DeclareMathOperator{\res}{res}

\DeclareMathOperator{\K}{K}
\DeclareMathOperator{\Sp}{Sp}

\DeclareMathOperator{\fib}{fib}

\DeclareMathOperator{\Z}{\mathbb{Z}}
\DeclareMathOperator{\F}{\mathbb{F}}
\DeclareMathOperator{\sd}{sd}
\DeclareMathOperator{\proj}{proj}
\DeclareMathOperator{\tran}{tran}

\DeclareMathOperator{\incl}{incl}

\DeclareMathOperator{\trf}{trf}

\DeclareMathOperator{\cof}{cof}
\DeclareMathOperator{\Lt}{L}
\DeclareMathOperator{\Mod}{Mod}
\DeclareMathOperator{\Gpd}{Gpd}

\newcommand{\class}[1]{\left\langle #1 \right\rangle} 

\begin{document}
\begin{center}\LARGE{On the geometric fixed-points of real topological cyclic homology}
\end{center}

\begin{center}\Large{Emanuele Dotto, Kristian Moi, Irakli Patchkoria}
\end{center}


\let\thefootnote\relax\footnotetext{\emph{2020 Mathematics Subject Classification:} Primary 19D55, 11E70; Secondary 55P91, 13F35}

\vspace{.05cm}

\abstract{
We give a formula for the geometric fixed-points spectrum of the real topological cyclic homology of a bounded below ring spectrum, as an equaliser of two maps between tensor products of modules over the norm. We then use this formula to carry out computations in the fundamental examples of spherical group-rings, perfect $\F_p$-algebras, and $2$-torsion free rings with perfect modulo $2$ reduction. Our calculations agree with the normal L-theory spectrum in the cases where the latter is known, as conjectured by Nikolaus.
}

\vspace{.05cm}

\section*{Introduction}
The trace methods were introduced in \cite{BHM} as an effective way of studying the algebraic K-theory of suitable rings, by mapping it to more computable invariants which are typically constructed from the topological Hochschild homology spectrum THH and its cyclic action. One particularly successful invariant is the topological cyclic homology TC defined from suitably derived fixed-points of the cyclic structure of THH. If one tries to extend these methods to the algebraic K-theory of forms or to cobordisms of forms (that is Grothendieck-Witt and L-theory respectively) one discovers the real topological Hochschild homology THR, a dihedral refinement of THH, and the real topological cyclic homology TCR. 

The real topological Hochschild homology  $\THR(A)$ of a ring (or ring spectrum) with anti-involution $A$ has been introduced in unpublished notes of Hesselholt and Madsen. It is an $O(2)$-equivariant spectrum whose underlying $S^1$-spectrum is $\THH(A)$, and where the subgroup $\Z/2$ of $O(2)$ generated by a reflection acts via a combination of a reflection of the circle and the anti-involution of $A$. The $\Z/2$-equivariant homotopy type of $\THR(A)$ has been studied extensively: In \cite{Amalie} it has been computed for spherical group-rings in terms of free loop spaces. In \cite{THRmodels} we studied some of its fundamental structural properties and we computed it for $\F_p$ and $\Z$. In \cite{GThom} it has been related to equivariant factorisation homology and calculated for equivariant Thom spectra. In \cite{HW} the Hopf algebroid structure on the homotopy groups of $\THR(\F_2)$ is described and used to give an independent proof of the Segal conjecture for the group of order $2$.
A key feature which makes these calculations accessible is the description of $\THR(A)$ as a derived tensor product, and in particular of its $\Z/2$-geometric fixed-points spectrum as the derived tensor product of module spectra
\[
\THR(A)^{\phi\Z/2}\simeq A^{\phi\Z/2}\otimes_AA^{\phi\Z/2},
\]
where $A$ acts on $A^{\phi\Z/2}$ on the left and on the right by the ``Frobenius actions'', described informally respectively by the formulas $a\cdot x=axw(a)$ and $x\cdot a=w(a)xa$, and $w$ is the anti-involution of $A$.

The real topological cyclic homology $\TCR(A;p)$, for a prime number $p$, is a $\Z/2$-equivariant spectrum introduced in  \cite{Amalie}, whose underlying spectrum is the $p$-typical topological cyclic homology $\TC(A;p)$. Its construction is analogous to the classical definition of $\TC(A;p)$ of \cite{BHM}, by taking the homotopy limit over certain maps 
\[R,F\colon \THR(A)^{C_{p^{n+1}}}\longrightarrow \THR(A)^{C_{p^{n}}}\]
 in the category of $\Z/2$-spectra, thus involving the equivariant structure of $\THR$ with respect to the finite dihedral subgroups $D_{p^n}$ of $O(2)$ and the Weyl actions of $\Z/2\cong D_{p^n}/C_{p^{n}}$ on $\THR(A)^{C_{p^{n}}}$. Alternatively, Quigley and Shah give in \cite{QS} a construction of $\TCR$ for bounded below spectra as an equaliser analogous to Nikolaus and Scholze's definition of $\TC$ of \cite{NS}. 
The goal of this paper is to describe the geometric fixed-points $\TCR(A;p)^{\phi \Z/2}$ in terms of derived smash products in the same spirit of the formula for $\THR(A)^{\phi\Z/2}$ above, and use this description to carry out calculations in some fundamental examples.

A ring spectrum with anti-involution $A$ is canonically a left and a right module over the Hill-Hopkins-Ravenel norm (see \cite{HHR}) of its underlying spectrum, by means of maps
\[
(N^{\Z/2}_eA)\otimes A\longrightarrow A \ \ \ \ \ \ \ \ \ \ \ \ \ \ \ \ \ \ \ A\otimes (N^{\Z/2}_eA)\longrightarrow A
\] 
described informally respectively by sending $a\otimes b\otimes x$ and $x\otimes a\otimes b$  to $axw(b)$ and $w(a)xb$. By taking $\Z/2$-geometric fixed-points these give the left and right Frobenius $A$-module structure on $A^{\phi\Z/2}$ mentioned above. By applying the monoidal functor $N^{C_2}_e$, we also obtain a left and a right $N^{C_2}_eA$-module structure on $N^{C_2}_e(A^{\phi\Z/2})$. Here we are making the point of distinguishing between the subgroups $\Z/2$ and $C_2$ of $O(2)$, generated respectively by a reflection and the rotation of order two. 

\begin{theoremA}\label{intro:geomTCR}
Let $A$ be a ring spectrum with anti-involution, and suppose that the underlying spectrum and the $\Z/2$-fixed-points of $A$ are bounded below. Then for every odd prime $p$ there is a natural equivalence of spectra
\[
\TCR(A;p)^{\phi\Z/2}\simeq \THR(A)^{\phi\Z/2}\simeq A^{\phi\Z/2}\otimes_AA^{\phi\Z/2}.
\]
For the prime $2$, there is a natural equivalence with the homotopy equaliser
\[
\TCR(A;2)^{\phi\Z/2}\simeq eq\big(
\xymatrix{
(A\otimes_{N^{C_2}_eA}N^{C_2}_e(A^{\phi\Z/2}))^{C_2}
\ar@<1ex>[r]^-{f}\ar[r]_-{r}
&
A^{\phi\Z/2}\otimes_AA^{\phi\Z/2}
}
\big),
\]
where $f$ forgets the fixed-points and  $r$ maps to the $C_2$-geometric fixed-points, followed by the respective identifications of $A\otimes_{A\otimes A}(A^{\phi\Z/2}\otimes A^{\phi\Z/2})$
and $(A\otimes_{N^{C_2}_eA}N^{C_2}_e(A^{\phi\Z/2}))^{\phi C_2}$ with $A^{\phi\Z/2}\otimes_AA^{\phi\Z/2}$.
\end{theoremA}

We prove this Theorem in \S\ref{sec:odd} for odd primes, and in \S\ref{sec:2} for the prime $2$.  Our proof proceeds by identifying the $\Z/2$-geometric fixed-points of $\THR(A)^{C_{p^n}}$  inductively over $n\geq 0$, together with the structure maps $R,F\colon \THR(A)^{C_{p^{n+1}}}\to \THR(A)^{C_{p^n}}$. The key ingredient is a result of \cite{LW09} which gives a certain pushout decomposition of the universal space of the family of reflections of $O(2)$.
We suspect that our theorem could also be proved starting from the description of $\TCR$ for bounded below spectra given in \cite{QS} using the same technique. 
Our proof of Theorem \ref{intro:geomTCR} is given more generally for bounded-below real $p$-cyclotomic spectra, see Theorem \ref{tcrformula1}.

We use the formula of Theorem \ref{intro:geomTCR} to compute the geometric fixed-points of $\TCR$ in some fundamental examples, starting with spherical group-rings. Every topological  monoid $M$ with anti-involution $w\colon M^{op}\to M$ has an underlying $\Z/2$-equivariant homotopy type. The genuine $\Z/2$-equivariant suspension spectrum $\mathbb{S}[M]:=\Sigma^{\infty}_+M$ of the latter gets canonically the structure of a ring-spectrum with anti-involution. The monoid $M$ acts on its fixed-points subspace $M^{\Z/2}$ by $m\cdot x=mxw(m)$ and $x\cdot m=w(m)xm$, and the corresponding $2$-sided bar construction admits a ``Frobenius endomorphism''
\[
\psi\colon B(M^{\Z/2},M,M^{\Z/2})\longrightarrow B(M^{\Z/2},M,M^{\Z/2})
\]
defined simplicially by $\psi(x,m_1,\dots,m_n,y)=(x,m_1,\dots,m_n,ym_n\dots m_1xw(m_1)\dots w(m_n)y)$.
It also has an involution which reverses the order of the factors of the bar construction and applies $w$ to each component.
 The following is analogous to the classical description of $\TC$ of spherical group-rings of \cite{BHM} and \cite[Theorem IV.3.6]{NS}.
\begin{theoremA}\label{intro:gprings}
Let $M$ be a well-pointed topological monoid with anti-involution. Then there is a pullback square
\[
\xymatrix@C=50pt{
\TCR(\mathbb{S}[M];2)^{\phi\Z/2}\ar[r]\ar[d]&\Sigma^{\infty}_+B(M^{\Z/2},M,M^{\Z/2})_{hC_2}\ar[d]^-{\trf}
\\
\Sigma^{\infty}_+B(M^{\Z/2},M,M^{\Z/2})\ar[r]^-{\id-\Sigma^{\infty}_+\psi}&\Sigma^{\infty}_+B(M^{\Z/2},M,M^{\Z/2})
}
\]
where the right vertical map is the transfer.
In particular for $M=\ast$ there is an equivalence 
\[ {\TCR(\mathbb S;2)}^{\phi \Z/2}\simeq {\mathbb S} \oplus {\mathbb R}P^{\infty}_{-1},\]
where ${\mathbb R}P^{\infty}_{-1}$ is the homotopy fibre of the transfer $\trf \colon \Sigma^{\infty}_{+} {\mathbb R}P^{\infty}=\mathbb{S}_{hC_2} \to \mathbb{S}$.
\end{theoremA}

Let us point out that the  pullback square of Theorem \ref{intro:gprings} does not require any $2$-completion.
In particular the calculation of $\TCR(\mathbb{S};2)^{\phi\Z/2}$ of Theorem \ref{intro:gprings} confirms the expected equivariant homotopy type of $\TCR(\mathbb{S};2)$, that appeared in unpublished work of H\o genhaven \cite{Amalie}.
We prove this theorem in section \S\ref{sec:assembly}, and we calculate this pullback in \S\ref{sec:discrete} in the case where $M$ is a discrete group with various assumptions on the involution and the $2$-torsion. In particular we determine it fully  for $M=\Z$ with the minus involution and with the trivial involution, and for $M=C_2$. 
In Corollary \ref{splittingTCRSG} for every pointed $\Z/2$-space $X$, we consider the special case of the equivariant loop space $M=\Omega^\sigma X=\map_\ast(S^\sigma,X)$, where $S^\sigma$ is the sign representation sphere, and $\Z/2$ acts on the loop space by conjugation.  We use Theorem \ref{intro:gprings} to describe the cofibre of an assembly map
\[
 \Sigma^{\infty}_+(X^{\Z/2})\otimes (\mathbb{S}\oplus {\mathbb R}P^{\infty}_{-1})\longrightarrow\TCR(\mathbb{S}[\Omega^\sigma X];2)^{\phi\Z/2}
\]
in terms of the cofibre of the diagonal $\Delta\colon X^{\Z/2}\to X^{\Z/2}\times_XX^{\Z/2}$, where  the homotopy pullback is along the fixed-points inclusions. In particular
if the involution on $X$ is trivial these cofibres vanish and we obtain a splitting
\[
\TCR(\mathbb{S}[\Omega^\sigma X];2)^{\phi\Z/2}\simeq \Sigma^{\infty}_+X\otimes (\mathbb{S}\oplus {\mathbb R}P^{\infty}_{-1}).
\] 
This calculation shows that $\TCR(\mathbb{S}[\Omega^\sigma X];2)^{\phi\Z/2}$ is equivalent to Weiss and Williams' hyperquadratic L-theory of the pointed space $X$, which satisfies the same decomposition by \cite[Theorem 4.3, Corollary 4.4]{WW3}. 

There is in fact a deeper relationship between TCR and L-theory, especially in view of the following result, which we explain in more details at the end of the introduction.
Given a discrete commutative ring $A$, and we write $\TCR(A;2)$ for the $\TCR$ spectrum of the $\Z/2$-equivariant Eilenberg-MacLane commutative ring spectrum of $A$ equipped with the trivial involution.

\begin{theoremA}\label{intro:calcgeom}
Let $k$ a perfect field of characteristic $2$, and $\Z$ the ring of integers.
There are equivalences of spectra
\begin{align*}
\TCR(k;2)^{\phi\Z/2}&\simeq \bigoplus_{n\geq 0}(\Sigma^{2n-1}Hk/ \langle x+x^2 \vert \ x \in k \rangle \oplus\Sigma^{2n}H\F_2)
\\
\TCR(\Z;2)^{{\phi \Z/2}}&\simeq \bigoplus_{n \geq 0} \big(\Sigma^{4n-1} H \F_2\oplus \Sigma^{4n} H \Z/8 \oplus \Sigma^{4n+1} H \F_2 \big).
\end{align*}
\end{theoremA}

In the case of perfect fields, we are in fact able to calculate the full $\Z/2$-equivariant homotopy type of $\TCR(k;2)$: In \S\ref{sec:TCRk2} we use the description of $\TRR(k;2)^{\phi\Z/2}$ from Theorem \ref{inductivePB} to show that $\TRR(k;2)$ is the Eilenberg-MacLane spectrum of the constant Mackey functor on the ring of $2$-typical  Witt vectors $W(k;2)$, where $F$ corresponds to its Frobenius. In particular in Theorem \ref{pi0TRRp2} we show that  $\pi_0 \THR(k;2)^{D_{2^{n}}}$ is isomorphic to the $(n+1)$-truncated $2$-typical Witt vectors of $k$  (this is true for all commutative rings at odd primes by \cite[Theorem C]{Polynomial}, but it fails in general at the prime $2$, see Remark \ref{WittZ}). We are then able to conclude the following:

\begin{theoremA}\label{intro:TCRk}
For every perfect field $k$ of characteristic $2$, there is an equivalence of $\Z/2$-equivariant spectra
\[ \TCR(k;2) \simeq H \underline{\Z_2} \oplus \Sigma^{-1}H \underline{\coker(1-F)},\]
where $F \colon W(k;2) \to W(k;2)$ is the $2$-typical Witt vector Frobenius and the underline denotes the constant Mackey functor. 
\end{theoremA}

A similar decomposition holds for odd primes by a much easier argument, see Proposition \ref{oddcomputation}.
Finally, we prove a flat base-change result for $\TCR^{\phi\Z/2}$, showing that if $f\colon A\to B$ is a flat map of discrete commutative rings such that the geometric fixed-points of $B$ are base-changed along $f$ from those of $A$, then  $\TCR(B;2)^{\phi\Z/2}$ is ``almost'' base-changed from $\TCR(A;2)^{\phi\Z/2}$, up to some care with the different module structures on $HA^{\phi\Z/2}$ (see Corollary \ref{TCRbase-change} for the precise statement). This allows us to extend the calculations above as follows.

\begin{theoremA}\label{intro:perfectrings}
For every  perfect $\F_2$-algebra $A$, there is an equivalence of spectra
\[
\TCR(A;2)^{\phi\Z/2}\simeq \bigoplus_{n\geq 0}\big(\Sigma^{2n-1}H(\coker(\id+(-)^2))\big)\oplus\big(\Sigma^{2n}H(\ker(\id+(-)^2))\big)
\]
where $(-)^2\colon A\to A$ is the Frobenius of $A$. For every ring $B$ with no $2$-torsion and perfect modulo $2$ reduction,  $\TCR(B;2)^{{\phi \Z/2}}$ is a wedge of Eilenberg-MacLane spectra with homotopy groups
\[\pi_n\TCR(B;2)^{{\phi \Z/2}}\cong\left\{
\begin{array}{ll}
B/\langle x+x^2 \vert\ x \in B \rangle & n=4l-1
\\
\ker\big(\pr+\pr^2\colon B/\langle 4(x+x^2) \vert \ x \in B \rangle \to B/2 \big) & n=4l
\\
\ker\big(\id+(-)^2\colon B/2\to B/2\big) & n=4l+1
\\
0& n=4l+2
\end{array}
\right.
\]
for all $l\geq 0$, and zero for $n\leq -2$.
\end{theoremA}

\paragraph*{Real TC and L-theory}\ \vspace{.25cm} \\
The relationship between TC and L-theory was  originally observed by Weiss and Williams and studied by Weiss and Rognes. They were investigating whether, under certain conditions on a ring spectrum with anti-involution $A$, the quadratic L-theory $\Lt^q(A)$ is equivalent to the $\Z/2$-Tate construction of the fibre of the trace map $\K(A)\to \TC(A)$ after $2$-completion. Nikolaus then formulated an uncompleted version of this statement, conjecturing that $\TCR(A;2)^{\phi\Z/2}$ should be equivalent to the genuine normal L-theory of $A$,  defined as the cofibre
\[
\Lt^{n}(A):=\cof(\Lt^q(A)\longrightarrow\Lt(\Mod^\omega_A,\text{\Qoppa}_A))
\]
of the canonical symmetrisation map. Here $\Mod^\omega_A$ is the $\infty$-category of compact $A$-modules, and  $\text{\Qoppa}_A\colon (\Mod^\omega_A)^{op}\to \Sp$ is a certain Poincar\'e structure in the sense of Lurie's formalism of L-theory, which
is defined using the Frobenius module structure of $A^{\phi\Z/2}$ (see \cite[3.2.6 and 3.2.10]{9I} for the details). A proof of this conjecture will appear in work of Harpaz, Nikolaus, and Shah \cite{HNS}.

By construction, $\Lt(\Mod^\omega_A,\text{\Qoppa}_A)$ is the symmetric L-theory spectrum $\Lt^s(A)$ if $A$ is Borel-complete, that is if the canonical map $A^{\phi\Z/2}\to A^{t\Z/2}$ is an equivalence. One can then see that Nikolaus' conjecture implies the original conjecture of Weiss and Williams provided the fibre of the trace map becomes Borel-complete after 2-completion.

In the case of spherical group-rings, $\Lt^{n}(\mathbb{S}[\Omega^\sigma X])$ is the hyperquadratic L-theory of \cite{WW3} by \cite[Corollary 4.6.1]{9II}, and as mentioned above it is equivalent to $\TCR(\mathbb{S}[\Omega^\sigma X];2)^{\phi\Z/2}$ by \cite[Theorem 4.3, Corollary 4.4]{WW3} and Corollary \ref{splittingTCRSG}. 
The normal L-theory  $\Lt^{n}(k)$ is also well-understood if $k$ is a perfect field of characteristic $2$, for example by work of Kato and Ranicki, and its homotopy groups agree with the ones of the geometric fixed-points of $\TCR$ of Theorem \ref{intro:calcgeom}  (see Remark \ref{rem:geomTCRk} for more details on the description of these L-groups). Finally, $\Lt^{n}(\Z)$ is calculated by Taylor and Williams in \cite{TW} (see also \cite[Corollary 3.9 and 6.2]{HLN}) and agrees with our calculation of $\TCR(\Z;2)^{\phi\Z/2}$ of Theorem \ref{intro:calcgeom}. We are not aware of a flat base-change type of result analogous to Corollary \ref{TCRbase-change} for these normal L-spectra, nor if they have been computed for all the rings of Theorem \ref{intro:perfectrings}.

\subsection*{Acknowledgements}
The idea of exploiting the description of the universal space of reflections as a pushout, which ultimately led us to Theorem \ref{intro:geomTCR}, was given to us by Amalie H\o genhaven at the Hausdorff Research Institute for Mathematics in Bonn in 2017. We wholeheartedly thank her for this contribution and the Hausdorff Institute for the hospitality. We thank Lars Hesselholt and Ib Madsen for continued support, and Yonatan Harpaz, Markus Land, Thomas Nikolaus, John Rognes and Michael Weiss for helpful discussions regarding L-theory and the aforementioned conjectures.

Dotto and Patchkoria were partially supported by the German Research Foundation Schwerpunktprogramm 1786 and the Hausdorff Centre for Mathematics at the University of Bonn. Moi was supported by the K\&A Wallenberg  Foundation.

\tableofcontents

\section{Preliminaries}

\subsection{Equivariant spectra}
Let $G$ be a compact Lie group. In this paper we will be interested in the case where $G$ is the orthogonal group $O(2)$, or one of its subgroups. We write $\Sp^{G}$ for the stable model category of orthogonal spectra with an action of $G$, equipped with the flat model structure of \cite{Sto}. This is a model for the homotopy theory of genuine $G$-spectra. We recall that the weak equivalences are  the $\underline{\pi}_\ast$-isomorphism, where $\underline{\pi}_\ast$ is the equivariant homotopy groups Mackey-functor.

We denote by $\otimes$ the derived smash product of $G$-spectra and of modules in $G$-spectra, which can be obtained by applying the smash product to a flat replacement of the orthogonal $G$-spectra (that is to a cofibrant replacement in the flat model structure). We also denote by $\otimes$ the tensor of a pointed $G$-space $Z$ and a $G$-spectrum $X$:
\[
Z\otimes X:= (\Sigma^{\infty}Z)\otimes X.
\]
For every closed subgroup $H\leq G$, we denote the genuine fixed-points (which is the strict fixed-points of a fibrant replacement) and the geometric fixed-points functors respectively by 
\[(-)^{H},(-)^{\phi H}\colon \Sp^{G} \to \Sp^{W_G(H)},\]
where $W_G(H)=N_G(H)/H$ is the Weyl group of $H$ in $G$. 
We recall that the geometric fixed-points functor can be defined from the genuine fixed-points functor as
\[
X^{\phi H}=(\widetilde{E(\nsupseteq H)}\otimes X)^{H}
\]
where $(\nsupseteq H)$ is the family of subgroups of $N_G(H)$ which do not contain $H$, the $N_G(H)$-space $E(\nsupseteq H)$ is its universal space, and $\widetilde{E(\nsupseteq H)}$ is the pointed $N_G(H)$-space defined as the cofibre of the map $(E(\nsupseteq H))_+\to S^0$ which collapses $E(\nsupseteq H)$ to the non-basepoint of $S^0$. This induces a fibre sequence of $W_G(H)$-spectra
\[
((E(\nsupseteq H))_+\otimes X)^{H}\longrightarrow X^H\longrightarrow X^{\phi H}
\]
called the isotropy separation sequence.

We will be particularly interested in the case where $G=O(2)$ and $H=C_{p}$ is the cyclic subgroup of $O(2)$ of rotations  of order $p$, for some prime $p$. Then for any $O(2)$-spectrum $X$, we have a fibre sequence of $O(2)/C_p$-spectra
\begin{equation*}\label{isotropygeneral1}
((E(\nsupseteq C_p))_+\otimes X)^{C_p}\longrightarrow X^{C_p}\longrightarrow X^{\phi C_p},
\end{equation*}
and hence for any $n \geq 0$, there is a fibre sequence of $O(2)/C_{p^{n+1}}$-spectra
\begin{equation*}\label{isotropygeneral2}
((E(\nsupseteq C_p))_+\otimes X)^{C_{p^{n+1}}}\longrightarrow X^{C_{p^{n+1}}}\longrightarrow (X^{\phi C_p})^{C_{p^{n+1}}/C_p}.
\end{equation*}
By choosing the reflection over the real coordinate axis, we can identify $O(2)$ with the semi-direct product $\Z/2 \ltimes S^1$. Here the nontrivial element $\tau$ of $\Z/2$ corresponds to the latter reflection. In particular, one has the dihedral subgroups $D_{p^n}=\Z/2 \ltimes C_{p^n} \leq \Z/2 \ltimes S^1 =G$.
If we restrict the family $(\nsupseteq C_p)$ to the dihedral group $D_{p^{n+1}}$ for $n \geq 0$, then it becomes the family $\mathcal{R}$ consisting of the trivial group and of those subgroups generated by the reflections in $D_{p^{n+1}}$. Hence by restricting to $D_{p^{n+1}}/C_p$, we get a fibre sequence of $D_{p^{n+1}}/C_p$-spectra 
\begin{equation*}\label{isotropygeneral3}
(E\mathcal{R}_+\otimes X)^{C_p}\longrightarrow X^{C_p}\longrightarrow X^{\phi C_p},
\end{equation*}
and by taking fixed-points a fibre sequence of $D_{p^{n+1}}/C_{p^{n+1}}=\Z/2$-spectra
\begin{equation}\label{isotropyR}
(E\mathcal{R}_+\otimes X)^{C_{p^{n+1}}}\longrightarrow X^{C_{p^{n+1}}}\longrightarrow (X^{\phi C_p})^{C_{p^{n+1}}/C_p}.
\end{equation}
We will abuse notation and always write $\mathcal{R}$ for the family of reflections in $D_{p^{n+1}}$, for different $p$ and $n$. Although these families are different, their classifying spaces $E\mathcal{R}$ are always modelled by the restriction to the appropriate dihedral group of the $O(2)$-space defined by the unit sphere $S(\mathbb{C}^{\infty})$.

In what follows we will always consider homotopy limits and homotopy colimits of spaces and spectra and will just refer to them as limits and colimits.
%

\subsection{Ring spectra with anti-involution and real topological Hochschild homology}\label{sec:subdivisions}

We give a short recollection of the construction and main properties of the dihedral structure on topological Hochschild homology, mainly from \cite{BHM} and \cite{THRmodels}.

We recall that a ring spectrum with anti-involution is an orthogonal ring spectrum $A$ equipped with a morphism of orthogonal ring spectra $w\colon A^{op}\to A$ such that $w^2=\id$. We endow $A$ with the genuine $\Z/2$-equivariant homotopy type defined by $w$. More precisely, a morphism of ring spectra $f\colon A\to B$ commuting with the involutions $w$ is an equivalence if it is a genuine $\Z/2$-equivariant equivalence in the category $\Sp^{\Z/2}$ of orthogonal $\Z/2$-spectra (see \cite[A1]{THRmodels} for a model structure on their category). 

The cyclic nerve of $A$ in the category of orthogonal spectra inherits a levelwise involution, which acts on $A \otimes A^{\otimes n}$ by applying in each factor $w$, fixes the first tensor factor, and reverses the order of the remaining $n$ factors. This involution, together with the levelwise $C_{n+1}$-actions which rotate the tensor factors, defines a dihedral spectrum in the sense of \cite[S 1.5, Example 5]{FLcrossed} and \cite{LodayDihedral} that we denote by $N^{di}A$. Its geometric realisation 
\[
\THR(A):=|N^{di}A|=|[n]\mapsto A^{\otimes n+1}|
\]
is then an orthogonal spectrum with $O(2)$-action (\cite[Theorem 5.3]{FLcrossed} and \cite[Proposition 3.10]{LodayDihedral}), which we regard as a genuine $O(2)$-equivariant spectrum.

In \cite{THRmodels} we studied the $\Z/2$-equivariant homotopy type of $\THR(A)$, where $\Z/2$ is the subgroup of $O(2)$ generated by the reflection over the $x$-axis. In particular we provided an equivalence of $\Z/2$-spectra
\[
\THR(A)\simeq B(A,N^{\Z/2}_eA,A)=A\otimes_{N^{\Z/2}_eA}A,
\]
(under the standing assumption that $A$ is flat)
where $\otimes_{N^{\Z/2}_eA}$ denotes the derived smash product in the category of modules over the Hill-Hopkins-Ravenel norm construction $N^{\Z/2}_eA$ of the underlying ring spectrum $A$ of \cite{HHR}. The norm acts on $A$ respectively on the left and on the right by
\[
N^{\Z/2}_eA\otimes A=A^{\otimes 3}\xrightarrow{A\otimes \tau}A^{\otimes 3}\xrightarrow{A^{\otimes 2}\otimes w}A^{\otimes 3}\xrightarrow{\mu}A
\]
\[
A\otimes N^{\Z/2}_eA=A^{\otimes 3}\xrightarrow{\tau\otimes A}A^{\otimes 3}\xrightarrow{w\otimes A^{\otimes 2}}A^{\otimes 3}\xrightarrow{\mu}A,
\]
where $\tau\colon A^{\otimes 2}\to A^{\otimes 2}$ is the symmetry isomorphism and $\mu$ is the multiplication map of $A$.
Here $B(A,N^{\Z/2}_eA,A)$ is the two-sided bar construction of these actions, which computes the derived smash product. We then deduced an equivalence of spectra
\[
\THR(A)^{\phi\Z/2}\simeq B(A^{\phi\Z/2},A,A^{\phi\Z/2})= A^{\phi\Z/2}\otimes_{A}A^{\phi\Z/2},
\]
where $A$ acts on $A^{\phi\Z/2}$ via the geometric fixed-points of the actions of $N^{\Z/2}_eA$ on $A$, using the diagonal equivalence $(N^{\Z/2}_eA)^{\phi\Z/2}\simeq A$. We refer to these actions as the Frobenius module structures of $A^{\phi\Z/2}$.

The present paper will focus on the equivariant homotopy type of $\THR(A)$ with respect to the finite dihedral subgroups of $O(2)$. We now give a recollection of materials on dihedral objects and simplicial subdivisions, which we use to model the equivariant homotopy type of $\THR(A)$ with respect to the finite dihedral subgroups simplicially. We recall that a dihedral orthogonal spectrum is a simplicial orthogonal spectrum $X_\bullet\colon \Delta^{op}\to \Sp$ whose $n$-simplicies $X_n$ are equipped with an action of the dihedral group $D_{n+1}=\Z/2\ltimes C_{n+1}$, which is suitably compatible with the simplicial structure \cite[Proposition 3.4]{FLcrossed}. The geometric realisation of $X_\bullet$ has an induced action of $O(2)=\Z/2\ltimes S^1$ by \cite[Theorem 5.3]{FLcrossed}. The action of the reflection generating $\Z/2$ on $|X_\bullet|$ is induced by the maps
\[
X_n\otimes\Delta^n_+\xrightarrow{w\otimes \omega_n} X_n\otimes\Delta^n_+
\]
where $w$ is the action of the generator of $\Z/2\leq D_{n+1}$ on the $n$-simplicies, and $\omega_n$ sends $(t_0,\dots,t_n)\in \Delta^n$ to $(t_n,\dots,t_0)$ \cite[Lemma 5.6(ii)]{FLcrossed}. The description of the cyclic action is more involved, and it requires simplicial subdivision.

Let $\sd_{r}\colon \Delta^{op}\to \Delta^{op}$ be the functor which sends the finite totally ordered set $[n]=\{0,1,\dots,n\}$ to the $r$-fold join $[n]\star [n]\star\dots\star [n]$, defined as the set $[r-1]\times [n]$ with the total order $(a,i)\leq (b,j)$ if either $a<b$ or if $a=b$ and $i\leq j$. Given a dihedral spectrum $X_\bullet\colon \Delta^{op}\to \Sp$, we let
\[
\sd_{r}X_\bullet:=X_\bullet\circ \sd_{r}
\]
be the $r$-fold subdivision of $X$. Let $g_n$ be the generator of $C_{n+1}$ and $w$ the generator of $\Z/2$. The action of $g_{rn+r-1}^{n+1}$ on the $n$-simplicies $(\sd_r X_\bullet)_n=X_{rn+r-1}$ defines a simplicial action of $C_r$ on $\sd_r X_\bullet$, and there is a $C_r$-equivariant isomorphism
\[
|\sd_rX_\bullet|\cong |X_\bullet|
\]
induced by the maps
\[
(\sd_rX)_n\otimes \Delta^n=X_{rn+r-1}\otimes \Delta^n_+\xrightarrow{\id\otimes \delta_r}X_{rn+r-1}\otimes \Delta^{rn+r-1}_+
\]
where $\delta_r$ sends $t\in \Delta^n$ to $(t,t,\dots,t)/r\in \Delta^{rn+r-1}$ \cite[\S 1]{BHM}. This isomorphism is moreover $\Z/2$-equivariant, where the action of $\Z/2$ on the left-hand side is defined from the maps $w\otimes \omega_n$ above as for $|X_\bullet|$. Let us finally make this $\Z/2$-action simplicial. 

Let $\sd_e\colon \Delta^{op}\to \Delta^{op}$ be the functor that sends $[n]$ to $[n]\star [n]^{op}$, where $[n]^{op}$ is the set $\{0,1,\dots,n\}$ with the canonical order reversed. Let $Y_\bullet\colon \Delta^{op}\to \Sp$ be a simplicial orthogonal spectrum with involutions $w_n$ on $Y_n$ for every $n\geq 0$, such that for every $\theta\colon [n]\to [m]$
\[
\theta^\ast\circ w_m=w_n\circ(\theta^{op})^\ast.
\]
For example, $Y_\bullet$ could be a dihedral object $X_\bullet$ where $w_n$ acts by the action of the generator of $\Z/2$ as a subgroup of $D_{n+1}$, or $Y_\bullet=\sd_r X_\bullet$ where $w_n$ acts as the generator of $\Z/2$ as a subgroup of $D_{r(n+1)}$. The geometric realisation of $Y_\bullet$ has a $\Z/2$-action defined as above from the maps $w_n\otimes\omega_n$.
We now let
\[
\sd_{e}Y_\bullet:=Y_\bullet\circ \sd_{e}
\]
be the corresponding subdivision. The action of $w_{2n+1}$ on the $n$-simplicies $(\sd_{e}Y_\bullet)_n=Y_{2n+1}$ defines a simplicial $\Z/2$-action on $\sd_{e}Y_\bullet$. There is an isomorphism
\[
|\sd_eY_\bullet|\cong |Y_\bullet|
\]
induced by the maps
\[
(\sd_eY)_n\otimes \Delta^n=Y_{2n+1}\otimes \Delta^n_+\xrightarrow{\id\otimes \delta_e}Y_{2n+1}\otimes \Delta^{rn+r-1}_+
\]
where $\delta_e$ sends $t\in \Delta^n$ to $(t,\omega_n(t))/2\in \Delta^{2n+1}$, and this isomorphism is clearly $\Z/2$-equivariant. Combining these subdivisions we obtain a $D_r$-equivariant isomorphism
\[
|\sd_e\sd_r X_\bullet|\cong |X_\bullet|
\]
for every dihedral orthogonal spectrum $X_\bullet$, and in particular an isomorphism 
\[
|\sd_e\sd_r N^{di}A|\cong \THR(A)
\]
of genuine orthogonal $D_r$-spectra for every $r\geq 1$. 

\begin{rem}
By the latest isomorphism it follows that $\THR$ sends equivalences of ring spectra with anti-involution to equivalences of genuine $D_r$-spectra, for every $r\geq 1$. Indeed, the $n$-simplicies of $\sd_e\sd_r N^{di}A$ are the orthogonal spectrum 
\[
(\sd_e\sd_r N^{di}A)_n=(N^{di}A)_{r(2n+1)+r-1}=A^{\otimes r(2n+2)}
\]
where $C_r$ acts cyclically on $r$ and the generator of $\Z/2$ acts as described above (and we recall that the tensor product indicates the smash product of a flat replacement). This indexed smash product sends an equivalence of orthogonal spectra to a genuine $D_r$-equivalence by \cite[Theorem 3.2.16]{BrDuSt} (see also \cite[Proposition B.209]{HHR}), and thus its realisation is also a genuine $D_r$-equivalence (since $\sd_e\sd_r N^{di}A$ is a good simplicial spectrum by the argument of \cite[Lemma 2.14]{THRmodels}). 
\end{rem}

The equivalence $\THR(A)^{\phi\Z/2}\simeq A^{\phi\Z/2}\otimes_{A}A^{\phi\Z/2}$ above is in fact induced by the $\Z/2$-equivariant isomorphism $\THR(A)\cong |\sd_eN^{di}A|$. We can now refine this equivalence to an equivariant equivalence with respect to the action of the Weyl group.
The normaliser of $\Z/2$ inside of $O(2)$ is $\Z/2\times C_2$, where $C_2$ is generated by the rotation of order $2$, and therefore the Weyl group of $\Z/2$ inside of $O(2)$ is isomorphic to $C_2$. In particular $\THR(A)^{\phi\Z/2}$ is a genuine $C_2$-spectrum, that we now describe in terms of derived smash products.
\begin{lemma}\label{C2THRphi}
There is an equivalence of $C_2$-spectra
\[
\THR(A)^{\phi\Z/2}\simeq B(A,N^{C_2}_eA,N^{C_2}_e(A^{\phi\Z/2}))=A\otimes_{(N^{C_2}_eA)}(N^{C_2}_e(A^{\phi\Z/2}))
\]
where $A$ is regarded as a $C_2$-spectrum via the isomorphism $\Z/2\cong C_2$, the norm $N^{C_2}_eA$ acts on $A$ by the right action defined above, and on $N^{C_2}_e(A^{\phi\Z/2})
$ by applying the monoidal functor $N^{C_2}_e$ to the left Frobenius action of $A$ on $A^{\phi\Z/2}
$.
\end{lemma}
\begin{proof}
The equivalence of genuine $D_4=C_2\times \Z/2$-equivariant spectra $\THR(A)\cong |\sd_e\sd_2 N^{di}A|$ defined above gives rise to an equivalence of $C_2$-spectra
\begin{align*}
\THR(A)^{\phi\Z/2}&\cong |\sd_e\sd_2 N^{di}A|^{\phi\Z/2}=|[n]\mapsto A^{\phi\Z/2}\otimes A^{2n+1}\otimes A^{\phi\Z/2}|
\\&\cong |[n]\mapsto A\otimes (N^{C_2}_eA)^{\otimes n}\otimes  N^{C_2}_e(A^{\phi\Z/2})|=B(A,N^{C_2}_eA,N^{C_2}_e(A^{\phi\Z/2}))
\\&=A\otimes_{(N^{C_2}_eA)}(N^{C_2}_e(A^{\phi\Z/2}))
\end{align*}
where $C_2$ acts on the third term by reversing the order of the smash products, and the isomorphism rearranges the factors by pairing the factors which are swapped.
\end{proof}

\subsection{Real cyclotomic spectra and real topological cyclic homology} \label{realcylcprelim}

We now review the definitions of the main objects of study of the paper. These are completely analogous to the classical definitions surrounding topological cyclic homology of \cite{BHM}, and are carried out by carefully lifting all the constructions to the category of $\Z/2$-equivariant spectra. These constructions were laid out in \cite{Amalie} using B\"okstedt's model for real topological Hochschild homology, and we recast them here for the model of $\THR$ above. The two approaches are equivalent by the comparison results of \cite{THRmodels} and \cite{DMPSW}.

\begin{defn}
Let $p$ be a prime. A real $p$-cyclotomic spectrum is an $O(2)$-spectrum $T\in \Sp^{O(2)}$ equipped with a map of $O(2)$-spectra 
\[ T^{\phi C_p}\stackrel{\simeq}{\longrightarrow} T,\]
where $O(2)$ acts on the left-hand side by restriction along the root isomorphism $O(2)\to O(2)/C_p$, and which is a $D_{p^n}$-equivalence for all $n\geq 0$. 
\end{defn}

The prime example of a real $p$-cyclotomic spectrum (for all prime $p$) is the real topological Hochschild homology spectrum $\THR(A)$ of a ring-spectrum with anti-involution $A$. The cyclotomic structure maps are in fact isomorphisms, defined on the dihedral bar construction from the diagonal isomorphisms
\[
A\cong (A^{\otimes p})^{\phi C_p}
\]
(see e.g. \cite{ABGHLMcyclotomic}, or  \cite[\S 5]{DMPSW}, and we remind that this is an isomorphism since $A$ is assumed to be flat). In particular, they induce a $S^1/C_p$-equivalence and a $\Z/2$-equivalence on realisations, and thus an $O(2)$-equivalence (see \cite[\S 3.3]{Polynomial}).
For every real $p$-cyclotomic spectrum $T$, the isotropy separation sequence (\ref{isotropyR}) defines fibre sequences of $\Z/2$-spectra
\[
(E\mathcal{R}_+\otimes T)^{C_{p^{n+1}}}\longrightarrow T^{C_{p^{n+1}}}\longrightarrow (T^{\phi C_p})^{C_{p^{n+1}}/C_p}\simeq T^{C_{p^{n}}}
\]
for every $n\geq 0$, and the composite of the right-hand arrow and the equivalence is denoted by $R\colon T^{C_{p^{n+1}}}\to T^{C_{p^{n}}}$. Since the cyclotomic structure map is $O(2)$-equivariant, using appropriate root isomorphisms, we see that $R$ is $O(2)$-equivariant. 

\begin{defn} Let $T$ be a real $p$-cyclotomic spectrum. For every integer $n\geq 0$, we let $\TRR^{n+1}(T;p)$ be the $\Z/2$-spectrum
\[
\TRR^{n+1}(T;p):=T^{C_{p^n}},
\]
where $\Z/2$ is identified with the subgroup of $O(2)/C_{p^n}$ generated by the reflection of the $x$-axes, and
\[
\TRR(T;p):=\holim\big(
\dots\xrightarrow{R}   \TRR^{n+1}(T;p)\xrightarrow{R}  \TRR^{n}(T;p)\xrightarrow{R} \dots\xrightarrow{R} \TRR^{1}(T;p)=T
\big).
\]
If $A$ is a ring spectrum with anti-involution, we write
\[
\TRR^{n+1}(A;p):=\TRR^{n+1}(\THR(A);p)\ \ \ \ \  \ \ \  \mbox{and}\ \  \ \ \  \ \ \ \ \ \TRR(A;p):=\TRR(\THR(A);p).
\]
\end{defn}
The inclusion of subgroups $C_{p^{n-1}}\leq C_{p^{n}}$ defines a map $F\colon  \TRR^{n+1}(T;p)\to  \TRR^{n}(T;p)$, which is equivariant for the Weyl actions and thus in particular $\Z/2$-equivariant. It also commutes with the map $R$ since $R$ is $O(2)$-equivariant and therefore induces a map of $\Z/2$-spectra
\[
F\colon \TRR(T;p)\longrightarrow\TRR(T;p)
\]
by passing to the limit, whose underlying map is the Frobenius of \cite{BHM}.
\begin{defn} \label{TCRdefTRR}
Let $T$ be a real $p$-cyclotomic spectrum. The real topological cyclic homology of $T$ is the $\Z/2$-spectrum defined as the equaliser
\[
\TCR(T;p)=eq\big(\xymatrix{\TRR(T;p)\ar@<.5ex>[r]^-\id\ar@<-.5ex>[r]_-F&\TRR(T;p)}\big).
\]
If $A$ is a ring spectrum with anti-involution, we let the real topological cyclic homology of $A$ be the $\Z/2$-spectrum
\[
\TCR(A;p)=eq\big(\xymatrix{\TRR(A;p)\ar@<.1ex>[r]^-\id\ar@<-1ex>[r]_-F&\TRR(A;p)}\big).
\]
\end{defn}
As in the classical definition of $\TC$, since $R$ and $F$ commute one can alternatively define
\[
\TFR(T;p):=\holim\big(
\dots\xrightarrow{F}   \TRR^{n+1}(T;p)\xrightarrow{F}  \TRR^{n}(T;p)\xrightarrow{F} \dots\xrightarrow{F} \TRR^{1}(T;p)=T
\big)
\]
and
\[
\TCR(T;p)=eq\big(\xymatrix{\TFR(T;p)\ar@<.1ex>[r]^-\id\ar@<-1ex>[r]_-R&\TFR(T;p)}\big).
\]
The underlying spectrum of $\TCR(T;p)$ is by construction the topological cyclic homology spectrum $\TC(T;p)$ of \cite{BHM} (see also \cite[Theorem 1.3]{DMPSW}). The focus of the paper is to understand the $\Z/2$-equivariant homotopy type of $\TCR(T;p)$, and in particular its geometric fixed-points spectrum.

\section{The Geometric fixed-points of \texorpdfstring{$\TCR$}{TCR}}

The aim of this section is to give a simple formula for the $\Z/2$-geometric fixed-points of $\TCR(A;p)$, when $A$ and its $\Z/2$-fixed-points are bounded below. This object turns out to be interesting only for the prime $p=2$, but we will start with the easier case of odd primes.

\subsection{The odd primary case}\label{sec:odd}

In the odd primary case the geometric fixed-points of $\TRR$ admit a very simple description, as they split as a product.

\begin{theorem} \label{geometricsplitting} For any odd prime $p$, real $p$-cyclotomic spectrum $T$, and $n \geq 1$, there is a natural equivalence
\[\TRR^n(T;p)^{\phi \Z/2} \simeq  \bigoplus_{i=1}^n T^{\phi \Z/2},\]
under which the maps $F,R\colon \TRR^{n+1}(T;p)^{\phi \Z/2}\to \TRR^n(T;p)^{\phi \Z/2}$ respectively project off the first and the $(n+1)$-st summand.
\end{theorem}

Before diving into the proof, we observe that if $T=\THR(A)$ is the real topological Hochschild homology of a ring spectrum with anti-involution $A$, we have an explicit description of the geometric fixed-points
\[
\THR(A)^{\phi \Z/2}\simeq A^{\phi \Z/2}\otimes_AA^{\phi \Z/2}.
\]
In particular if $A$ is the Eilenberg-MacLane ring spectrum of a discrete ring with anti-involution and $\frac{1}{2} \in A$, then we have that $\THR(A)^{\phi \Z/2}=0$, and we obtain the following.

\begin{cor} If $A$ is a discrete ring with anti-involution and $\frac{1}{2} \in A$, then 
\[\TRR^n(A;p)^{\phi \Z/2}\simeq \TRR(A;p)^{\phi \Z/2}\simeq \TCR(A;p)^{\phi \Z/2}=0\]
for every odd prime $p$.
\end{cor}

The crucial combinatorial ingredient that makes the odd-primary case so simple compared to the prime $2$ is that for $p$ odd, any two reflections in $D_{p^n}$ are conjugate, and the Weyl group of a reflection is trivial. Applying \cite[Corollary 2.8]{LW09} to the trivial family $\{1\} \subset \mathcal{R}$, one gets a pushout square of $\Z/2$-spaces 
\begin{equation}\label{square}
\xymatrix{ D_{p^n} \times_{\Z/2} E\Z/2 \ar[r] \ar[d] & ED_{p^n} \ar[d] \\ D_{p^n} / \Z/2 \ar[r] & E\mathcal{R}.} 
\end{equation}
This pushout is the main ingredient for establishing the following result, of which Theorem \ref{geometricsplitting} is an immediate consequence.

\begin{prop} \label{cofibreTRodd}\label{htpypullbackodd}  Let $T$ be a real $p$-cyclotomic spectrum with $p$ odd.
Then for every $n\geq 1$, the square
\[\xymatrix{ \TRR^{n+1}(T;p)^{\phi \Z/2} \ar[d]^F \ar[r]^-R & \TRR^{n}(T;p)^{\phi \Z/2} \ar[d]^F \\  \TRR^{n}(T;p)^{\phi \Z/2} \ar[r]^-R & \TRR^{n-1}(T;p)^{\phi \Z/2}}\]
is a pullback whose horizontal fibres are equivalent to $T^{\phi\Z/2}$. Here we interpret $\TRR^{0}(T;p)=0$.
\end{prop}

\begin{proof}
Let us start by identifying the fibre of the map $R$.
For any group $G$ and $G$-space $E$, we write $\widetilde{E}$ for the cofibre of the based map $E_+\to S^0$ which takes $E$ to the non-basepoint of $S^0$. 
By definition, the geometric fixed-points of the map $R$ fits into the commutative square
\[
\xymatrix@C=15pt{
\TRR^{n+1}(T;p)^{\phi \Z/2} \ar[r]^-R\ar[d]_{\simeq} & \TRR^{n}(T;p)^{\phi \Z/2}\ar@{}[r]|-{\simeq}&((T^{\phi C_{p}})^{C_{p^{n-1}}})^{\phi \Z/2}\ar@{=}[d]
\\
((T\otimes \widetilde{ED}_{p^n})^{C_{p^n}})^{\phi\Z/2}\ar[rr]&&((T\otimes \widetilde{E\mathcal{R}})^{C_{p^n}})^{\phi\Z/2}\rlap{\ ,}
}
\]
where the left vertical map and the bottom horizontal map are induced by the canonical maps $S^0\to \widetilde{ED}_{p^n}\to \widetilde{E\mathcal{R}}$. The left vertical map is an equivalence since its fibre is
\[
((T\otimes {ED_{p^n}}_{+})^{C_{p^n}})^{\phi \Z/2}=((T\otimes {ED_{p^n}}_{+})^{C_{p^n}}\otimes \widetilde {E\Z/2})^{ \Z/2}\simeq ((T\otimes (\varepsilon^* \widetilde {E\Z/2} \wedge {ED_{p^n}}_{+}))^{C_{p^n}})^{ \Z/2},
\]
where $\varepsilon\colon D_{p^n}\to D_{p^n}/C_{p^n}=\Z/2$ is the quotient map,
and $\varepsilon^* \widetilde {E\Z/2} \wedge {ED_{p^n}}_{+}$ is a contractible $D_{p^n}$-space.
 
By mapping the pushout square (\ref{square}) with additional disjoint base points to the pushout of $D_{p^n}\ltimes_{\Z/2} S^0=D_{p^n}\ltimes_{\Z/2} S^0\to S^0$ (where $D_{p^n}\ltimes_{\Z/2}-$ denotes the induction) and taking cofibres,  we get a pushout of pointed $D_{p^n}$-spaces
\[\xymatrix{ D_{p^n} \ltimes_{\Z/2} \widetilde{E\Z/2} \ar[r] \ar[d] & \widetilde{ED_{p^n}} \ar[d] \\ \ast  \ar[r] & \widetilde{E\mathcal{R}}.} \]
The fibre of $R$ is therefore given by the spectrum
\begin{align*}
((T\otimes (D_{p^n} \ltimes_{\Z/2} \widetilde{E\Z/2}))^{C_{p^n}})^{\phi\Z/2}
&\simeq
((T\otimes\varepsilon^* \widetilde {E\Z/2} )\otimes (D_{p^n} \ltimes_{\Z/2} \widetilde{E\Z/2}))^{D_{p^n}}
\\
&\simeq
(D_{p^n} \ltimes_{\Z/2} (T\otimes (\widetilde{E\Z/2} \wedge\widetilde{E\Z/2})))^{D_{p^n}}
\\
&\simeq 
(D_{p^n} \ltimes_{\Z/2} (T\otimes \widetilde{E\Z/2}))^{D_{p^n}}
\\
&\simeq 
(T\otimes  \widetilde{E\Z/2})^{\Z/2}
\\
&\simeq T^{\phi\Z/2}.
\end{align*}
By restricting the map $\widetilde{ED_{p^n}} \to \widetilde{E\mathcal{R}}$ to $D_{p^{n-1}}$, we recover the map $\widetilde{ED_{p^{n-1}}} \to \widetilde{E\mathcal{R}}$. Using this and that the Frobenius map $F\colon  \TRR^{n+1}(T;p)\to  \TRR^{n}(T;p)$ is induced by the subgroup inclusion $C_{p^{n-1}}\subset C_{p^n}$, under the equivalences above the map between the horizontal fibres identifies with the map
\[
(D_{p^n} \ltimes_{\Z/2} (T\otimes \widetilde{E\Z/2}))^{D_{p^n}}\longrightarrow (D_{p^n} \ltimes_{\Z/2} (T\otimes \widetilde{E\Z/2}))^{D_{p^{n-1}}}
\]
induced by the subgroup inclusion $D_{p^{n-1}}\subset D_{p^{n}}$. By applying the double coset formula on the source and target this map corresponds to the identity of $T^{\phi\Z/2}$, showing that the Frobenius on horizontal fibres is an equivalence.
\end{proof}

We want to conclude the section with a similar splitting for $\TRR(T;p)^{\phi \Z/2}$, by commuting geometric fixed-points with an infinite limit. This can be done by means of the following well-known result originally observed by Adams (see e.g. \cite[Section III.15.2]{Adamsbook}), and we sketch an argument for completeness. We say that a $G$-spectrum is bounded below if all of its fixed-points are bounded below, and in case $G$ is infinite we also require that there is a uniform bound over all the closed subgroups of $G$.

\begin{lemma} \label{inverse limits} Let  $\dots X_n \to X_{n-1} \to \dots X_2 \to X_1 \to X_0$
be a tower of uniformly bounded below $G$-spectra, where $G$ is finite. Then ${(\holim_n X_n)}^{\phi G} \simeq \holim_n {(X_n)}^{\phi G}$. 
\end{lemma}

\begin{proof} Since equivariant homotopy groups commute with infinite products it suffices to show that $\widetilde{E G} \otimes \holim_n X_n \simeq \holim_n (\widetilde{E G}  \otimes X_n)$. This reduces to showing that
\[\widetilde{E G} \otimes \prod_n X_n \simeq \prod_n (\widetilde{E G} \otimes X_n).\]
For a fixed homotopy group only the finite skeleta of $\widetilde{E G}$ matter since the $X_i$ are uniformly bounded below, and using that $\widetilde{E G}$ is of finite type, the statement reduces to showing that $\mathbb{S} \otimes -$ and $G_+ \otimes -$ commute with infinite products. The first is obvious and the second follows from the Wirthm\"uller isomorphism. 
\end{proof}

\begin{cor} \label{tcrodd}
Let $T$ be a bounded below real $p$-cyclotomic spectrum, with $p$ odd. Then there are natural equivalences 
\[
\TRR(T;p)^{\phi \Z/2} \simeq \prod_{i=1}^{\infty} T^{\phi \Z/2}
\quad\text{ and }\quad
\TCR(T;p)^{\phi\Z/2}\simeq T^{\phi\Z/2}.
\]
\end{cor}
\begin{proof}
The first equivalence follows from Lemma \ref{inverse limits} and Theorem \ref{geometricsplitting}, since the tower defining $\TRR(T;p)$ is uniformly bounded below since the spectra $T^{C_{p^n}}$ and $T^{D_{p^n}}$ are by assumption bounded below with a uniform bound over $n$. For the second equivalence we observe that by the same results the Frobenius of $\TRR(T;p)^{\phi \Z/2} $ is equivalent to the map
\[
F\colon \prod_{i=1}^{\infty} T^{\phi \Z/2}\longrightarrow \prod_{i=1}^{\infty} T^{\phi \Z/2}
\]
which projects off the first component. The equaliser of $F$ and the identity is thus $T^{\phi\Z/2}$ mapping diagonally to the infinite product.
\end{proof}


\subsection{The prime 2}\label{sec:2}

\subsubsection{The geometric fixed-points of \texorpdfstring{$\TRR$}{TRR}}

In this section we give a  formula for $\TRR^n(T;2)^{\phi \Z/2}$ for any bounded below real $2$-cyclotomic spectrum $T$. As mentioned earlier, the subgroups structure of the dihedral groups $D_{2^n}$ is more involved than the one of $D_{p^n}$ for odd $p$, and this makes our formula for $\TRR^n(T;2)^{\phi \Z/2}$ more interesting. The idea of the proof is again to compute $\TRR^n(T;2)^{\phi \Z/2}$ inductively, by finding a suitable replacement for the square of Proposition \ref{htpypullbackodd}.


Recall that we have chosen the reflection along the $x$-axis $\tau$ inside $D_{2^n}=\Z/2 \ltimes C_{2^n}$, where $\Z/2$ is the subgroup generated by $\tau$. 
If we denote by $\sigma_n$ the generator of $C_{2^n}$, then  $\sigma_n \tau$ is a reflection which is not conjugate to $\tau$. We denote the normaliser and Weyl group of $\Z/2$ inside $D_{2^n}/C_{2^{n-1}}$ respectively by
\[
N(\Z/2)= \Z/2\times C_2 \quad, \quad W(\Z/2)\cong C_2.
\]
The Weyl group $C_2$ is generated by the image of $\sigma_n$, which we denote again by $\sigma_n$. This group acts on the spectrum $\TRR^{n}(T;2)^{\phi\Z/2}$, for all $n\geq 1$ and any real $2$-cyclotomic spectrum $T$. 

In the case $n=1$ we are interested in two maps
\[r\colon (T^{\phi\Z/2})^{C_2}\longrightarrow  (T^{\phi\Z/2})^{\phi C_2}\simeq T^{\phi\Z/2}\ \ \ \ \ \ \ \ \ \ \mbox{and} \ \ \ \ \ \ \ \ \ f\colon (T^{\phi\Z/2})^{C_2}\longrightarrow  T^{\phi\Z/2}\]
analogous to the restriction and the Frobenius.
The map $r$ is induced by the canonical map from the fixed-points to the geometric fixed-points, followed by the equivalence given by the cyclotomic structure of $T$, and $f$ is the canonical map induced by the subgroup inclusion $1\subset C_2$.

\begin{example}\label{minicycloTHR}
Suppose that $T=\THR(A)$ is the real topological Hochschild homology spectrum of a ring spectrum with anti-involution $A$. Under the equivalence 
\[
\THR(A)^{\phi\Z/2}\simeq A\otimes_{(N^{C_2}_eA)}(N^{C_2}_e(A^{\phi\Z/2}))
\]
of Lemma \ref{C2THRphi}, the identification $(T^{\phi\Z/2})^{\phi C_2}\simeq T^{\phi\Z/2}$ coming from the cyclotomic structure corresponds to the equivalence
\[
(A\otimes_{(N^{C_2}_eA)}(N^{C_2}_e(A^{\phi\Z/2})))^{\phi C_2}\simeq A^{\phi\Z/2}\otimes_AA^{\phi\Z/2}\simeq A\otimes_{(A\otimes A)}(A^{\phi\Z/2}\otimes A^{\phi\Z/2})
\]
where the first equivalence is the monoidality of the geometric fixed-points combined with the diagonal isomorphism $(N^{C_2}_eA)^{\phi C_2}\cong A$, and the second is the general canonical equivalence $X\otimes_A Y=A\otimes_{A\otimes A}(X\otimes Y)$ for respectively right and left $A$-modules $X$ and $Y$, where $X$ is regarded as a left $A$-module via the anti-involution.
\end{example}

The maps $f$ and $r$ are related to $F$ and $R$ respectively, in the following manner. Let $c$ be the map
\[c \colon \TRR^{2}(T;2)^{\phi\Z/2}=(T^{C_2}\otimes \widetilde{E\Z/2})^{\Z/2}\simeq(T\otimes\varepsilon^{\ast}\widetilde{E\Z/2})^{C_2\times \Z/2} \to(T\otimes\widetilde{E(\nsupseteq \Z/2)})^{C_2\times \Z/2}= (T^{\phi\Z/2})^{C_2} \]
where $\varepsilon\colon C_2\times \Z/2\to \Z/2$ is the projection, and the arrow is induced by including families of subgroups, by noticing that $\varepsilon^{\ast}\widetilde{E\Z/2}=\widetilde{E\{1,C_2\}}$ as universal spaces of subgroups of $C_2\times \Z/2$. The naturality of the canonical map from fixed-points to geometric fixed-points gives canonical homotopies
 $r \circ c \simeq R$ and $f \circ c \simeq F$. In particular, for every $n\geq 1$ the iterated Frobenius map factors as
\[
F^n\colon \TRR^{n+1}(T;2)^{\phi\Z/2}\stackrel{F^{n-1}}{\longrightarrow} \TRR^{2}(T;2)^{\phi\Z/2} \stackrel{c}{\longrightarrow}  (T^{\phi\Z/2})^{C_2}\stackrel{f}{\longrightarrow}  T^{\phi\Z/2}.
\]
Finally, we recall from Section \ref{realcylcprelim} that the maps $R$ and $F$ commute since $R$ is $O(2)$-equivariant. 

\begin{theorem}\label{inductivePB}
For every real $2$-cyclotomic spectrum $T$ and $n\geq 1$, the square
\[
\xymatrix@C=60pt{
\TRR^{n+1}(T;2)^{\phi \Z/2}\ar[r]^-R\ar[d]_{c F^{n-1}\times c F^{n-1}\sigma_{n+1}}
&
\TRR^{n}(T;2)^{\phi\Z/2}\ar[d]^{(F^{n-1}, \sigma_1F^{n-1}\sigma_{n})}
\\
(T^{\phi \Z/2})^{C_2}\times (T^{\phi \Z/2})^{C_2}\ar[r]^-{r\times \sigma_{1}r}&T^{\phi \Z/2}\times T^{\phi \Z/2}
}
\]
is a pullback of spectra, where the square commutes by the homotopies $R F=F R$, $\sigma_nR = R\sigma_{n+1}$, and  $r \circ c \simeq R$. The Weyl action of $\sigma_{n+1}$ on $\TRR^{n+1}(T;2)^{\phi \Z/2}$ is given inductively by the strictly commutative diagram
\[
\xymatrix{
\TRR^{n}(T;2)^{\phi\Z/2}\ar[d]^{\sigma_n}\ar[r]
&
T^{\phi \Z/2}\times T^{\phi \Z/2}\ar[d]^{(\sigma_1\times\sigma_1)\tau}
&
(T^{\phi \Z/2})^{C_2}\times (T^{\phi \Z/2})^{C_2}\ar[d]^{\tau}\ar[l]
\\
\TRR^{n}(T;2)^{\phi\Z/2}\ar[r]
&
T^{\phi \Z/2}\times T^{\phi \Z/2}
&
(T^{\phi \Z/2})^{C_2}\times (T^{\phi \Z/2})^{C_2}\ar[l]
}
\]
where $\tau$ is the flip action. The Frobenius $F\colon \TRR^{n+1}(T;2)^{\phi \Z/2}\to \TRR^{n}(T;2)^{\phi \Z/2}$ is induced inductively on pullbacks by
\[
\xymatrix{
\TRR^{n}(T;2)^{\phi\Z/2}\ar[d]^{F}\ar[r]&T^{\phi \Z/2}\times T^{\phi \Z/2}\ar[d]^{((\id\times\sigma_1)\Delta)\vee 0}&(T^{\phi \Z/2})^{C_2}\times (T^{\phi \Z/2})^{C_2}\ar[l]\ar[d]^{\Delta\vee 0}
\\
\TRR^{n-1}(T;2)^{\phi\Z/2}\ar[r]&T^{\phi \Z/2}\times T^{\phi \Z/2}&(T^{\phi \Z/2})^{C_2}\times (T^{\phi \Z/2})^{C_2}\ar[l]
}
\]
for all $n\geq 2$, where the left hand square commutes since $\sigma_nF = F$. For $n=1$ the Frobenius is the composite
\[
F\colon (T^{\phi})^{C_2}{}_r\!\times_{\sigma_1r} 
(T^{\phi})^{C_2}\xrightarrow{\proj_1}  (T^{\phi})^{C_2} \xrightarrow{f} T^{\phi}.
\]
\end{theorem}

\begin{rem} \label{geometric formula}\label{htpiesF}  By inductively unravelling the formula of \ref{inductivePB} we obtain an equivalence
\begin{align*}
&\TRR^{n+1}(T;2)^{\phi \Z/2}  \simeq
\\ &\underbrace{
(T^{\phi})^{C_2}
{}_r\!\times_f
(T^{\phi})^{C_2}
{}_r\!\times_f
\cdots 
{}_r\!\times_f(T^{\phi})^{C_2}{}
}_{n}
{}_r\!\times_{\sigma_1r}
\underbrace{
(T^{\phi})^{C_2}
{}_f\!\times_{\sigma_1r} 
\cdots
{}_f\!\times_{\sigma_1r}
(T^{\phi})^{C_2}
{}_f\!\times_{\sigma_1r} 
(T^{\phi})^{C_2}
}_{n},
\end{align*}
for all $n\geq 1$, where we wrote $T^{\phi}:=T^{\phi\Z/2}$ for short and all the products denote fibre products over $T^{\phi}$, subscripts indicating along which maps we are taking the pullbacks. We can then further unravel the structure maps.
The restriction map $R \colon \TRR^{n+1}(T;2)^{\phi \Z/2} \to \TRR^{n}(T;2)^{\phi \Z/2}$ corresponds to projecting away the outer two factors for $n\geq 2$, and for $n=1$ to the composite
\[
R\colon (T^{\phi})^{C_2}{}_r\!\times_{\sigma_1r} 
(T^{\phi})^{C_2}\xrightarrow{\proj_1}  (T^{\phi})^{C_2} \xrightarrow{r} T^{\phi}.
\]
The Weyl action $\sigma_{n+1} \colon \TRR^{n+1}(T;2)^{\phi \Z/2} \to \TRR^{n+1}(T;2)^{\phi \Z/2}$ is induced by the map which reverses the order of the product factors.

The Frobenius $F\colon \TRR^{n+1}(T;2)^{\phi \Z/2} \to \TRR^{n}(T;2)^{\phi \Z/2}$ is slightly more delicate to describe.  For $n=1$ we have already mentioned the description. For $n\geq 2$ it is induced by the map defined schematically on the product by
\[(x_{-n}, \dots, x_{-3}, x_{-2}, x_{-1},x_1, x_2,\dots,x_n) \mapsto (x_{-n}, \dots, x_{-3}, x_{-2}, x_{-2},x_{-3},\dots, x_{-n}),\]
interpreted as follows. In order to map into the pullback $\TRR^{n}(T;2)^{\phi \Z/2}$ we need to exhibit homotopies  $\gamma_{-i}\colon r(x_{-i})\sim f(x_{-i+1})$ and $\gamma_i\colon \sigma_1r(x_{-i})\sim f(x_{-i+1})$ for $i=n,\dots, 3$, as well as a homotopy $\gamma_{0}\colon r(x_{-2})\sim \sigma_1r(x_{-2})$. The identifications $\gamma_{-i}$ are already present in the pullback $\TRR^{n+1}(T;2)^{\phi \Z/2}$, and $\gamma_i$ is the composite
\[
\gamma_i\colon  \sigma_1r(x_{-i})\stackrel{\sigma_1\gamma_{-i}}{\sim} \sigma_1 f(x_{-i+1})\sim f(x_{-i+1})
\]
where the second is the canonical homotopy provided by the equivariance of $f$ with respect to the Weyl action. Similarly, $\gamma_0$ is given by
\[
\gamma_0\colon r(x_{-2})\stackrel{\gamma_{-2}}{\sim}f(x_{-1})\sim \sigma_1f(x_{-1})\stackrel{\sigma_1\overline{\gamma}_{-2}}{\sim} \sigma_1r(x_{-2})
\]
where $\gamma_{-2}$ is the identification present in  $\TRR^{n+1}(T;2)^{\phi \Z/2}$.
\end{rem}

The rest of the section will be devoted to the proof of Theorem \ref{inductivePB}.
Our proof relies on a pushout relating $ED_{2^n}$ and $E\mathcal{R}$ which we now deduce from \cite[Corollary 2.8]{LW09}.

Recall that $\tau$ and $\sigma_n$ are the respective  generators of $\Z/2$ and $C_{2^{n}}$ inside $D_{2^n} \cong \Z/2 \ltimes C_{2^n}$,
and that  $\sigma_n \tau$ is a reflection which is not conjugate to $\tau$. We let $H_n$ be the subgroup generated by $\sigma_n\tau$.
 We denote the respective normalisers and Weyl groups by
\[
N(\Z/2)= \Z/2 \times C_2 \quad, \quad W(\Z/2)\cong C_2\quad , \quad N(H_n)=\{ 1, \sigma_n \tau, \sigma_n^{2^{n-1}}, \sigma_n^{2^{n-1}+1} \tau \} \quad , \quad W(H_n)\cong C_2.
\]
We observe that $N(H_n)$ is abstractly isomorphic to $\Z/2 \times C_2$ but one has to be careful with this identification, since $\Z/2$ and $H_n$ represent different conjugacy classes. 
By \cite[Corollary 2.8]{LW09}, there is a pushout of $D_{2^n}$-spaces
\begin{equation}\label{po}
\xymatrix{ D_{2^n} \times_{N(\Z/2)} EN(\Z/2) \coprod D_{2^n} \times_{N(H_n)} EN(H_n) \ar[r] \ar[d] & ED_{2^n} \ar[d] \\ D_{2^n} \times_{N(\Z/2)} EW(\Z/2) \coprod D_{2^n} \times_{N(H_n)} EW(H_n) \ar[r] & E\mathcal{R}. }
\end{equation}
We observe that the classifying spaces that show up in \cite[Corollary 2.8]{LW09} at the lower left corner are indeed equivalent to $EW(H_n)$ and $EW(\Z/2)$. This pushout square leads to the following analogue of Proposition \ref{htpypullbackodd}.

\begin{prop} \label{Frobonfibres} For every $n\geq 2$, the Frobenius induces a commutative diagram
\[\xymatrix{
(T^{\phi \Z/2})_{hW(\Z/2)} \oplus (T^{\phi H_{n}})_{hW(H_{n})} 
\ar[d]_-{\left(\begin{smallmatrix}1&0\\ \sigma_{n}&0\end{smallmatrix}\right)}
\ar[r] & \TRR^{n+1}(T;2)^{\phi \Z/2} \ar[r]^-R \ar[d]_-F& \TRR^{n}(T;2)^{\phi \Z/2}\ar[d]^-F
\\
 (T^{\phi \Z/2})_{hW(\Z/2)} \oplus (T^{\phi H_{n-1}})_{hW(H_{n-1})} \ar[r] & \TRR^n(T;2)^{\phi \Z/2} \ar[r]^-R & \TRR^{n-1}(T;2)^{\phi \Z/2}
 } \]
where the rows are fibre sequences, and $\sigma_{n}\colon (T^{\phi \Z/2})_{hW(\Z/2)} \xrightarrow{\simeq} (T^{\phi H_{n-1}})_{hW(H_{n-1})}$ is induced by the generator $\sigma_{n} \in C_{2^{n}}$ which conjugates $\Z/2$ to $H_{n-1}$ in $D_{2^n}$. 
\end{prop}

\begin{rem}
We note that  the spectra $(T^{\phi \Z/2})_{hW(\Z/2)}$, $(T^{\phi H_n})_{hW(H_n)}$ and $(T^{\phi H_{n-1}})_{hW(H_{n-1})}$ are all equivalent. This is because  $H_{n-1}$ and $\Z/2$ become conjugated in $D_{2^n}$, and consequently $H_{n}$ and $H_{n-1}$ are conjugated in $D_{2^{n+1}}$. It is however important to point out that $H_{n-1}$ and $\Z/2$ are not conjugated in $D_{2^{n-1}}$, and this plays a role while identifying the map $F$ on the fibres of $R$. 
\end{rem}

\begin{proof}[Proof of \ref{Frobonfibres}] Let us start by calculating the fibres of the horizontal maps.
The pushout square (\ref{po}) induces a pushout square of pointed $D_{2^n}$-spaces
\begin{equation}\label{potilde}
\xymatrix{ D_{2^n} \ltimes_{N(\Z/2)} \widetilde{EN(\Z/2)} \vee D_{2^n} \ltimes_{N(H_n)} \widetilde{EN(H_n)} \ar[r] \ar[d] & \widetilde{ED_{2^n}} \ar[d] \\ D_{2^n} \ltimes_{N(\Z/2)} \widetilde{EW(\Z/2)} \vee D_{2^n} \ltimes_{N(H_n)} \widetilde{EW(H_n)} \ar[r] & \widetilde{E\mathcal{R}}.} 
\end{equation}
The map $R$ is, just as in the proof of Proposition \ref{htpypullbackodd}, given by the map
\[
R\colon  \TRR^{n+1}(T;2)^{\phi \Z/2}\simeq ((T\otimes  \widetilde{ED_{2^n}})^{C_{2^n}})^{\phi\Z/2}  \longrightarrow ((T\otimes  \widetilde{E\mathcal{R}})^{C_{2^n}})^{\phi\Z/2}\simeq \TRR^{n}(T;2)^{\phi \Z/2}
\]
induced by the right vertical map of square (\ref{potilde}) (this part of  \ref{htpypullbackodd} does not use that $p$ is odd).  Since square (\ref{potilde}) is a pushout, the cofibre of $R$ is equivalent to
\[  ((T \otimes D_{2^n} \ltimes_{N(\Z/2)} \Sigma E\class {\Z/2})^{C_{2^n}}){ ^{\phi \Z/2} } \oplus ((T \otimes D_{2^n} \ltimes_{N(H_n)} \Sigma E\class {H_n})^{C_{2^n}}){ ^{\phi \Z/2} },  \]
where $E\class { \Z/2}$ and $E\class {H_n}$ are 
 pointed $N(\Z/2)$ and $N(H_n)$-spaces which fixed-points $S^0$ only at $\Z/2$ and $H_n$, respectively, and have contractible fixed-points at all other subgroups. Let us now identify the first summand, the identification of the second one being  similar. By the projection formula and untwisting the action on $T$, we see that 
\begin{align*}(((T \otimes D_{2^n} \ltimes_{N(\Z/2)} \Sigma E\class{\Z/2}))^{C_{2^n}}){ ^{\phi \Z/2} }
&\simeq
 (D_{2^n} \ltimes_{N(\Z/2)} ( T \otimes  (\varepsilon^* \widetilde {E\Z/2} \wedge \Sigma E\class { \Z/2} )))^{D_{2^n}}
 \\
 &\simeq
 ( T \otimes  (\varepsilon^* \widetilde {E\Z/2}\wedge \Sigma E\class { \Z/2} ))^{N(\Z/2)}
 \end{align*}
 where $\varepsilon\colon N(\Z/2)\cong C_2\times \Z/2\to \Z/2$ is the projection. Now we observe that since $(\varepsilon^* \widetilde {E\Z/2})^{\Z/2}=S^0$ and $E\class { \Z/2}$ has isotropy $S^0$ concentrated at $\Z/2$, we have that $\varepsilon^* \widetilde {E\Z/2}\wedge E\class { \Z/2}\simeq E\class { \Z/2}$, and therefore that the first summand of the fibre of $R$ is equivalent to
 \[
 ( T \otimes E\class { \Z/2} )^{N(\Z/2)}.
 \]
The other summand is identified similarly. Now the $N(\Z/2)$-space $E \class{\Z/2}$ is equivalent to $\widetilde{E(\nsupseteq \Z/2)} \wedge\pi^\ast EW(\Z/2)_+$, where $(\nsupseteq \Z/2)$ is the family of subgroups of $N(\Z/2)$ not containing $\Z/2$ and $ \pi\colon N(\Z/2)\to W(\Z/2)$ is the projection, again by observing that this smash product also has isotropy $S^0$ at $\Z/2$ and contractible isotropy on all other subgroups (using Elmendorf's theorem \cite{Elmendorf}). By definition of the geometric fixed points with respect to a normal subgroup, we get that one summand of the fibre of $R$ is
\begin{align*} (T \otimes  E\class {\Z/2})^{N(\Z/2)} 
&\simeq (T \otimes  (\widetilde{E(\nsupseteq \Z/2)} \wedge \pi^\ast EW(\Z/2)_+))^{N(\Z/2)}
\\
& \simeq ((T \otimes  (\widetilde{E(\nsupseteq \Z/2)} \wedge \pi^\ast EW(\Z/2)_+))^{\Z/2})^{W(\Z/2)} 
\\
&\simeq  (T^{\phi \Z/2} \otimes EW(\Z/2)_+)^{W(\Z/2)} \simeq (T^{\phi \Z/2})_{hW(\Z/2)} \end{align*}
where the last equivalence uses the Adams isomorphism.

Let us now identify the map $F$ on the fibres of $R$. As we have just seen $R\colon  \TRR^{n+1}(T;2)^{\phi \Z/2}\to  \TRR^{n}(T;2)^{\phi \Z/2}$ is equivalent to the appropriate fixed-points of the map $T\otimes  \widetilde{ED_{2^n}}\to T\otimes  \widetilde{E\mathcal{R}}$. For simplicity let us denote the summands of its fibre by
\[
A_n\oplus B_n:=(T \otimes D_{2^n} \ltimes_{N(\Z/2)}  E\class {\Z/2})\oplus (T \otimes D_{2^n} \ltimes_{N(H_n)} E\class {H_n}).
\]
Since $ED_{2^n}$ and  $ED_{2^{n-1}}$ are equivalent as  $D_{2^{n-1}}$-spaces, the diagram of Proposition \ref{Frobonfibres} is then equivalent to the outer diagram of
\[\xymatrix{
(A^{C_{2^n}}_n)^{\phi\Z/2}\oplus (B^{C_{2^n}}_n)^{\phi\Z/2}\ar[r]\ar[d]^{F\oplus F}&((T\otimes  \widetilde{ED_{2^n}})^{C_{2^n}})^{\phi\Z/2}
\ar[r]^-R\ar[d]^F&
((T\otimes  \widetilde{E\mathcal{R}})^{C_{2^n}})^{\phi\Z/2}\ar[d]^F
\\
(A^{C_{2^{n-1}}}_n)^{\phi\Z/2}\oplus (B^{C_{2^{n-1}}}_n)^{\phi\Z/2}\ar[r]
&
((T\otimes  \widetilde{ED_{2^{n}}})^{C_{2^{n-1}}})^{\phi\Z/2} 
\ar[r]^-R&
((T\otimes  \widetilde{E\mathcal{R}})^{C_{2^{n-1}}})^{\phi\Z/2}
\\
(A^{C_{2^{n-1}}}_{n-1})^{\phi\Z/2}\oplus (B^{C_{2^{n-1}}}_{n-1})^{\phi\Z/2}\oplus 0\ar[r]\ar[u]^{\simeq}&
((T\otimes  \widetilde{ED_{2^{n-1}}})^{C_{2^{n-1}}})^{\phi\Z/2} \ar[u]^{\simeq}
\ar[r]^-R&
((T\otimes  \widetilde{E\mathcal{R}})^{C_{2^{n-1}}})^{\phi\Z/2}\ar[u]^{\simeq}
 } \]
where the maps $F$ are the restriction maps on fixed-points, and the vertical arrows are induced by the inclusion $D_{2^{n-1}}\subset D_{2^n}$. Let us analyse the bottom left vertical equivalence. The summand $(B^{C_{2^{n-1}}}_n)^{\phi\Z/2}$ is contractible, since there is a single double coset $D_{2^{n-1}}\backslash D_{2^n}/N(H_{n})$, and therefore
\[
(B^{C_{2^{n-1}}}_n)^{\phi\Z/2}\simeq  (D_{2^n} \ltimes_{N(H_n)} (T\otimes E\class {H_n}))^{D_{2^{n-1}}}\simeq (T\otimes E\class {H_n})^{C_2}=0,
\]
where we used that $N(H_n) \cap D_{2^{n-1}}=C_2$ and that $E\class{H_n}$ is trivial when restricted to $C_2$. The first equivalence follows from the projection formula as in the identification of the summands of the fibre of $R$ above. This shows that $F$ is trivial on the second summand of the fibres. Let us now apply the double coset formula to the first summand $(A^{C_{2^{n-1}}}_n)^{\phi\Z/2}$. This time there are two double cosets $D_{2^{n-1}}\backslash D_{2^n}/N(\Z/2)=\{1,\sigma_n\}$, where $\sigma_n$ is the generator of $C_{2^n}$, with $N(\Z/2)\cap D_{2^{n-1}}=N(\Z/2)$ and $\sigma_n$ conjugating $\Z/2$ and $H_{n-1}$. We therefore find that
\[
(A^{C_{2^{n-1}}}_n)^{\phi\Z/2}\simeq  (D_{2^n} \ltimes_{N(\Z/2)} (T\otimes E\class {\Z/2}))^{D_{2^{n-1}}}\simeq  (T \otimes  E\class{ \Z/2})^{N(\Z/2)} \oplus (T \otimes  E\class{H_{n-1}})^{N(H_{n-1})}
\]
and  $F\colon (A^{C_{2^{n}}}_n)^{\phi\Z/2}\to (A^{C_{2^{n-1}}}_n)^{\phi\Z/2}$ corresponds to $(1,\sigma_n)$.
\end{proof}

The previous proposition holds as stated only for $n\geq 2$. The correct analogue for $n=1$ is the following:

\begin{prop}\label{Frobonfibresn=1} The map $c$ induces $\id\oplus 0$ of horizontal fibres in the following diagram
\[\xymatrix{
(T^{\phi \Z/2})_{hW(\Z/2)} \oplus (T^{\phi H_{1}})_{hW(H_{1})} 
\ar[d]_-{\id\oplus 0}
\ar[r] & \TRR^{2}(T;2)^{\phi \Z/2} \ar[r]^-R \ar[d]_-c&T^{\phi \Z/2}\ar@{=}[d]
\\
 (T^{\phi \Z/2})_{hW(\Z/2)} \ar[r] &(T^{\phi \Z/2})^{C_2}\ar[r]^r &T^{\phi \Z/2}\rlap{\ .}
 } \]
\end{prop}

\begin{proof}
The top horizontal fibre sequence is from \ref{Frobonfibres}, and the bottom one is immediate from the isotropy separation sequence of the $C_2$-spectrum $T^{\phi\Z/2}$ and the definition of $r$.

In order to describe $F$ on the fibre we observe that there is a commutative diagram of pointed $D_{2^2}$-spaces
\[\hspace{-3cm}\ \xymatrix{  
\hspace{2.5cm} \varepsilon^* \widetilde {E\Z/2}\simeq \varepsilon^* \widetilde {E\Z/2} \wedge \widetilde{ED}_{2^2}\ar[r]\ar[d]
&
\varepsilon^* \widetilde {E\Z/2} \wedge \widetilde{ER}
\\
\widetilde{E(\nsupseteq \Z/2)}\ar[r]
 &
 \widetilde{E(\nsupseteq D_{2^2})}\ar[u]_-{\simeq}
  }
\]
where $\varepsilon\colon D_{2^2}=C_2\times\Z/2\to \Z/2$ is the projection, $(\nsupseteq \Z/2)$  and $(\nsupseteq D_{2^2})$ are the families of subgroups of $D_{2^2}$ which do note contain $\Z/2$ and $D_{2^2}$ respectively, and the arrows in the diagram are induced by the inclusions of families. By Elmendorf's theorem \cite{Elmendorf}, the induced map on horizontal cofibres is the projection
\[
\id\vee 0\colon \Sigma E \class{\Z/2} \vee\Sigma E \class{H_1}\longrightarrow\Sigma E \class{\Z/2},
\]
where we have used the identifications of Proposition \ref{Frobonfibres}.
By the calculations of Proposition \ref{Frobonfibres}, by smashing the square above with $T$ and taking $C_2\times\Z/2$-fixed-points, we obtain the diagram of the statement.
\end{proof} 

The identification of the Weyl action follows immediately from the proof of Proposition \ref{Frobonfibres}:

\begin{lemma} \label{FrobeniusWeyl} For any $n \geq 1$, the Weyl action on fibres is described by the diagram 
\[\xymatrix{ (T^{\phi \Z/2})_{hW(\Z/2)} \oplus (T^{\phi H_n})_{hW(H_n)} \ar[d]_-{\left(\begin{smallmatrix}0&\sigma_{n+1}^{-1}\\ \sigma_{n+1}&0\end{smallmatrix}\right)}  \ar[r]  & \TRR^{n+1}(T;2)^{\phi \Z/2} \ar[r]^R\ar[d]^{\sigma_{n+1}} 
&
\TRR^{n}(T;2)^{\phi \Z/2} \ar[d]^{\sigma_n}
\\
(T^{\phi \Z/2})_{hW(\Z/2)} \oplus (T^{\phi H_n})_{hW(H_n)} \ar[r] &  \TRR^{n+1}(T;2)^{\phi \Z/2}
\ar[r]^R
&
\TRR^{n}(T;2)^{\phi \Z/2}
}.\] 	
\end{lemma}

\begin{proof} As seen in the proof of Proposition \ref{Frobonfibres}, the fibre of $R$ consists of two summands
\[(A^{C_{2^n}}_{n+1})^{\phi\Z/2}\oplus (B^{C_{2^n}}_{n+1})^{\phi\Z/2},\]
where $A_{n+1}:=(T \otimes D_{2^{n+1}} \ltimes_{N(\Z/2)}  E\class {\Z/2})$ and $B_{n+1}:=(T \otimes D_{2^{n+1}} \ltimes_{N(H_{n+1})} E\class {H_{n+1}})$, and $(B^{C_{2^n}}_{n+1})^{\phi\Z/2}$ vanishes. The first summand $(A^{C_{2^n}}_{n+1})^{\phi\Z/2}$ decomposes into the wedge of two summands according to the double cosets $D_{2^n}\backslash D_{2^{n+1}}/N(\Z/2)=\{1,\sigma_{n+1}\}$, and the action of $\sigma_{n+1}$ permutes these two summands. \end{proof}

\begin{proof}[Proof of \ref{inductivePB}]
By iterating Proposition \ref{Frobonfibres} and Proposition \ref{Frobonfibresn=1} the map $F^{n-1}$ induces an equivalence between the first summands of the horizontal fibres of the diagram of Theorem \ref{inductivePB}. Similarly by \ref{FrobeniusWeyl} the maps $F^{n-1}\sigma_{n+1}$ and $F^{n-1}\sigma_{n}$ induce an equivalence between the second summands of the horizontal fibres.

Let us now identify the Frobenius $F\colon \TRR^{n+1}(T;2)^{\phi \Z/2}\to \TRR^{n}(T;2)^{\phi \Z/2}$ for $n\geq 1$. The identification for $n=1$ follows from the pullback description and the canonical homotopy $c \circ f \simeq F$.
For $n \geq 2$, we consider the following diagram whose front and back faces are pullbacks:
\[
\xymatrix@C=12pt{
\TRR^{n+1}(T;2)^{\phi \Z/2}\ar[dr]^-(.6){F}\ar[rr]^-R
\ar[dd]^-{(cF^{n-1},cF^{n-1}\sigma_{n+1})}
&&\TRR^{n}(T;2)^{\phi \Z/2}\ar[dr]^-F\ar[dd]^-(.75){(F^{n-1},\sigma_1F^{n-1}\sigma_{n})}
\\
&
\TRR^{n}(T;2)^{\phi \Z/2}\ar[rr]^-(.3){R}\ar[dd]_-(.25){(cF^{n-2},cF^{n-2}\sigma_{n})}
&&\TRR^{n-1}(T;2)^{\phi \Z/2}\ar[dd]^-{(F^{n-2},\sigma_1F^{n-2}\sigma_{n-1})}
\\
 (T^{\phi \Z/2})^{C_2}\times  (T^{\phi \Z/2})^{C_2}\ar[dr]_-{\Delta\vee 0} \ar[rr]^-(.75){r\times \sigma_1r}&&T^{\phi \Z/2}\times T^{\phi \Z/2}\ar[dr]^-{((\id\times\sigma_1)\Delta)\vee 0}
\\
&
 (T^{\phi \Z/2})^{C_2}\times  (T^{\phi \Z/2})^{C_2}\ar[rr]_-{r\times \sigma_1r}
 &&T^{\phi \Z/2}\times T^{\phi \Z/2}
}
\]
This diagram is a homotopy commutative cube, since its front and back faces and its arrows are equivalent to those of the outer part of the strictly commutative diagram
\[
\xymatrix@C=12pt{
(T^{C_{2^n}})^{\phi \Z/2}\ar[dr]^-(.6){F}\ar[rr]^-\phi\ar[dd]_-(.5){(F^{n-1},F^{n-1}\sigma_{n+1})}
&&
((T\otimes\widetilde{E\mathcal{R}})^{C_{2^n}})^{\phi \Z/2}\ar[dr]^-F\ar[dd]^-(.75){(F^{n-1},\sigma_2F^{n-1}\sigma_{n+1})}
\\
&
(T^{C_{2^{n-1}}})^{\phi \Z/2}\ar[rr]^-(.3){\phi}\ar[dd]_-(.25){(F^{n-2},F^{n-2}\sigma_{n})}
&&
((T\otimes\widetilde{E\mathcal{R}})^{C_{2^{n-1}}})^{\phi \Z/2}\ar[dd]^-{(F^{n-2},\sigma_1F^{n-2}\sigma_{n})}
\\
((T^{C_2})^{\phi \Z/2})^{\times 2}\ar[dr]_-{\Delta\vee 0} \ar[rr]^-(.75){\phi\times \sigma_2\phi}\ar[dd]^{c\times c}
&&
(((T\otimes\widetilde{E\mathcal{R}})^{C_2})^{\phi \Z/2})^{\times 2}\ar[dr]^-{((\id\times\sigma_2)\Delta)\vee 0} \ar[dd]^-(.75){c\times c}
\\
&
((T^{C_2})^{\phi \Z/2})^{\times 2}\ar[rr]^-(.25){\phi\times \sigma_2\phi}\ar[dd]_-(.25){c\times c}
 &&
(((T\otimes\widetilde{E\mathcal{R}})^{C_2})^{\phi \Z/2})^{\times 2} \ar[dd]^{c\times c}
\\
((T^{\phi \Z/2})^{C_2})^{\times 2}\ar[dr]_-{\Delta\vee 0} \ar[rr]^-(.75){\phi\times \sigma_2\phi}
&&
(((T\otimes\widetilde{E\mathcal{R}})^{\phi \Z/2})^{C_2})^{\times 2}\ar[dr]^-{((\id\times\sigma_2)\Delta)\vee 0} 
\\
&
((T^{\phi \Z/2})^{C_2})^{\times 2}\ar[rr]^-{\phi\times \sigma_2\phi}
 &&
(((T\otimes\widetilde{E\mathcal{R}})^{\phi \Z/2})^{C_2})^{\times 2}
}
\]
where the maps $\phi$ are induced by the canonical map $S^0\to\widetilde{E\mathcal{R}}$.
This identifies the Frobenius map. The Weyl action can be identified with a similar argument.
\end{proof}

\subsubsection{The geometric fixed-points of \texorpdfstring{$\TFR$}{TFR} and \texorpdfstring{$\TCR$}{TCR}}\label{sec:TFR}

In this section, we use Theorem \ref{inductivePB} to identify the $\Z/2$-geometric fixed-points of $\TCR(T;2)$ for every bounded-below real $2$-cyclotomic spectrum $T$.

It turns out that it is simpler to describe the endomorphism $R$ on the limit
\[
\TFR(T;p):=\holim\big(
\dots\xrightarrow{F}   \TRR^{n+1}(T;p)\xrightarrow{F}  \TRR^{n}(T;p)\xrightarrow{F} \dots\xrightarrow{F} \TRR^{1}(T;p)=T
\big).
\]
taken over the Frobenius, rather than describing the Frobenius on $\TRR(T;2)$. For simplicity, we will again write $T^{\phi}$ for $T^{\phi\Z/2}$.

\begin{theorem} \label{TFformula} Let $T$ be a bounded below real $2$-cyclotomic spectrum. Then $\TFR(T;2)^{\phi \Z/2}$ is equivalent to the homotopy inverse limit
\[\TFR(T;2)^{\phi \Z/2}\simeq\holim_n (\underbrace{
(T^{\phi})^{C_2}
{}_r\!\times_f
(T^{\phi})^{C_2}
{}_r\!\times_f
\cdots 
{}_r\!\times_f(T^{\phi})^{C_2}{}
}_{n} )\]
along the maps $\proj_{l} \colon \underbrace{
(T^{\phi})^{C_2}
{}_r\!\times_f
(T^{\phi})^{C_2}
{}_r\!\times_f
\cdots 
{}_r\!\times_f(T^{\phi})^{C_2}{}
}_{n+1} \to \underbrace{
(T^{\phi})^{C_2}
{}_r\!\times_f
(T^{\phi})^{C_2}
{}_r\!\times_f
\cdots 
{}_r\!\times_f(T^{\phi})^{C_2}{}
}_{n}$ which project away the last factor, i.e. given by  $(x_1, x_2, \dots, x_{n+1}) \mapsto (x_1, \dots, x_n)$.
Under this identification, the endomorphism $R \colon \TFR(T;2)^{\phi \Z/2} \to \TFR(T;2)^{\phi \Z/2}$ corresponds to the map induced on limits by the projection
\[ \proj_{f} \colon 
\underbrace{
(T^{\phi})^{C_2}
{}_r\!\times_f
(T^{\phi})^{C_2}
{}_r\!\times_f
\cdots 
{}_r\!\times_f(T^{\phi})^{C_2}{}
}_{n+1}
 \to
 \underbrace{
(T^{\phi})^{C_2}
{}_r\!\times_f
(T^{\phi})^{C_2}
{}_r\!\times_f 
\cdots 
{}_r\!\times_f(T^{\phi})^{C_2}{}
}_{n}
\]
off the first factor $(x_1, x_2, \dots, x_{n+1}) \mapsto (x_2, \dots, x_{n+1})$.  
\end{theorem}

\begin{proof} 
Let us first observe that since $T$ is bounded-below, by Lemma \ref{inverse limits}, $\TFR(T;2)^{\phi \Z/2}$ is equivalent to the homotopy inverse limit of
\[\xymatrix{\cdots \ar[r]  &  (\TRR^n(T;2))^{\phi \Z/2} \ar[r]^-F & \cdots \ar[r]^-F & (\TRR^2(T;2))^{\phi \Z/2}  \ar[r]^F  &  (\TRR^1(T;2))^{\phi \Z/2}}.\]
For convenience we introduce the notation:
\[A_n := \underbrace{
(T^{\phi})^{C_2}
{}_r\!\times_f
(T^{\phi})^{C_2}
{}_r\!\times_f
\cdots 
{}_r\!\times_f(T^{\phi})^{C_2}{}
}_{n}
{}_r\!\times_{\sigma_1r}
\underbrace{
(T^{\phi})^{C_2}
{}_f\!\times_{\sigma_1r} 
\cdots
{}_f\!\times_{\sigma_1r}
(T^{\phi})^{C_2}
{}_f\!\times_{\sigma_1r} 
(T^{\phi})^{C_2}
}_{n}  \]
and
\[B_n := \underbrace{
(T^{\phi})^{C_2}
{}_r\!\times_f
(T^{\phi})^{C_2}
{}_r\!\times_f
\cdots 
{}_r\!\times_f(T^{\phi})^{C_2}{}
}_{n}.\]
Projecting onto the first $n$ factors gives maps $A_n\to B_n$. These maps commute with the Frobenius $F$ on $A_n$ and the projection $\proj_l$ on $B_n$ by the description of $F$ in Remark \ref{geometric formula}, thus defining a morphism of towers. We will now show that this morphism is a pro-equivalence and thus induces an equivalence on homotopy inverse limits. Hence by Theorem \ref{inductivePB} and Remark \ref{geometric formula} we obtain the description of $\TFR(T;2)^{\phi \Z/2}$.  

Let us define a homotopy pro-inverse $B_{n+1}\to A_n$, for $n \geq 2$,  by the map induced by
\[(x_1, x_2, \dots, x_n, x_{n+1})\longmapsto(x_1, x_2, \dots, x_n, x_n, \dots, x_2,x_1).\]
The identifications between the components in the pullback $A_n$ are defined exactly as in Remark \ref{geometric formula} for the Frobenius, in particular
\[r x_n \simeq f x_{n+1} \simeq \sigma_1fx_{n+1} \simeq \sigma_1 r x_n\]
where the middle homotopy is the canonical one and the last path is just $\sigma_1$ applied to the first homotopy. That this map is indeed a pro-inverse \cite[Section III-\S 2]{BK}, follows immediately from the description of $F$ in Remark \ref{geometric formula}.

Let us now identify $R$ on $\TFR(T;2)^{\phi \Z/2}$. By the description of the map $R$ in Remark \ref{geometric formula} we see that for any $n \geq 1$, the diagram 
\[\xymatrix{A_{n+1} \ar[d]^R \ar[r]^-{\proj} & B_{n+1} \ar[d]^-{ \proj_{f}} 
\\ A_n \ar[r]^-{\proj} &B_n}\]
commutes.
Since the horizontal maps are entries of a pro-equivalence, passing to limits along $F$ gives the desired result. 
 \end{proof}

%
%

Finally we are ready to prove the main result of this section: 

\begin{theorem} \label{tcrformula1}
Let $T$ be a bounded below real $2$-cyclotomic spectrum. Then there is a natural equivalence 
\[
\TCR(T;2)^{\phi\Z/2}\simeq eq\left(
\xymatrix{(T^{\phi\Z/2})^{C_2}\ar@<1ex>[r]^-f\ar@<-1ex>[r]_-r&T^{\phi\Z/2}}\right).
\]
\end{theorem}

\begin{proof} Recall that $\TCR(T;2)^{\phi\Z/2}$ is equivalent to the equaliser
\[\TCR(T;2)^{\phi\Z/2}\simeq eq\left(
\xymatrix{\TFR(T;2)^{\phi \Z/2} \ar@<1ex>[r]^-{\id} \ar@<-1ex>[r]_-R&\TFR(T;2)^{\phi \Z/2}}\right).
\]
Now consider the commutative diagram
\[\small \xymatrix@C=14pt{ \cdots \ar[r]^-{\proj_l} & (T^{\phi})^{C_2}
{}_r\!\times_f
(T^{\phi})^{C_2}
{}_r\!\times_f (T^{\phi})^{C_2}
{}_r\!\times_f (T^{\phi})^{C_2} \ar[r]^-{\proj_{l}}  \ar@<-0.5ex>[d]_-{\proj_{f}}\ar@<+0.5ex>[d]^-{\proj_{l}}  & (T^{\phi})^{C_2}
{}_r\!\times_f
(T^{\phi})^{C_2}
{}_r\!\times_f(T^{\phi})^{C_2}   \ar[r]^-{\proj_{l}}  \ar@<-0.5ex>[d]_-{\proj_{f}} \ar@<+0.5ex>[d]^-{\proj_{l}}  & (T^{\phi})^{C_2}
{}_r\!\times_f
(T^{\phi})^{C_2} \ar@<-0.5ex>[d]_-{\proj_{f}} \ar@<+0.5ex>[d]^-{\proj_{l}}   \\  \cdots \ar[r]^-{\proj_{l}}  & (T^{\phi})^{C_2}
{}_r\!\times_f
(T^{\phi})^{C_2}
{}_r\!\times_f (T^{\phi})^{C_2}
\ar[r]^-{\proj_{l}}  & (T^{\phi})^{C_2}
{}_r\!\times_f
(T^{\phi})^{C_2}  \ar[r]^-{\proj_{l}}  & (T^{\phi})^{C_2}.
   }\]
By Theorem \ref{TFformula} if we pass to the inverse limits horizontally and then take the equaliser we get $\TCR(T;2)^{\phi\Z/2}$. Equivalently we can take equalisers in each degree vertically and then pass to the inverse limit. In general, given maps $a,b \colon X \to Y$, the equaliser of the projections
\[\xymatrix{\underbrace{
X
{}_a\!\times_b
X
{}_a\!\times_b
\cdots 
{}_a\!\times_bX{}
}_{n+1} \ar@<-0.5ex>[rr]_-{\proj_{f}} \ar@<+0.5ex>[rr]^-{\proj_{l}} & & \underbrace{
X
{}_a\!\times_b
X	
{}_a\!\times_b
\cdots 
{}_a\!\times_bX{}
}_{n}},\]
off the first and last component (and where all the products denote fibre products over $Y$), is equivalent to the equaliser of $a$ and $b$. Thus each vertical equaliser above is equivalent to 
\[eq\left(
\xymatrix{(T^{\phi\Z/2})^{C_2}\ar@<1ex>[r]^-f\ar@<-1ex>[r]_-r&T^{\phi\Z/2}}\right)\]
and the induced maps are equivalences.
\end{proof}

\section{TCR of spherical monoid-rings}\label{sec:grouprings}

We apply the formulas of the previous section to calculate the geometric fixed points of the real topological cyclic homology of spherical monoid-rings, and in particular for the sphere spectrum. In \S\ref{sec:assembly} we give the general formula and analyse a certain assembly map, and in  \S\ref{sec:discrete} we carry out some calculations for discrete groups.

\subsection{TCR of spherical monoid-rings and assembly}\label{sec:assembly}
Let $M$ be a topological monoid with anti-involution, that is a map of monoids $w\colon M^{op}\to M$ such that $w^2=\id$ (e.g. $M$ is a group and $w$ is inversion). The $\Z/2$-equivariant suspension spectrum of the underlying $\Z/2$-equivariant space
\[
\mathbb{S}[M]:=\Sigma^{\infty}_+M
\]
is then a ring spectrum with anti-involution, where the multiplication is inherited from the multiplication of $M$. 
We recall that, since $\mathbb{S}$ is the monoidal unit of the tensor product of spectra, there is an equivalence of $O(2)$-spectra
\[
\THR(\mathbb{S}[M])\simeq\Sigma^{\infty}_+B^{di}M,
\]
where $B^{di}M$ is the dihedral bar construct of the monoid $M$ with respect to the product of spaces (see \cite[Proposition 5.12]{THRmodels}). Thus from Corollary \ref{tcrodd} and Lemma \ref{C2THRphi} we immediately obtain that for every odd prime $p$
\[
\TCR(\mathbb{S}[M];p)^{\phi\Z/2}\simeq (\Sigma^{\infty}_+B^{di}M)^{\phi\Z/2}\simeq \Sigma^{\infty}_+B(M^{\Z/2},M,M^{\Z/2}),
\]
where the two-sided bar construction is for the left and right actions of $M$ on its fixed-points subspace $M^{\Z/2}$ defined respectively by $m\cdot x=mxw(m)$ and $x\cdot m=w(m)xm$. 

For the prime $2$ the situation is more delicate. The Weyl action of $C_2$ on
\[
(B^{di}M)^{\Z/2}\cong B(M^{\Z/2},M,M^{\Z/2})
\]
is given, after an edgewise subdivision of the simplicial space $B(M^{\Z/2},M,M^{\Z/2})$, by reversing the order of the product factors of the  $n$-simplices $M^{\Z/2}\times M^{\times 2n+1}\times M^{\Z/2}$ and apply $w$ on the $M$-factors. Therefore there is an isomorphism
\[
B(M^{\Z/2},M,M^{\Z/2})^{C_2}\cong |(\sd_e N_{\bullet}(M^{\Z/2},M,M^{\Z/2})^{C_2}|\cong B(M^{\Z/2},M,M^{\Z/2}),
\] 
which corresponds to the residual cyclotomic structure on the $\Z/2$-geometric fixed-points of $\THR(\mathbb{S}[M])$  (cf. Example \ref{minicycloTHR}). Under this identification, the fixed-points inclusion corresponds to the endomorphism $\psi$ of $B(M^{\Z/2},M,M^{\Z/2})$ given in simplicial degree $n$ by
\[
\psi(x,m_1,\dots,m_n,y)=(x,m_1,\dots,m_n,yw(m_n)\dots w(m_1)xm_1\dots m_ny),
\]
that is to say there is a commutative diagram
\[
\xymatrix{
B(M^{\Z/2},M,M^{\Z/2})^{C_2}\ar[dr]\ar[r]^-{\cong}&B(M^{\Z/2},M,M^{\Z/2})\ar[d]^-{\psi}
\\
& B(M^{\Z/2},M,M^{\Z/2})
}
\]
where the diagonal map is the fixed-points inclusion.
\begin{example}\label{psiOmega}
The typical example of a monoid with anti-involution is the signed loop space
\[M=\Omega^\sigma X:=\map_\ast(S^\sigma,X)\]
where $X$ is a pointed $\Z/2$-space, $S^\sigma$ is the sign representation sphere, and $\Z/2$ acts on the loop space by conjugation. In this case the dihedral bar construction is equivalent to the free loop space
\[
B^{di}\Omega^\sigma X\simeq \Lambda^\sigma X:=\map(S^\sigma,X)
\]
again with the conjugation action of $\Z/2$ (see \cite[Remark 5.13]{THRmodels}). Let us spell out the map $\psi$ under this identification. By passing to the upper half-circle, the $\Z/2$-fixed points of $\Lambda^\sigma X$ can be identified with those paths in $X$ which start and end at a fixed-point, or in other words the homotopy pullback
\[
(\Lambda^\sigma X)^{\Z/2}\simeq X^{\Z/2}\times_X X^{\Z/2}
\]
of the fixed-points inclusion along itself. Since the $C_2$-action on $\Lambda^\sigma X$ is given by the $180^{\circ}$ degrees rotation followed by the involution of $X$ pointwise, the residual $C_2$-action on $X^{\Z/2}\times_X X^{\Z/2}$ flips the direction of the path and applies the involution of $X$ pointwise. Hence, there is an isomorphism
\[
(X^{\Z/2}\times_X X^{\Z/2})^{C_2}\cong X^{\Z/2}\times_X X^{\Z/2}
\]
that restricts a $C_2$-fixed path $\gamma\colon [0,1]\to X$ to $[0,1/2]$. Under this identification the fixed-points inclusion corresponds to the map
\[
\psi\colon X^{\Z/2}\times_X X^{\Z/2}\longrightarrow X^{\Z/2}\times_X X^{\Z/2}
\]
which sends a path $\gamma$ to the concatenation $\gamma\star w(\overline{\gamma})$, where $\overline{\gamma}$ is the inverse path. This is some sort of squaring operation reminiscent of the Frobenius.
\end{example}

We are finally able to describe the geometric fixed-points of $\TCR(\mathbb{S}[M];2)$ (notice the analogy with \cite{BHM} and \cite[Theorem IV.3.6]{NS}):

\begin{theorem}\label{TCRSM}
Let $M$ be a well-pointed topological monoid with anti-involution. Then there is a pullback square
\[
\xymatrix@C=50pt{
\TCR(\mathbb{S}[M];2)^{\phi\Z/2}\ar[r]\ar[d]&\Sigma^{\infty}_+B(M^{\Z/2},M,M^{\Z/2})_{hC_2}\ar[d]^-{\trf}
\\
\Sigma^{\infty}_+B(M^{\Z/2},M,M^{\Z/2})\ar[r]^-{\id-\Sigma^{\infty}_+\psi}&\Sigma^{\infty}_+B(M^{\Z/2},M,M^{\Z/2})
}
\]
where the right vertical map is the transfer from orbits to fixed-points, followed by the forgetful map to the underlying spectrum.
In particular for $M=\ast$ there is an equivalence 
\[ {\TCR(\mathbb S;2)}^{\phi \Z/2}\simeq {\mathbb S} \oplus {\mathbb R}P^{\infty}_{-1},\]
where ${\mathbb R}P^{\infty}_{-1}$ is the fibre of the transfer $\trf \colon \Sigma^{\infty}_{+} {\mathbb R}P^{\infty}=\mathbb{S}_{hC_2} \to \mathbb{S}$.
\end{theorem}

\begin{proof}
From the identification of $\THR(\mathbb{S}[M])$ with the dihedral bar construction of $M$ we obtain an equivalence of $C_2$-spectra
\[
\THR(\mathbb{S}[M])^{\phi\Z/2}\simeq (\Sigma^{\infty}_+B^{di}M)^{\phi\Z/2}\simeq \Sigma^{\infty}_+B(M^{\Z/2},M,M^{\Z/2}).
\]
Thus the tom-Dieck splitting (see e.g. \cite[Section 6]{Schwede}) gives an equivalence
\begin{align*}
(\THR(\mathbb{S}[M])^{\phi\Z/2})^{C_2}&\simeq (\Sigma^{\infty}_+B(M^{\Z/2},M,M^{\Z/2}))^{C_2}
\\&\simeq \Sigma^{\infty}_+(B(M^{\Z/2},M,M^{\Z/2})^{C_2})\oplus \Sigma^{\infty}_+B(M^{\Z/2},M,M^{\Z/2})_{hC_2}
\\&\cong \Sigma^{\infty}_+B(M^{\Z/2},M,M^{\Z/2})\oplus \Sigma^{\infty}_+B(M^{\Z/2},M,M^{\Z/2})_{hC_2} .
\end{align*}
where the maps $r$ and $f$ correspond respectively to the maps
\[
(1,0),(\psi,\trf)\colon \Sigma^{\infty}_+B(M^{\Z/2},M,M^{\Z/2})\oplus \Sigma^{\infty}_+B(M^{\Z/2},M,M^{\Z/2})_{hC_2} \to \Sigma^{\infty}_+B(M^{\Z/2},M,M^{\Z/2})
\]
where $r$ is the projection onto the first summand, and $f$ is the map $\psi$ on the first summand and the transfer on the second. Thus the equaliser of $r$ and $f$ is computed by the pullback above, and it is equivalent to $\TCR(\mathbb{S}[M];2)^{\phi\Z/2}$ by Theorem \ref{geometricsplitting}.

If $M=\ast$ the bottom horizontal map of the pullback square is zero, and the pullback splits as the fibre of $\trf$ and $\mathbb{S}$.
\end{proof}

\begin{rem}
The explicit identification of the maps $r,f$ of the proof of \ref{TCRSM} in fact gives a description of the full $\TR$-tower of $\mathbb{S}[M]$. 
Indeed, one can see by direct calculation that for every $2\leq n\leq \infty$ there is an equivalence of spectra
\[
\TRR^{n}(\mathbb{S}[M];2)^{\phi \Z/2} \simeq \Sigma^{\infty}_+B(M^{\Z/2},M,M^{\Z/2}) \times \prod_{j=1}^{2n-2} \Sigma^{\infty}_+B(M^{\Z/2},M,M^{\Z/2})_{hC_2},\]
and the maps $R,F \colon  \TRR^{n+1}(\mathbb{S}[M];2)^{\phi \Z/2}  \to \TRR^{n}(\mathbb{S}[M];2)^{\phi \Z/2}$ are given respectively by the projection
\[
R(a, x_{-n}, \dots, x_{-1},x_1,\dots,x_n)=(a, x_{-n+1}, \dots, x_{-1},x_1,\dots,x_{n-1})
\]
and by
\[F(a, x_{-n}, \dots, x_{-1},x_1,\dots,x_n)=(\psi(a)+\trf(x_{-1}), x_{-n}, \dots, x_{-2},x_{-2},\dots,x_{-n}).\]
\end{rem}

Let us now turn our attention to the case where $G$ is a group-like topological monoid with involution, that is a  topological monoid with involution $G$ such that $\pi_0G$ is a group. In this case the canonical map
\[
G\longrightarrow \Omega^{\sigma}B^\sigma G
\]
is an equivalence (see \cite{Kristian}), where $B^{\sigma}G$ is the subdivision of $BG$ with the simplicial involution that sends $(g_1,\dots,g_{2n+1})$ to $(w(g_{2n+1}),\dots,w(g_1))$. The fixed-points space of $X=B^{\sigma}G$ is then the one-sided bar construction
\[
(B^{\sigma}G)^{\Z/2}\simeq B(G^{\Z/2}, G)
\]
of $G$ acting on its fixed-points set by $x \cdot g =w(g)xg$. We will therefore phrase the next results in terms of signed loop spaces $G=\Omega^\sigma X$, where $X$ is any pointed $\Z/2$-space. We also note that the fixed-points subspace of $G=\Omega^\sigma X$ is the fibre of the inclusion
\[
G^{\Z/2}=(\Omega^\sigma X)^{\Z/2}=\fib(X^{\Z/2}\to X)
\]
where $\alpha \in\Omega X$ acts on a path $\gamma$ from the base-point to a fixed-point by concatenation $\gamma \cdot \alpha =\omega (\overline{\alpha})\star \gamma$.

\begin{cor}\label{splittingTCRSG}
For every well-pointed $\Z/2$-space $X$ there is a fibre sequence
\[
(\Sigma^{\infty}_+X^{\Z/2})\otimes(\mathbb{S}\oplus {\mathbb R}P^{\infty}_{-1})\xrightarrow{\Delta} \TCR(\mathbb{S}[\Omega^\sigma X];2)^{\phi\Z/2}\to Q,
\]
where $Q$ is the pullback of $\Sigma^{\infty}_+C\xrightarrow{\id-\Sigma^{\infty}_+\psi} \Sigma^{\infty}_+C\xleftarrow{\trf} \Sigma^{\infty}_+C_{hC_2}$ and $C$ the cofibre of the diagonal $\Delta\colon X^{\Z/2}\to X^{\Z/2}\times_XX^{\Z/2}$.

If the involution of $X$ is trivial, $Q$ is zero and there is a natural equivalence
\[
\TCR(\mathbb{S}[\Omega^\sigma X];2)^{\phi\Z/2}\simeq (\Sigma^{\infty}_+X)\otimes(\mathbb{S}\oplus {\mathbb R}P^{\infty}_{-1}).
\]
\end{cor}

\begin{proof}
The diagonal $\Delta\colon X^{\Z/2}\to X^{\Z/2}\times_XX^{\Z/2}$ is clearly equivariant for the Weyl $C_2$-action on the pullback and the trivial action on $X^{\Z/2}$. It therefore induces a commutative diagram
\[
\xymatrix@C=50pt{
\Sigma^{\infty}_+(X^{\Z/2}\times_XX^{\Z/2})\ar[r]^-{\id-\Sigma^{\infty}_+\psi}
&
\Sigma^{\infty}_+(X^{\Z/2}\times_XX^{\Z/2})
&
\Sigma^{\infty}_+ (X^{\Z/2}\times_XX^{\Z/2})_{hC_2}\ar[l]_-{\trf}
\\
\Sigma^{\infty}_+(X^{\Z/2})\ar[r]^-{0}\ar[u]^{\Delta}
&
\Sigma^{\infty}_+(X^{\Z/2})\ar[u]^{\Delta}
&
\Sigma^{\infty}_+ (X^{\Z/2})_{hC_2}\ar[l]_-{\trf}\ar[u]^{\Delta}
}
\]
where the bottom left map is zero since $\psi$ is the identity on constant paths. The limit of the top row is $\TCR(\mathbb{S}[\Omega^\sigma X];2)^{\phi\Z/2}$ by Theorem \ref{TCRSM} and Example \ref{psiOmega} the limit of the bottom row is $(\Sigma^{\infty}_+X^{\Z/2})\otimes(\mathbb{S}\oplus {\mathbb R}P^{\infty}_{-1})$. By taking cofibres vertically we obtain the fibre sequence of the statement. 


 If the involution on $X$ is trivial the diagonal map $\Delta\colon X\to X\times_XX$ is an equivalence and thus $Q$ is trivial.
\end{proof}

If the involution of $X$ is not trivial the cofibre $Q$ need not be zero, as illustrated in the following example.

\begin{example}
Suppose  that $X$ is a pointed space, and let us consider the pointed free $\Z/2$-space $X^b=X\wedge E\Z/2_+$. Since its fixed-points are contractible, the fixed-points of $G=\Omega^\sigma X^b$ are
\[
G^{\Z/2}=(\Omega^\sigma X^b)^{\Z/2}\simeq \Omega X,
\]
since this is the space of paths from the base-point to a fixed-point of $X^b$.
In this case 
\[B(G^{\Z/2},G,G^{\Z/2})\simeq B(\Omega X,\Omega X,\Omega X)\simeq \Omega X,\]
and the map $\psi\colon \Omega X\to \Omega X$ sends a loop $\gamma$ to $\gamma\star \overline{\gamma}$ and it is therefore null. By Theorem \ref{TCRSM} there is a pullback
\[
\xymatrix@C=50pt{
\TCR(\mathbb{S}[\Omega^\sigma X^b];2)^{\phi\Z/2}\ar[r]\ar[d]&(\Sigma^{\infty}_+ \Omega^{\sigma} X)_{hC_2}\ar[d]^-{\trf}
\\
\Sigma^{\infty}_+\Omega X\ar[r]^-{\id-0_+}&\Sigma^{\infty}_+\Omega X
}
\]
where $C_2$ acts on $\Omega^{\sigma} X$ by the loop inversion.
There is therefore a splitting
\[
\TCR(\mathbb{S}[\Omega^\sigma X^b];2)^{\phi\Z/2}\simeq \mathbb{S}\oplus {\mathbb R}P^{\infty}_{-1}\oplus (\Sigma^{\infty} \Omega^\sigma X)_{hC_2}.
\]
In this case the map $\Delta$ of Corollary \ref{splittingTCRSG} is easily seen to split, and the homotopy orbits summand corresponds to the summand $Q$.
\end{example}

\subsection{TCR of spherical group-rings for some discrete groups}\label{sec:discrete}

Let us now consider a discrete group $G$ with anti-involution. The map $\Delta$ of Corollary \ref{splittingTCRSG} corresponds to the simplicial map
\[
B(G^{\Z/2},G)\longrightarrow B(G^{\Z/2},G,G^{\Z/2})
\]
that sends $(x,g_1,\dots,g_n)$ to $(x,g_1,\dots,g_n,g_n^{-1}\dots g^{-1}_1x^{-1}w(g_n^{-1}\dots g^{-1}_1))$. This follows from identifying 
\[B(G^{\Z/2},G) \times_{BG} B(G^{\Z/2},G) \]
with $B(G^{\Z/2},G,G^{\Z/2})$ via the map 
\[(x, g_1, \dots, g_n, y, h_1, \dots,h_n) \mapsto (x, g_1, \dots, g_n, h_n^{-1}\dots h^{-1}_1y^{-1}w(h_n^{-1}\dots h^{-1}_1)).\]

\begin{example}
 Suppose that the involution of $G$ is inversion $w=(-)^{-1}\colon G^{op}\to G$. Then the fixed-points space of $G$ consists of the set of elements of order $2$.
If $G$ has no $2$-torsion we are in the situation of Corollary \ref{splittingTCRSG} where $G^{\Z/2}=1$ and $(B^{\sigma}G)^{\Z/2}\simeq BG$, and  
 \[
 \TCR(\mathbb S[G];2)^{\phi \Z/2}\simeq (\Sigma^{\infty}_+BG)\otimes(\mathbb{S}\oplus {\mathbb R}P^{\infty}_{-1}).
\]
For example let us consider the spherical Laurent polynomial ring $\mathbb{S}[t,t^{-1}]:=\mathbb{S}[\mathbb{Z}]$, where the involution acts by inversion in $\mathbb{Z}$, i.e. swaps $t$ and $t^{-1}$. Then
 \[
 \TCR(\mathbb S[t,t^{-1}];2)^{\phi \Z/2}\simeq (\Sigma^{\infty}_+S^1)\otimes(\mathbb{S}\oplus {\mathbb R}P^{\infty}_{-1}).
\]
\end{example}

Now suppose that $G$ is a discrete group with a general anti-involution $w\colon G^{op}\to G$. 
The bar construction $B(G^{\Z/2},G,G^{\Z/2})$ is the nerve of a groupoid, and therefore after a choice of representatives for its isomorphism classes it decomposes as
\[
B(G^{\Z/2},G,G^{\Z/2})\cong\coprod_{[x,y]\in (G^{\Z/2}\times G^{\Z/2})/\sim}BAut(x,y)
\]
where the equivalence relation identifies $(w(g)xg,y)$ with $(x,gyw(g))$, and the automorphism group of $(x,y)\in G^{\Z/2}\times G^{\Z/2}$ is the subgroup $Aut(x,y)=\{g\in G\ |\ w(g)xg=x\  ,\  gyw(g)=y\}$. Let us now determine the map $\psi$ and the Weyl action, so that all the ingredients of Theorem \ref{TCRSM} are in place.

\begin{lemma}
The maps $\psi$ and the action of the generator $\tau$ of the Weyl group (as a homotopy coherent action) are given, under the decomposition of $B(G^{\Z/2},G,G^{\Z/2})$ above, respectively by the maps
\[
\psi([x,y],g)=([x,yxy],g)\ \ \ \ \ \ \  \mbox{and}  \ \ \ \  \ \ \ \tau([x,y],g)=([y,x],w(g^{-1})).
\]
In particular the homotopy orbits of  $B(G^{\Z/2},G,G^{\Z/2})$ for the Weyl action can be computed using this strict action of $\tau$.
\end{lemma}

\begin{proof}
The description of $\psi$ is immediate from the formula before Example \ref{psiOmega}. The description of the action of the generator follows from the well-known fact that if $\mathcal{G}$ is a groupoid with strict duality, that is a functor $w\colon \mathcal{G}^{op}\to \mathcal{G}$ such that $w^2=\id$, then the $\Z/2$-actions on $B\mathcal{G}$ defined respectively by the levelwise duality together with inverting the order of the simplex coordinates, as in \S\ref{sec:subdivisions}, and the one defined by the endofunctor
\[
\mathcal{G}\xrightarrow{(-)^{-1}}\mathcal{G}^{op}\xrightarrow{w} \mathcal{G},
\]
are homotopy coherently equivalent. We were not able to track down a proof, so we include an argument for the reader's convenience. After applying the subdivision functor $\sd_e$ from \S\ref{sec:subdivisions} to the nerve of $\mathcal{G}$, the two actions are respectively equivalent to the simplicial actions defined levelwise on $\sd_eN\mathcal{G}$ by $w$ and $w\circ (-)^{-1}$. The subdivided nerve $\sd_eN\mathcal{G}$ is isomorphic to the nerve of the twisted arrow category of $\mathcal{G}$, and the two actions correspond respectively to the ones induced by the (covariant) endofunctors $w$ and $w\circ (-)^{-1}$, defined on the objects of the twisted arrow category respectively by
\[
w(x\xrightarrow{g}y)=(w(y)\xrightarrow{w(g)}w(x))\ \ \ \ \ \ \  \mbox{and}  \ \ \ \  \ \ \ w((x\xrightarrow{g}y)^{-1})=(w(x)\xrightarrow{w(g^{-1})}w(y)).
\]
Let us regard these actions as functors $\Z/2\to \Gpd$ to the category of groupoids, where both send the unique object of $\Z/2$ to the twisted arrow category of $\mathcal{G}$. Then the diagram
\[
\xymatrix{
w(y)\ar[r]^-{w(g)}\ar[d]_-{w(g)}&w(x)
\\
w(x)\ar[r]^-{w(g^{-1})}&w(y)\ar[u]_-{w(g)}
}
\]
exhibits the 2-cells of a pseudonatural isomorphism on the identity transformation between the functors $\Z/2\to \Gpd$. Thus the two actions on the geometric realisations are homotopy coherently equivalent.
\end{proof}

With these formulas at hand one should in principle be able to determine the pullback of Theorem \ref{TCRSM}, as its maps consist of products of diagonals and transfers $\Sigma^{\infty}_+{\mathbb R}P^{\infty}\to \mathbb{S}$. The combinatorics of which components are hit by the diagonals are complicated in this generality, but we compute them fully in the following special cases.

\begin{example}
 Suppose that the order-two elements of $G$ are included in the centre of $G$, and that the involution on $G$ is inversion. This is exactly the situation where the action of $G$ on $G^{\Z/2}$ is trivial. It follows that $G^{\Z/2}=G_2$ consists of the elements of order $2$, and
\[
B(G_2,G,G_2)=\coprod_{G_2\times G_2}BG.
\]
The map $\psi$ sends the component $(x,y)$ to the component is $(x,x)$ via the identity of $BG$, and the involution freely permutes the components indexed by pairs $(x,y)$ with $x\neq y$, and is trivial on the components $(x,x)$. There is therefore a splitting
\[
\TCR(\mathbb{S}[G];2)^{\phi\Z/2}\simeq  ((\Sigma^{\infty}_+G_{2})\oplus P)\otimes \Sigma^{\infty}_+BG,
\]
where $P$ is the pullback
\[
\xymatrix@C=50pt{
P\ar[r]\ar[d]&(G_2)_+\otimes\Sigma^{\infty}_+{\mathbb R}P^{\infty}\ar[d]^-{\id\otimes\trf}
\\
(\Delta^c/C_2)_+\otimes \mathbb{S}\ar[r]^-{
q
}&(G_2)_+\otimes \mathbb{S}
}
\]
where $\Delta^c\subset G_2\times G_2$ is the complement of the diagonal with the involution that flips the factors, and $q$ is the sum of the maps that send the component $[x\neq y]$ respectively to the components $x$ and $y$.

For example for $G=\mathbb{Z}$ with the minus involution we recover the calculation for $\mathbb{S}[t,t^{-1}]$ of the example above, since in this case $G_2=1$. On the other hand for $G=C_2$ the map $q$ is the diagonal, and $P$ is the pullback of the transfer along itself, and
\[
\TCR(\mathbb S[C_2];2)^{\phi \Z/2}\simeq ((\Sigma^{\infty}_+C_{2})\oplus (\Sigma^{\infty}_+{\mathbb R}P^{\infty}\times_{\mathbb{S}} \Sigma^{\infty}_+{\mathbb R}P^{\infty}))\otimes \Sigma^{\infty}_+BC_2.
\]
Here we also see that $\TCR(\mathbb S[C_2];2)^{\phi \Z/2}$ splits off an ${\mathbb R}P^{\infty}_{-1}\otimes \Sigma^{\infty}_+BC_2$-summand, since the pullback of the two transfers splits as ${\mathbb R}P^{\infty}_{-1}\oplus  \Sigma^{\infty}_+{\mathbb R}P^{\infty}$, but this splitting is however non-canonical.

Notice that $P$ depends only on the order-two elements of $G$, so in fact for every even integer $n\geq 2$
\[
\TCR(\mathbb S[C_n];2)^{\phi \Z/2}\simeq ((\Sigma^{\infty}_+C_{2})\oplus (\Sigma^{\infty}_+{\mathbb R}P^{\infty}\times_{\mathbb{S}} \Sigma^{\infty}_+{\mathbb R}P^{\infty}))\otimes \Sigma^{\infty}_+BC_n
\]
where again the involution on $C_n$ is inversion. 
\end{example}

\begin{example}
 Now suppose that $G$ is abelian and endowed with the trivial involution. Then $G^{\Z/2}=G$ with left and right $G$-actions $g\cdot x:=2g+x$. The components of the two-sided bar construction are described by a bijection
\[
(G\times G)/\sim\ \ \cong G\times G/2
\]
which sends $[x,y]$ to $(x+y,[y])$. Under this equivalence the $C_2$-action sends $(x,z)$ to $(x,[x]+z)$, and $\psi$ to
\[
\psi(x,z)=(2x,[x]+z).
\]
The $C_2$-fixed-points set of $G\times G/2$ is therefore the set of pairs of the form $(2g,x)$, and $G\times G/2$ decomposes $C_2$-equivariantly as
\[
G\times G/2\cong (2G\times G/2)\amalg (((G\setminus 2G)\times G/2)/C_2)\times C_2.
\]
If we assume additionally that $G$ has no $2$-torsion, then the fundamental groups of the two-sided bar construction vanish since the corresponding groupoid has only trivial automorphisms. The pullback diagram describing $\TCR(\mathbb S[G];2)^{\phi \Z/2}$ then takes the form
\[
\xymatrix@C=50pt{
\TCR(\mathbb{S}[G];2)^{\phi\Z/2}\ar[r]\ar[d]&((2G\times G/2)_+\otimes\Sigma^{\infty}_+{\mathbb R}P^{\infty})\oplus (((G\setminus 2G)\times G/2)/C_2)_+\otimes \mathbb{S}
\ar[d]^-{\incl\otimes\trf\oplus \Delta}
\\
(G\times G/2)_+\otimes\mathbb{S}\ar[r]^-{\id-\Sigma^{\infty}_+\psi}&(G\times G/2)_+\otimes\mathbb{S}
}
\]
where $\Delta$ sends the component of an orbit $[g,x]$ with $g\notin 2G$ diagonally to the components $(g,x)$ and $(g,[g]+x)$.

Let us now identify this pullback under the additional assumption that $G$ does not have elements infinitely divisible by $2$, that is for any $0\neq g \in G$  there exists $n \in \mathbb{N}$, such that $g=2^nx$ does not have a solution. Under this assumption we can easily compute the cofibre of $\id - \Sigma^{\infty}_+\psi$. Indeed, from the commutative diagram
\[ \xymatrix{(2G\times G/2)_+\otimes\mathbb{S}\ar[r]^-{\id-\Sigma^{\infty}_+\psi} \ar@{^{(}->}[d] &(2G\times G/2)_+\otimes\mathbb{S} \ar@{^{(}->}[d] \\ (G\times G/2)_+\otimes\mathbb{S}\ar[r]^-{\id-\Sigma^{\infty}_+\psi} \ar[d] &(G\times G/2)_+\otimes\mathbb{S} \ar[d] \\ ((G\setminus 2G) \times G/2)_+\otimes\mathbb{S}\ar@{-->}[r]^{\simeq} &((G\setminus 2G) \times G/2)_+\otimes\mathbb{S}   }  \]
we see that the lower horizontal map induced on cofibres is an equivalence. Hence the cofibre of $\id - \Sigma^{\infty}_+\psi$ is equivalent to the cofibre of its restriction
\[\id - \Sigma^{\infty}_+\psi \colon (2G\times G/2)_+\otimes\mathbb{S}  \to (2G\times G/2)_+\otimes\mathbb{S}.\]
Since $\psi(0,z)=(0,z)$, we see that the zero map $0 \colon (G/2)_+\otimes\mathbb{S} \to (G/2)_+\otimes\mathbb{S}$ splits off from the given map and hence the cofibre contains the summand $(G/2)_+\otimes(\mathbb{S}  \oplus \mathbb{S}^1)$.  Let us now compute the cofibre of 
\[\id - \Sigma^{\infty}_+\psi \colon ((2G\setminus 0) \times G/2)_+\otimes\mathbb{S}  \to ((2G\setminus 0)\times G/2)_+\otimes\mathbb{S}.\]
By the non-divisibility condition, this morphism induces injection on homotopy groups. By inspecting its cokernel on homotopy groups, we see that the cofibre is $((2G\setminus 4G) \times G/2)_+ \otimes \mathbb{S}$.
All in all we get a cofibre sequence
\[ \xymatrix@C=15pt{(G\times G/2)_+\otimes\mathbb{S} \ar[rr]^-{\id-\Sigma^{\infty}_+\psi} && (G\times G/2)_+\otimes\mathbb{S} \ar[r]^-{\zeta} & (((2G\setminus 4G) \times G/2)_+ \otimes \mathbb{S}) \oplus ((G/2)_+\otimes(\mathbb{S}\oplus\mathbb{S}^1)),  } \]
where $\zeta$ includes $(0 \times (G/2))_+\otimes\mathbb{S}$ into $(G/2)_+\otimes\mathbb{S}$, sends the $(2^n h \times G/2)_+\otimes\mathbb{S} $-summand via the identity to the summand $(2 h \times G/2)_+\otimes\mathbb{S} $ for any $h \in G \setminus 2G$ and $n\geq 1$, and sends the summand $((G \setminus 2G) \times G/2)_+ \otimes \mathbb{S})$ to $(((2G\setminus 4G) \times G/2)_+ \otimes \mathbb{S})$ via $(g,x) \mapsto (2g, g+x)$. 

From the pullback square above we find that $\TCR(\mathbb{S}[G];2)^{\phi\Z/2}$ is the fibre of the map 
\begin{align*}
 \zeta \circ (\incl\otimes\trf\oplus \Delta): &((2G\times G/2)_+\otimes\Sigma^{\infty}_+{\mathbb R}P^{\infty})\oplus (((G\setminus 2G)\times G/2)/C_2)_+\otimes \mathbb{S} \longrightarrow
 \\
 &(((2G\setminus 4G) \times G/2)_+ \otimes \mathbb{S}) \oplus ((G/2)_+\otimes(\mathbb{S}\oplus\mathbb{S}^1)),
\end{align*}
which is given by the wedge
\[ ((G/2)_+\otimes(\mathbb{S}\oplus {\mathbb R}P^{\infty}_{-1})) \oplus  P \]
where $P$ is the pullback
\[\xymatrix{ P \ar[r] \ar[d] & ((2G \setminus 0)\times G/2)_+\otimes\Sigma^{\infty}_+{\mathbb R}P^{\infty} \ar[d] \\ (((G\setminus 2G)\times G/2)/C_2)_+\otimes \mathbb{S} \ar[r] & (((2G\setminus 4G) \times G/2)_+ \otimes \mathbb{S}).}\]
By using that any non-zero element $g \in G$ can be uniquely written as $2^n \gamma$, where $n$ is a non-negative integer and $\gamma \in G \setminus 2G$, we can write $P$ as
\[P \simeq (((G\setminus 2G)\times G/2)/C_2)_+ \otimes ((\mathbb{N}_+\otimes\Sigma^{\infty}_+{\mathbb R}P^{\infty}) \times_{\mathbb{S}}  (\mathbb{N}_+\otimes\Sigma^{\infty}_+{\mathbb R}P^{\infty})). \]
We note that $((\mathbb{N}_+\otimes\Sigma^{\infty}_+{\mathbb R}P^{\infty}) \times_{\mathbb{S}}  (\mathbb{N}_+\otimes\Sigma^{\infty}_+{\mathbb R}P^{\infty}))$ is non-canonically equivalent to
\[{\mathbb R}P^{\infty}_{-1} \oplus \Sigma^{\infty}_+{\mathbb R}P^{\infty} \oplus \Sigma^{\infty}_+{\mathbb R}P^{\infty}  \oplus \Sigma^{\infty}_+{\mathbb R}P^{\infty}  \oplus \dots. \]
To summarise, for every abelian $G$ with trivial involution, no $2$-torsion, and no elements infinitely divisible by $2$
\begin{align*}\TCR(\mathbb{S}[G];2)^{\phi\Z/2} &\simeq  ((G/2)_+\otimes(\mathbb{S}\oplus {\mathbb R}P^{\infty}_{-1})) 
\\&\oplus (((G\setminus 2G)\times G/2)/C_2)_+ \otimes ((\mathbb{N}_+\otimes\Sigma^{\infty}_+{\mathbb R}P^{\infty}) \times_{\mathbb{S}}  (\mathbb{N}_+\otimes\Sigma^{\infty}_+{\mathbb R}P^{\infty})). \end{align*}

In particular the group $G=\mathbb{Z}$ with the trivial involution gives rise to the spherical Laurent polynomials $\mathbb{S}[t,t^{-1}]:=\mathbb{S}[\mathbb{Z}]$  with the involution which acts trivially on the generators, and
\[\TCR(\mathbb{S}[t,t^{-1}];2)^{\phi \Z/2}\simeq ((\mathbb{Z}/2)_+\otimes (\mathbb{S}\oplus {\mathbb R}P^{\infty}_{-1}))\oplus (\mathbb{Z}_+\otimes ((\mathbb{N}_+\otimes\Sigma^{\infty}_+{\mathbb R}P^{\infty}) \times_{\mathbb{S}}  (\mathbb{N}_+\otimes\Sigma^{\infty}_+{\mathbb R}P^{\infty})))   \]
where we took the liberty of enumerating the summands non-canonically.
\end{example}

\section{TCR of perfect fields}

In \cite{Wittvect}, Hesselholt and Madsen identified the $p$-typical topological cyclic homology spectrum $\TC(k;p)$ of a perfect field $k$ of characteristic $p$ as the sum
\[ \TC(k;p) \simeq H\Z_p \oplus \Sigma^{-1}H \coker(1-F),\]
where $F\colon W(k;p)\to W(k;p)$ is the Frobenius homomorphism of the ring $W(k;p)$ of $p$-typical Witt vectors.
 Their calculation relies on the fact that the ring $\pi_0\THR(A)^{C_{p^n}}$ is isomorphic to the ring $W_{n+1}(A;p)$ of $(n+1)$-truncated $p$-typical Witt vectors, which holds for every commutative ring $A$ (see \cite[Theorem F]{Wittvect}).
The situation for $\pi_0\THR(A)^{D_{p^n}}$ is not completely analogous, and requires particular care.

We start by recalling from \cite[Corollary 5.2]{THRmodels} that, for every commutative ring with involution $A$, there is an isomorphism of rings
\[
\pi_0(\THR(A)^{\Z/2})\cong A^{\Z/2}\otimes_N A^{\Z/2}:=(A^{\Z/2}\otimes A^{\Z/2})/\langle1\otimes a\overline{a}-a\overline{a}\otimes 1\rangle,
\]
where $A^{\Z/2}$ is the subring of invariants of $A$, and the quotient is by the ideal generated by the elements of the form $1\otimes a\overline{a}-a\overline{a}\otimes 1$ for some $a\in A$
(here we use that $a+\overline{a}=(a+1)\overline{(a+1)}-a\overline{a}-1$ to simplify the second relation of \cite[5.2]{THRmodels}, so that in particular $2b\otimes 1=1\otimes 2b$ if $b\in A^{\Z/2}$). The restriction map $\pi_0(\THR(A)^{\Z/2})\to \pi_0\THR(A)$ then corresponds to the multiplication map
\[
A^{\Z/2}\otimes_N A^{\Z/2}\stackrel{\mu}{\longrightarrow}A^{\Z/2}\longrightarrow A,
\]
where the second map is the inclusion. For perfect fields, the map  $\mu$ induces an isomorphism $A^{\Z/2}\otimes_N A^{\Z/2}\cong A^{\Z/2}$, and the same is true in the following cases:

\begin{rem}\label{semiperfect}\
\begin{enumerate}[label=\roman*)]
\item
 Let us start by noticing that for every additive generator $a\otimes b\in A^{\Z/2}\otimes_N A^{\Z/2}$ we have that
\[
2\mu(a\otimes b)\otimes 1=(2ab)\otimes 1=a\otimes 2b=2(a\otimes b),
\]
and therefore all the elements of the kernel of $\mu$ are $2$-torsion (where the second equality follows from the identity above). Thus $\mu$ is an isomorphism when $A$ is $2$-torsion free, for example for fields of odd characteristic.
\item
There is a section for $\mu\colon A^{\Z/2}\otimes_N A^{\Z/2}\to A^{\Z/2}$, that sends $a$ to $a\otimes 1$. Therefore $\mu$ is always surjective, and it is an isomorphism if and only if this section is itself surjective. 
\item
If the multiplication map $A^{\Z/2}\otimes A^{\Z/2}\to A^{\Z/2}$ is an isomorphism, for example for $A=\Z/n$ for any integer $n$, then so is $\mu$.
\item
If the involution of $A$ is trivial and the modulo $2$ reduction of $A$ is semi-perfect (that is the mod $2$ Frobenius is surjective), then every element $a\in A$ can be written as $a=c^2+2d$ for some $c,d\in A$. Then we can write a generator of $A\otimes_N A$ as
\[
a\otimes b=(c^2+2d)\otimes b= 1\otimes (c^2+2d)b=1\otimes ab,
\]
which shows that the  section $A\to A\otimes_N A$ is surjective. This example covers the case of perfect fields of characteristic $2$. 
\item If the involution of $A$ is not trivial, a similar argument shows that the section $A^{\Z/2}\to A^{\Z/2}\otimes_NA^{\Z/2}$ is surjective if every element $a\in A^{\Z/2}$ can be written as $a=c\overline{c}+d+\overline{d}$ for some $c,d\in A$, or in other words if the composite
\[
A\stackrel{N}{\longrightarrow}A^{\Z/2}\twoheadrightarrow A^{\Z/2}/\tran
\]
is surjective, where $N(a)=a\overline{a}$ and $\tran(a)=a+\overline{a}$.
\item Suppose that there exists an element $e\in A$ with the property that $e+w(e)=1$, for example if $2\in A$ is a unit. By Frobenius reciprocity, this is equivalent to the surjectivity of $\tran\colon A\to A^{\Z/2}$, since any element $x\in A^{\Z/2}$ can be written as
\[
x=1\cdot x=\tran(e)\cdot x=\tran (e\res (x))
\]
(explicitly, $x=ex+w(ex)$). Thus this condition is equivalent to the vanishing of $HA^{\phi \Z/2}$. By the previous item $\mu$ is an isomorphism.
\end{enumerate}

An example where the multiplication map is not an isomorphism is provided by the group-ring $\mathbb{Z}[C_2]$ with the trivial involution, where 
\[\mathbb{Z}[C_2]^{\Z/2}\otimes_N \mathbb{Z}[C_2]^{\Z/2}\cong \Z[C_2]\oplus (\Z/2)^{\oplus 2}\]
is not isomorphic to $\Z[C_2]^{\Z/2}$ (see \cite[Section 5.2]{THRmodels}).
\end{rem}

If the multiplication $\mu\colon A^{\Z/2}\otimes_N A^{\Z/2}\to A^{\Z/2}$ is an isomorphism, it follows from \cite[Theorem 5.1]{THRmodels} that the $\Z/2$-Mackey functor $\underline{\pi}_0\THR(A)$ is the fixed-points Mackey functor of the ring with involution $A$.
 On the other hand, if the prime $p$ odd we show in \cite[Theorem 3.7]{Polynomial} that $\pi_0\THR(A)^{D_{p^n}}$ is also a ring of Witt vectors, and combining these results we obtain a ring isomorphism 
\[\pi_0\THR(A)^{D_{p^n}}\cong W_{n+1}(A^{\Z/2}\otimes_N A^{\Z/2};p)\cong W_{n+1}(A;p),\]
for every odd prime $p$ and commutative ring $A$ satisfying any of the assumptions of Remark \ref{semiperfect}.

In the next section we use this last isomorphism to determine TCR of perfect fields of odd characteristic. In the subsequent ones we examine the relationship between $\pi_0\THR(A)^{D_{2^n}}$ and the Witt vectors for the prime $2$, and determine TCR of perfect fields of characteristic $2$.

\subsection{TCR of perfect fields of odd characteristic}

Let $p$ be an odd prime, and $A$ a commutative ring with involution. We let $W(A;p)$ denote the ring of $p$-typical Witt vectors of $A$. 
By Remark \ref{semiperfect} and \cite[Theorem D]{Polynomial} there is an isomorphism of $\Z/2$-Mackey Functors
\[
\underline{\pi}_0\TRR(A;p)\cong \underline{W(A;p)},
\]
between the components of $\TRR(A;p)$ and the the fixed-points Mackey functor of $W(A;p)$ with the involution induced functorially by the involution of $A$. In particular $\pi^{\Z/2}_0\TRR(A;p)\cong W(A;p)^{\Z/2} =  W(A^{\Z/2};p)$, where the latter holds since the $\Z/2$ action is given coordinate-wise and fixed points commute with products. 

\begin{prop} \label{oddcomputation} Let $p$ be an odd prime, and  $k$ a perfect field of characteristic $p$ with involution. Then there are equivalences of genuine $\Z/2$-spectra
\[\TRR(k;p) \simeq H\underline{W(k;p)}\]
and
\[ \TCR(k;p) \simeq H\underline{\Z_p} \oplus \Sigma^{-1}H \underline{\coker(1-F)},\]
where $F\colon W(k;p)\to W(k;p)$ is the Witt vector Frobenius.
\end{prop}

\begin{proof}
The $0$-th Postnikov section provides a map of $\Z/2$-equivariant spectra
\[
\TRR(k;p)\longrightarrow H \underline{W(k;p)}.
\]
This map is an equivalence on underlying spectra by \cite[Theorem 5.5]{Wittvect}, and it is therefore sufficient to prove that it is an equivalence on geometric fixed-points. The spectrum $\TRR(k;p)^{\Z/2}$ has the structure of a ring spectrum. Moreover there is an isomorphism $\pi_0(\TRR(k;p)^{\Z/2})\cong W(k;p)$ and therefore $2=\tran(1)$ is a unit in $\pi_0(\TRR(k;p)^{\Z/2})$, see  \cite[Corollary 3.14]{Polynomial}. Since the transfers vanish in the geometric fixed points, we have that  $2$ is both a unit and zero in $\pi_0(\TRR(k;p)^{\phi \Z/2})$, and therefore $\pi_0(\TRR(k;p)^{\phi \Z/2})$ is the zero ring. Since $\TRR(k;p)^{\phi \Z/2}$ is a ring spectrum its homotopy groups are a module over the zero ring, and therefore it must be contractible.

According to Definition \ref{TCRdefTRR} and the previous paragraph, the $\Z/2$-spectrum $\TCR(k;p)$ is equivalent to the equaliser of $\Z/2$-spectra
\[
eq\big(\xymatrix{H \underline{W(k;p)}\ar@<.5ex>[r]^-\id\ar@<-.5ex>[r]_-F&H \underline{W(k;p)}}\big).
\]
The kernel of $\id-F \colon W(k;p) \to W(k;p)$ is equal to $W(\F_p;p)$ which is isomorphic to $\Z_p$, and this completes the proof.

\end{proof}

\subsection{TCR of perfect fields of characteristic 2}\label{sec:TCRk2}

The calculation of $\TCR(k;2)$ for a perfect field of characteristic $2$ is more involved than the odd primary case. This is because the geometric fixed-points spectrum of $\TRR(k;2)$ is not trivial, and thus we cannot directly apply the argument of Proposition \ref{oddcomputation}. The first step is to understand the geometric fixed-points of $\TRR^n(k;2)$, using the formula of Theorem \ref{inductivePB}.

\subsubsection{The geometric fixed points of \texorpdfstring{$\TRR^n$}{TRRn} for perfect fields of characteristic 2}

Let us fix a perfect field $k$ of characteristic $2$, and let us compute additively ${\TRR^n(k;2)}^{{\phi \Z/2}}$, ${\TRR(k;2)}^{{\phi \Z/2}}$ and ${\TCR(k;2)}^{{\phi \Z/2}}$. We let $\underline{k}$ denote the constant Mackey functor of $k$.

\begin{lemma}\label{geofixgenuine} There is an equivalence of genuine $C_2$-equivariant spectra
\[{\THR(k)}^{{\phi \Z/2}}\simeq \bigoplus_{n \geq 0} \Sigma^{n\rho}  H\underline{k} \oplus \bigoplus_{ \begin{smallmatrix} (n,m) \\ 0 \leq n  < m \end{smallmatrix}} \Sigma^{n+m} {C_2}_{+} \otimes Hk,\]
where $\rho$ is the regular representation of $C_2$. In particular there is an equivalence of spectra
\[
({\THR(k)}^{{\phi \Z/2}})^{C_2}\simeq (\bigoplus_{n \geq 0} \bigoplus_{0 \leq j \leq n} \Sigma^{n+j} Hk) \oplus (\bigoplus_{\begin{smallmatrix} (n,m) \\ 0 \leq n  < m \end{smallmatrix}}  \Sigma^{n+m} Hk).
\]
\end{lemma}

\begin{proof} 
By splitting $H\underline{k}^{{\phi \Z/2}}$ using the Frobenius $Hk$-module structure we obtain a decomposition 
\[H\underline{k}^{{\phi \Z/2}} \simeq \bigoplus_{n \geq 0} \Sigma^{n} Hk.\]
This uses that $k$ is perfect and hence $k$ considered as a $k$-module via the Frobenius is again a $1$-dimensional $k$-vector space. By using inductively that, for every $R$-modules $X$ and $Y$, there is an equivalence of  $N^{C_2}(R)$-modules
\[N^{C_2}(X \oplus Y) \simeq N^{C_2}(X) \oplus N^{C_2}(Y) \oplus ((C_2)_+\otimes X \otimes Y),\]
we find that the $C_2$-norm of $H\underline{k}^{{\phi \Z/2}}$ decomposes as a $N^{C_2}_eHk$-module as
\begin{align*}
N^{C_2}_e(H\underline{k}^{{\phi \Z/2}})&\simeq N^{C_2}_e(\bigoplus_{n \geq 0} \Sigma^{n} Hk)\simeq
(\bigoplus_{n \geq 0}\Sigma^{n\rho} N^{C_2}_e(Hk))\oplus (\bigoplus_{\begin{smallmatrix} (n,m) \\  0\leq n  < m \end{smallmatrix}}\Sigma^{n+m} (C_2)_+\otimes N^{C_2}_e(Hk)).
\end{align*} 
By the description of $\THR(k;2)^{{\phi\Z/2}}$ of Lemma \ref{C2THRphi}, there is an equivalence of $C_2$-spectra
\begin{align*}
\THR(k;2)^{{\phi\Z/2}}&\simeq H\underline{k}\otimes_{N^{C_2}_e(Hk)}(N^{C_2}_e(Hk^{\phi\Z/2}))
\\&
\simeq
(\bigoplus_{n \geq 0}\Sigma^{n\rho} H\underline{k}\otimes_{N^{C_2}_e(Hk)}(N^{C_2}_e(Hk)))\oplus (\bigoplus_{\begin{smallmatrix} (n,m) \\  0\leq  n  < m \end{smallmatrix}}\Sigma^{n+m} (C_2)_+\otimes H\underline{k}\otimes_{N^{C_2}_e(Hk)}(N^{C_2}_e(Hk)))
\\&\simeq
(\bigoplus_{n \geq 0}\Sigma^{n\rho} H\underline{k})\oplus (\bigoplus_{\begin{smallmatrix} (n,m) \\  0\leq n  < m \end{smallmatrix}}\Sigma^{n+m} (C_2)_+\otimes Hk),
\end{align*}
where the last identification is the equivalence
\[
H\underline{k}\stackrel{\simeq}{\longrightarrow}H\underline{k}\otimes_{N^{C_2}_e(Hk)}(N^{C_2}_e(Hk))
\]
given by tensoring with the unit of the norm of $Hk$ (and similarly for the identification on the induced summands).

Now let us identify the fixed-points. Notice that $\Sigma^{n\rho}  H\underline{k}$ is a module over $H\underline{k}$ and therefore its fixed-points spectrum is a wedge of Eilenberg-MacLane spectra. Moreover a straightforward calculation in Bredon homology shows that 
\[\pi_i^{C_2}(\Sigma^{n\rho}  H\underline{k}) =H^{C_2}_i(S^{n\rho}; \underline{k}) \cong k\]
when $n \leq i \leq 2n$, and $\pi_i^{C_2}(\Sigma^{n\rho}  H\underline{k})=0$ otherwise.
 \end{proof}
 
In the following proposition the summands are arranged exactly as in Lemma \ref{geofixgenuine}. In particular, the summands indexed on $(n,m)$ with $n<m$ in the source come from the induced summands. Similarly, the summands indexed on $(n,m)$ with $n \neq m$ in the target corresponds to the induced summands.

\begin{prop} \label{lowerFR} For any perfect field $k$ of characteristic $2$, the maps $r,f \colon {({\THR(k)}^{{\phi \Z/2}})}^{C_2} \to {\THR(k)}^{{\phi \Z/2}}$ induce on $\pi_\ast$ the maps
\[ r,f\colon \bigoplus_{\begin{smallmatrix}(n,m)\\
n, m \geq 0, n+m=\ast\end{smallmatrix}}k  \to \bigoplus_{\begin{smallmatrix}(n,m)\\
n, m \geq 0, n+m=\ast\end{smallmatrix}} k,\]
where $r$ kills the $(n,m)$-summands with $n<m$ and maps the $(n,m)$-summands with $m \leq n$ to the $(n,m)$-summand via the inverse Frobenius of $k$, and $f$ kills the $(n,m)$-summands with $m<n$, includes the summand $(n,n)$, and embeds diagonally the $(n,m)$-summands with $n<m$ into the sum of the summands $(n,m)$ and $(m,n)$.
\end{prop}

\begin{proof} Recall that the map $r$ in general is given by
\[r\colon (T^{\phi\Z/2})^{C_2}\longrightarrow  (T^{\phi\Z/2})^{\phi C_2}\simeq T^{\phi\Z/2},\]
where first map is the canonical map from the genuine to the geometric fixed-points and the equivalence is given by the cyclotomic structure. If we take $T=\THR(k)$, then the last equivalence is described in Example \ref{minicycloTHR}. As genuine and geometric fixed-points commute with coproducts we need to identify $r$ on each of the summands of Lemma \ref{geofixgenuine}. Since geometric fixed-points kill induced spectra, $r$ vanishes on the summands $(n,m)$ with $n<m$. The identification on the other summands follows from observing that the canonical map
\[\pi_*^{C_2}(\Sigma^{n\rho}  H\underline{k}) \to \pi_*((\Sigma^{n\rho}  H\underline{k})^{\phi C_2})\cong\pi_*(\Sigma^n (H\underline{k})^{\phi C_2})\cong\pi_*(\bigoplus_{l\geq 0}\Sigma^{n+l}H\underline{k}) \]
induces the inverse Frobenius of $k$ in degrees $n \leq \ast \leq 2n$ (cf. with \cite[Example IV.1.2]{NS}).
Similarly, we need to compute $f$ on each summand, and its identification follows from the fact that the restriction map
\[\res^{C_2}_e \colon H^{C_2}_*(S^{n\rho}; \underline{k}) \to H_*(S^{2n}; k) \]
is the identity only in degree $*=2n$ and zero otherwise, whereas
\[\res^{C_2}_e \colon \pi_*^{C_2} (\Sigma^{n+m} {C_2}_{+} \otimes H\underline{k})\cong  \pi_* (\Sigma^{n+m}  H\underline{k})  \to  \pi_* (\Sigma^{n+m}  H\underline{k})\oplus \pi_* (\Sigma^{n+m}  H\underline{k})\cong \pi_*^e (\Sigma^{n+m} {C_2}_{+} \otimes H\underline{k})  \]
is the diagonal for all $n$ and $m$. 
\end{proof}

\begin{rem}\label{rem:geomTCRk}
From Proposition \ref{lowerFR} and Theorem \ref{tcrformula1} we obtain that 
\[
\pi_\ast \TCR(k;2)^{\phi\Z/2}=\left\{\begin{array}{ll}\F_2&\ast=2l\geq 0\\ k/(x^2+x)&\ast=2l-1\geq -1\\
0&\ast\leq -2.\end{array}\right.
\]
Indeed by \ref{lowerFR}, the map $r-f$ is an isomorphism in $\pi_\ast$ when restricted and corestricted to the summands with $n\neq m$. It is therefore an isomorphism in odd degrees, and its long exact sequence decomposes into exact sequences
\[
0\to\pi_{2l} \TCR(k;2)^{\phi\Z/2}\to  \bigoplus_{\begin{smallmatrix}(n, m )\\
n, m \geq 0,n+m=2l\end{smallmatrix}} k\xrightarrow{r-f} \bigoplus_{\begin{smallmatrix}(n,m)\\
n, m \geq 0,n+m=2l\end{smallmatrix}} k\to \pi_{2l-1} \TCR(k;2)^{\phi\Z/2}\to 0
\] 
for every $l\geq 0$. Again by \ref{lowerFR}, the kernel of $r-f$ is the kernel of $\id-\sqrt{(-)}\colon k\to k$. Since $k$ is a field this is $\F_2$. Similarly the cokernel of  $r-f$ is the cokernel of $\id-\sqrt{(-)}$, which since $k$ is perfect it is also the cokernel of $\id+(-)^2$.

We also remark that these groups agree with the homotopy groups of the cofibre $\Lt^n(k)$ of the canonical map
\[
\Lt^q(k)\longrightarrow\Lt(\Mod^\omega_A,\text{\Qoppa}^{gs}_k).
\]
induced by the symmetrisation map from the quadratic to the genuine Poincar\'e structure, as defined in \cite{9I,9II,9III}. Indeed the even homotopy groups of $\Lt^q(k)$ are the Witt groups of quadratic forms over $k$, and since $k$ is a field the odd groups vanish \cite[Proposition 22.7]{TopMan}. The map above is an isomorphism in degrees lass than or equal to $-3$ and surjective in degree $-2$ by \cite[Theorem 5]{9III}, and therefore the cofibre $\Lt^n(k)$ is $-1$-connected. In degrees greater or equal to $-1$ the homotopy groups of the target are the symmetric Witt groups of $k$ in even degrees and zero in odd degrees, by \cite[Corollary 1.3.5]{9III}. The map is the symmetrisation map from quadratic to symmetric Witt groups, which is zero since $k$ has characteristic $2$. Thus the homotopy groups of $\Lt^n(k)$ are the symmetric Witt groups of $k$ in even non-negative degrees, and the quadratic ones in odd non-negative degrees. The $(-1)$-st homotopy group of $\Lt^n(k)$ is the kernel of the symmetrisation map, and therefore again the quadratic Witt group. The quadratic and symmetric Witt groups of a perfect field of characteristic $2$ are respectively $k/(x+x^2)$ and $\F_2$, see e.g.  \cite[Theorem (1)]{Kato}.
\end{rem}

In order to understand the full equivariant homotopy type of $\TCR(k;2)$ will need to calculate the homotopy groups of $\TRR(k;2)^{{\phi \Z/2}}$.

\begin{theorem} \label{geometricofTRnF2} Let $k$ be a perfect field of characteristic $2$. For any $l \geq 1$, there is an isomorphism
\[\pi_\ast {\TRR^l(k;2)}^{{\phi \Z/2}} 
\cong \bigoplus_{\begin{smallmatrix}(n,m)\\
n, m \geq 0,n+m=\ast\end{smallmatrix}}k.
\]
The maps $R,F \colon {\TRR^{l+1}(k;2)}^{{\phi \Z/2}}  \to {\TRR^{l}(k;2)}^{{\phi \Z/2}}$ and the Weyl action are described on homotopy groups as follows. The map $R$ kills the $(n, m)$-summands with $n \neq m$ and is the inverse Frobenius of $k$ on the summands $(n,n)$. The map $F$ kills the $(n, m)$-summands with $m<n$, is the identity of $k$ on the summands $(n,n)$, and embeds the $(n,m)$-summands with $n<m$ diagonally into the sum of the $(n,m)$ and $(m,n)$-summands. The Weyl action of $\sigma_{l}$ swaps the $(n,m)$-summand and the $(m,n)$-summand for all $n,m \geq 0$. In particular the homotopy groups and the maps are all independent of $l$.
\end{theorem}

\begin{proof} We prove the theorem by induction on $l$, using the pullbacks of Theorem \ref{inductivePB}. For $n=1$ the pullback of \ref{inductivePB} implies that $ {\TRR^2(k;2)}^{{\phi \Z/2}}$ is equivalent to the pullback $(T^{\phi})^{C_2}{\times_{T^{\phi}}}
(T^{\phi})^{C_2}$ (since the right vertical map is the diagonal for $n=1$ in \ref{inductivePB}). Consider the  Mayer-Vietoris sequence associated to $(T^{\phi})^{C_2}{\times_{T^{\phi}}}(T^{\phi})^{C_2}$:
\[
\dots\stackrel{\partial}{\longrightarrow}\pi_\ast {\TRR^2(k;2)}^{{\phi \Z/2}} \stackrel{}{\longrightarrow}(\!\!\!\!\bigoplus_{\begin{smallmatrix}(n, m) \\
n, m \geq 0, n+m=\ast\end{smallmatrix}}\!\!\!\!\!\!\!\! k)\oplus (\!\!\!\!\bigoplus_{\begin{smallmatrix}(n, m) \\
n, m \geq 0, n+m=\ast\end{smallmatrix}}\!\!\!\!\!\!\!\!k)
\stackrel{r-\sigma_1 r}{\longrightarrow}
\!\!\!\!\bigoplus_{\begin{smallmatrix}(n, m)\\
n, m \geq 0, n+m=\ast\end{smallmatrix}}\!\!\!\!\!\!\!\!k\stackrel{\partial}{\longrightarrow}\dots
\]
where $r$ is determined in Proposition \ref{lowerFR}. Since $r-\sigma_1 r$ is clearly surjective on each homotopy group, the Mayer-Vietoris sequence decomposes into short exact sequences and $\pi_\ast {\TRR^2(k;2)}^{{\phi \Z/2}}$ is the kernel of $r-\sigma_1r$. This kernel consists of the pairs of finite sequences $(x,y)$ indexed on the pairs of non-negative integers $(n,m)$, such that
\[
\begin{array}{rl}
0&=r(x)_{(n,m)}=(\sigma_1 r(y))_{(n,m)}=r(y)_{(m,n)}=\sqrt{y_{(m,n)}}, \;\; \text{for} \; n<m,
\\
\sqrt{x_{(n,m)}}&=r(x)_{(n,m)}=(\sigma_1 r(y))_{(n,m))}=r(y)_{(m,n)}=0, \;\; \text{for} \; n>m,
\\
\sqrt{x_{(n,n)}}&=r(x)_{(n,n)}=(\sigma_1 r(y))_{(n,n)}=r(y)_{(n,n)}=\sqrt{y_{(n,n)}}
\end{array}
\]
where $\sqrt{(-)}$ denotes the inverse of the Frobenius $(-)^2\colon k\to k$.
These are the pairs $(x,y)$ where $x_{(n,m)}=0$ and $y_{(n,m)}=0$ for $n>m$, and $x_{(n,n)}=y_{(n,n)}$, which gives the description of the homotopy groups of $\TRR^2(k;2)^{\phi \Z/2}$. The maps $R,F\colon \TRR^2(k;2)^{\phi \Z/2}\to \THR(k;2)^{{\phi \Z/2}}$ are described in \ref{inductivePB} and send such a pair $(x,y)$ to $r(x)$ and $f(x)$ respectively, and are therefore the  maps of Theorem \ref{geometricofTRnF2}. The Weyl action flips $x$ and $y$ by  \ref{inductivePB}.

Now let $l \geq 2$ and suppose inductively that the decomposition holds for $\pi_\ast {\TRR^h(k;2)}^{{\phi \Z/2}}$ for all $h \leq l$ and that the maps $R,F\colon {\TRR^{h}(k;2)}^{{\phi \Z/2}}\to {\TRR^{h-1}(k;2)}^{{\phi \Z/2}}$ and $\sigma_h$ are given in homotopy groups by the formulas of  \ref{geometricofTRnF2}. We will show that the same holds for $\pi_\ast {\TRR^{l+1}(k;2)}^{{\phi \Z/2}}$ and the maps $R,F\colon {\TRR^{l+1}(k;2)}^{{\phi \Z/2}}\to {\TRR^{l}(k;2)}^{{\phi \Z/2}}$ and $\sigma_{l+1}$. The Mayer-Vietoris sequence of the pullback square of  \ref{inductivePB} is then (we recall $\sigma_1F =F$, and that $n,m \geq 0$)
\[
\xymatrix@C=15pt@R=15pt{
\dots_{\ }\ar[d]^-{\partial}
\\
\pi_\ast {\TRR^{l+1}(k;2)}^{{\phi \Z/2}} \ar[r]
&
\displaystyle\big((\bigoplus_{\begin{smallmatrix}(n,m)\\
n+m=\ast\end{smallmatrix}}k)
\oplus
(\bigoplus_{\begin{smallmatrix}(n,m)\\
n+m=\ast\end{smallmatrix}}k)
\big)
\oplus
(\bigoplus_{\begin{smallmatrix}(n,m)\\
n+m=\ast\end{smallmatrix}}k)
\ar[rrrr]^-{r\oplus \sigma_1 r-(F^{l-1},F^{l-1}\sigma_l)}
&&&&
\displaystyle(\bigoplus_{\begin{smallmatrix}(n,m)\\
n+m=\ast\end{smallmatrix}}k)
\oplus
(\bigoplus_{\begin{smallmatrix}(n,m)\\
n+m=\ast\end{smallmatrix}}k)\ar[d]^-\partial
\\
&&&&&\dots
}
\]
By the inductive assumption the iterated map $F^{l-1}$ is in fact equal to a single map $F$. The right horizontal map then sends a triple $(x,y,z)$ of finite sequences indexed on the pairs of integers $n,m\geq 0$ to the pair of sequences
\[
(r(x)-F(z))_{(n,m)}=\left\{
\begin{array}{ll}
\sqrt{x_{(n,m)}}-z_{(m,n)}&, n>m
\\
-z_{(n,m)}&, n<m
\\
\sqrt{x_{(n,n)}}-z_{(n,n)}&, n=m
\end{array}
\right.
\]
\[
(\sigma_1r(y)-F\sigma_l(z))_{(n,m)}=\left\{
\begin{array}{ll}
-z_{(n,m)}&, n>m
\\
\sqrt{y_{(m,n)}}-z_{(m,n)}&, n<m
\\
\sqrt{y_{(n,n)}}-z_{(n,n)}&, n=m .
\end{array}
\right.
\]
This map is clearly surjective for all $\ast$, and therefore $\pi_\ast {\TRR^{l+1}(k;2)}^{{\phi \Z/2}}$ is isomorphic to its kernel. This consists of those triples $(x,y,z)$ such that $x_{(n,m)}=y_{(n,m)}=0$ for $n>m$ and $z_{(n,m)}=0$, for $n \neq m$, and $\sqrt{x_{(n,n)}}=\sqrt{y_{(n,n)}}=z_{(n,n)}$, which is isomorphic to the direct sum on all pairs of natural numbers by setting $w_{(n,m)}=y_{(m,n)}$ for $n>m$, and $w_{(n,m)}=x_{(n,m)}$, for $n<m$, and $w_{(n,n)}=x_{(n,n)}$. Let us now describe $R$ and $F$ under these isomorphisms.
By  \ref{inductivePB} the map $R$ sends $(x,y,z)$ to $z$, and therefore under the isomorphism above 
\[
R(w)_{(n,m)}=\left\{
\begin{array}{ll}
0&, n\neq m
\\
\sqrt{w_{(n,n)}}&, n=m.
\end{array}
\right.
\]
Again by  \ref{inductivePB} the map $F$ sends $(x,y,z)$ to $(x,x,F(z))$. Thus under the identification above
\[
F(w)_{(n,m)}=
\left\{
\begin{array}{ll}
x_{(m,n)}=w_{(m,n)}&, n>m
\\
x_{(n,m)}=w_{(n,m)}&, n<m
\\
x_{(n,n)}=w_{(n,n)}&, n=m.
\end{array}
\right.
\]
Finally, the Weyl action $\sigma_{l+1}$ sends $(x,y,z)$ to $(y,x,\sigma_l(z))$, and under the isomorphism above
$
\sigma_{l+1}(w)_{(n,m)}=w_{(m,n)}
$.
\end{proof}

\begin{cor} \label{geometricofTRRF2} Let $k$ be a perfect field of characteristic $2$. There is a natural isomorphism
\[\pi_\ast {\TRR(k;2)}^{{\phi \Z/2}} \cong \left\{\begin{array}{ll}k & \mbox{if $\ast$ is even}
\\
0& \mbox{otherwise},
 \end{array}\right.\]
and the Frobenius endomorphism $F\colon {\TRR(k;2)}^{{\phi \Z/2}}\to {\TRR(k;2)}^{{\phi \Z/2}}$ is the Frobenius of $k$ on homotopy groups.
\end{cor}

\begin{proof}
By Theorem \ref{geometricofTRnF2} the map $R$ on homotopy groups is the map 
\[\bigoplus_{\begin{smallmatrix}(n,m)\\
n+m=\ast\end{smallmatrix}}k\to\bigoplus_{\begin{smallmatrix}(n,m)\\
n+m=\ast\end{smallmatrix}}k\] 
(where $n,m \geq 0$) which is the inverse Frobenius on the summands $(n,n)$, and zero everywhere else. It is an idempotent up to isomorphism, and therefore it satisfies the Mittag-Leffler condition. It follows that
\[
\pi_\ast \TRR(k;2)^{\phi \Z/2}\cong \lim_{R}\pi_\ast\TRR^l(k;2)^{\phi \Z/2}\cong  \lim_{R}\bigoplus_{\begin{smallmatrix}(n,m)\\
n+m=\ast\end{smallmatrix}}k\cong\bigoplus_{2n=\ast}k,
\]
where the last isomorphism is induced by the projection onto the first component of the limit and onto the summand $2n=\ast$ when $\ast$ is even, and it is zero otherwise. After composing with the shift automorphism of the limit, $R$ becomes by definition the identity, and $F$ the Frobenius of $k$.
\end{proof}

\subsubsection{The components of TRR and the ring of Witt vectors of perfect fields}

As we will show in Remark \ref{WittZ} below, the ring $\pi_0\THR(A)^{D_{2^n}}$ is not necessarily the ring of Witt vectors of $\pi_0\THR(A)^{\Z/2}$, not even when the latter is isomorphic to $A$. However, this is still the case for perfect fields, as we show now.

\begin{theorem}\label{pi0TRRp2}
Let $k$ be a perfect field of characteristic $2$, equipped with the trivial involution.
Then for every $n\geq 0$, the restriction map
\[\res^{D_{2^n}}_{C_{2^n}}\colon \pi_0\THR(k)^{D_{2^n}}\longrightarrow \pi_0\THH(k)^{C_{2^n}}\cong W_{n+1}(k;2)\]
is an isomorphism, and the Verschiebung, Frobenius, and restriction maps of the Witt vectors correspond respectively to $\tran_{D_{2^{n-1}}}^{D_{2^{n}}}$, $\res_{D_{2^{n-1}}}^{D_{2^{n}}}$, and $R$.
%
\end{theorem}

\begin{proof} Let us start with a commutative ring with involution $A$, and follow the strategy of  \cite{Wittvect} and \cite{Polynomial} of analysing the long exact sequence induced on homotopy groups by the fibre sequence
\[
E\mathcal{R}_+\otimes_{C_{2^{n}}} \THR(A)\longrightarrow \THR(A)^{C_{2^{n}}}\longrightarrow \THR(A)^{C_{2^{n-1}}}.
\]
The components of the fixed-points of the fibre are then calculated by the colimit
\[
\pi_0(E\mathcal{R}_+\otimes_{C_{2^{n}}} \THR(A))^{\Z/2}
\cong \colim_{\mathcal{O}_{\mathcal {R}}} \underline{\pi}^{(-)}_0\THR(A),
\]
where $\mathcal{O}_{\mathcal{R}}$ is the full subcategory of the orbit category of $D_{2^n}$ generated by the reflections and the trivial group (this follows for example from the fact that $E\mathcal{R}$ is the homotopy colimit over $\mathcal{O}_{\mathcal{R}}$ of the funtor that takes $D_{2^n}/H$ to the discrete space $D_{2^n}/H$, see e.g. \cite[Lemma 2.2]{LO}). The crucial difference between $2$ and the odd primes is that for the prime $2 $ the category $\mathcal{O}_{\mathcal{R}}$ has two components, generated by the distinct conjugacy classes of the reflections $\tau$ and $\sigma\tau$, where $\sigma$ is the generator of the cyclic group $C_{2^n}$. Therefore the colimit above is isomorphic to the pushout of abelian groups
\begin{align*}
\pi_0(E\mathcal{R}_+\otimes_{C_{2^{n}}} \THR(A))^{\Z/2}&\cong  \colim_{\mathcal{O}_{\mathcal{R}}} \underline{\pi}^{(-)}_0\THR(A)\cong (\pi_0\THR(A)^{\Z/2})_{C_2}\oplus_{A}  (\pi_0\THR(A)^{\Z/2})_{C_2}
\end{align*}
along the transfer maps $\tran^{\Z/2}_e\colon A\to (\pi_0\THR(A)^{\Z/2})_{C_2}$, where the coinvariants are taken with respect to the action of the Weyl group $C_2$.
Under this identification, the transfer map to $\pi_0\THR(A)^{D_{2^n}}$ is the transfer $\tran_{\Z/2}^{D_{2^{n}}}$ on the first summand, and $\sigma_{n}\tran_{\Z/2}^{D_{2^{n}}}$ on the second summand, where $\sigma_{n}$ is the action of the generator of the Weyl group. The corresponding long exact sequence on homotopy groups is then
\[
\xymatrix@C=20pt{
\dots \ar[d]^-{\partial}
\\
 (\pi_0\THR(A)^{\Z/2})_{C_2}\oplus_{A}  (\pi_0\THR(A)^{\Z/2})_{C_2}
\ar[rrr]^-{\tran_{\Z/2}^{D_{2^{n}}}+\sigma_n\tran_{\Z/2}^{D_{2^{n}}}}&&&\pi_0\THR(A)^{D_{2^n}}\ar[r]^-R&\pi_0\THR(A)^{D_{2^{n-1}}}\ar[d]
\\&&&&0\rlap{.}
}
\]
Let us now compute the boundary map of this sequence in the case where $k$ is a perfect field of characteristic $2$ with the trivial involution. Since $k$ is a field of characteristic $2$, the isomorphism of  \cite[Corollary 5.2]{THRmodels} is
\[
 \pi_0\THR(k)^{\Z/2}\cong k\otimes_S k
\]
where $S\subset k$ is the subfield of squares. Moreover since $k$ is perfect, the restriction map
\[
 \pi_0\THR(k)^{\Z/2}\cong k\otimes_S k\longrightarrow k\cong  \pi_0\THH(k),
\]
which is induced by the multiplication map of $k$, is an isomorphism, and the Weyl action on the source must be trivial. The transfer map 
\[
\xymatrix{ k\cong\pi_0\THH(k)\ar[r]^-{\tran^{\Z/2}_e}&\pi_0\THR(k)^{\Z/2}\ar[r]^-{\res}_-{\cong}&k}
\]
is multiplication by $2$ by the double-coset formula, and therefore zero. Thus the long exact sequence above becomes
\[
\xymatrix@C=20pt{
\dots \ar[r]^-{\partial}
&
k\oplus k
\ar[rrr]^-{\tran_{\Z/2}^{D_{2^{n}}}+\sigma_n\tran_{\Z/2}^{D_{2^{n}}}}&&&\pi_0\THR(k)^{D_{2^n}}\ar[r]^-R&\pi_0\THR(k)^{D_{2^{n-1}}}\ar[r]
&0\rlap{.}
}
\]
Now suppose inductively that the restriction map $\res\colon \pi_0\THR(k)^{D_{2^{n-1}}}\to \pi_0\THR(k)^{C_{2^{n-1}}}$ is an isomorphism, and identify the target with $W_{n}(k;2)$ by the isomorphism of \cite[Theorem F]{Wittvect}. Then the restriction map defines a commutative diagram with exact rows
\[
\xymatrix@C=20pt{
\dots \ar[r]^-{\partial}
&
k\oplus k\ar[d]^{(1,1)}
\ar[rrr]^-{\tran_{\Z/2}^{D_{2^{n}}}+\sigma_n\tran_{\Z/2}^{D_{2^{n}}}}&&&\pi_0\THR(k)^{D_{2^n}}\ar[r]^-R\ar[d]^-{\res}&\pi_0\THR(k)^{D_{2^{n-1}}}\ar[r]\ar[d]^-{\res}_{\cong}
&0
\\
\dots \ar[r]^-{0}
&
 k
\ar[rrr]^-{V}&&&W_{n+1}(k;2)\ar[r]^-R&W_{n}(k;2)\ar[r]
&0
}
\]
 where the right vertical map is an isomorphism. It is therefore sufficient to show that the image of $\partial$ is equal to the kernel of $(1,1)$, which is the diagonal $\Delta\subset k\oplus k$. Since the connecting homomorphism of the bottom sequence is zero, we know at least that the image of $\partial$ is included in  $\Delta$, and that the middle restriction map $\res$ is surjective.
 In order to understand the image of $\partial$, we map the sequence above to the corresponding sequence on geometric fixed-points of  Proposition \ref{Frobonfibres}. Since $k$ is perfect, this last sequence is determined by Theorem \ref{geometricofTRnF2}, giving a diagram with exact rows
\[
\xymatrix@C=17pt{
\dots\ar[r]&\pi_1 \THR(k)^{D_{2^{n-1}}}\ar[d]^\phi\ar[r]^-{\partial}
&
k\oplus k\ar[d]^{\id}
\ar[rrr]^-{\tran_{\Z/2}^{D_{2^{n}}}+\sigma_n\tran_{\Z/2}^{D_{2^{n}}}}&&&\pi_0\THR(k)^{D_{2^n}}\ar[r]^-R\ar[d]&\pi_0\THR(k)^{D_{2^{n-1}}}\ar[r]\ar[d]^{\cong}
&0
\\
\dots \ar[r]_0^R&k\oplus k\ar[r]
&
k\oplus k
\ar[rrr]^-{V}_-0&&&k\ar[r]_{\cong}^-R&k\ar[r]
&0\rlap{.}
}
\]
Since the map below $\partial$ must be an isomorphism, the image of $\partial$ is isomorphic to the image of the vertical map $\phi$. The isotropy separation sequence for the $\Z/2$-spectrum $\THR(A)^{C_{2^{n-1}}}$ gives a long exact sequence
\[
\xymatrix@C=20pt{
\pi_1\THR(k)^{D_{2^{n-1}}}\ar[r]^-\phi
&
\pi_1(\THR(k)^{C_{2^{n-1}}})^{\phi\Z/2}\ar[r]
&
\pi_0\THR(k)^{C_{2^{n-1}}}\ar[rr]^{\tran_{C_{2^{n-1}}}^{D_{2^{n-1}}}}
&&
\pi_0\THR(k)^{D_{2^{n-1}}}
.}
\]
By the inductive assumption the $\Z/2$-action on $\pi_0\THR(k)^{C_{2^{n-1}}}$ is trivial (since $\res$ is surjective) and the transfer $\tran_{C_{2^{n-1}}}^{D_{2^{n-1}}}$ identifies with the multiplication by $2$ map on the Witt vectors $W_{n}(k;2)$, whose kernel is $k$. It follows that the cokernel of $\phi$, and therefore that of $\partial$, is isomorphic to $k$. Thus, if $k$ is finite, the image of $\partial$ must have as many elements as $k$ does, and therefore since it is included in $\Delta$ it must be equal to it. This concludes the proof in the case where $k$ is finite. 

If $k$ is infinite, we cannot yet conclude that the image of $\partial$ is the whole diagonal, since we only know that $(k\oplus k)/\im \partial$ is abstractly isomorphic to $k$. We do however know that the image of $\partial$ is $\Delta$ for the finite field $\mathbb{F}_2$, and the naturality of $\partial$ with respect to the morphism of fields $\F_2\to k$ shows that at least $(1,1)$ must belong to $\im \partial$. Since $R\colon\THR(k)^{D_{2^{n}}}\to \THR(k)^{D_{2^{n-1}}}$ is a map of ring spectra, $\partial$ is a map of $\pi_0\THR(k)^{D_{2^{n}}}$-modules. Moreover the isomorphism 
\[
\pi_0(E\mathcal{R}_+\otimes_{C_{2^{n}}} \THR(k))^{\Z/2}\cong  \colim_{\mathcal{O}_{\mathcal{R}}} \underline{\pi}^{(-)}_0\THR(k)\cong \pi_0\THR(k)^{\Z/2}\oplus_{k}  \pi_0\THR(k)^{\Z/2}\cong k\oplus k
\]
is an isomorphism of $\pi_0\THR(k)^{D_{2^{n}}}$-modules, where $\pi_0\THR(k)^{D_{2^{n}}}$ acts on each $\underline{\pi}^{H}_0\THR(k)$ via the restriction map, and the transfers are linear over these restrictions by the Frobenius reciprocity formula of the $D_{2^{n}}$-Mackey functor $\underline{\pi}_0\THR(k)$. In particular $\pi_0\THR(k)^{D_{2^{n}}}$ acts diagonally on $k\oplus k$, via the restriction map $\res^{D_{2^n}}_e\colon \pi_0\THR(k)^{D_{2^{n}}}\to \pi_0\THH(k)=k$. This map factors as
\[
\pi_0\THR(k)^{D_{2^{n}}}\xrightarrow{\res}\pi_0\THR(k)^{C_{2^{n}}}\cong W_{n-1}(k;2)\xrightarrow{F^{n-2}}k
\]
where the first map is surjective by the argument above. Since $k$ is of characteristic $2$, the iterated Frobenius  is given by
\[F^{n-2}(a_1,\dots, a_{n-1})=a^{2^{n-2}}_1,\]
which is surjective since $k$ is perfect.
Thus given any $x\in k$, we can choose an element $z$ of $\pi_0\THR(k)^{D_{2^{n}}}$ which maps to $x$ by the restriction $\res^{D_{2^n}}_e$. Then since the image of $\partial$ is a submodule of $k\oplus k$ containing $(1,1)$, we have that $(x,x)=z\cdot (1,1)$ is also in the image of $\partial$, and thus $\im\partial=\Delta$.

The identification of the Witt vectors operators $V$, $F$ and $R$ follows from the commutative diagrams
\[
\xymatrix@C=40pt{
\pi_0\THR(k)^{D_{2^{n-1}}}\ar[r]^-{\tran_{D_{2^{n-1}}}^{D_{2^{n}}}}\ar[d]^-\res_{\cong}
&
\pi_0\THR(k)^{D_{2^{n}}}\ar[d]^-\res_{\cong}
\\
\pi_0\THR(k)^{C_{2^{n-1}}}\ar[r]^-{\tran_{C_{2^{n-1}}}^{C_{2^{n}}}}
&
\pi_0\THR(k)^{C_{2^{n}}}
}
\ \ \ \ \ \ \ \ \ \ \ \
\xymatrix@C=40pt{
\pi_0\THR(k)^{D_{2^{n}}}\ar[r]^-{\res_{D_{2^{n-1}}}^{D_{2^{n}}}}\ar[d]^-\res_{\cong}
&
\pi_0\THR(k)^{D_{2^{n-1}}}\ar[d]^-\res_{\cong}
\\
\pi_0\THR(k)^{C_{2^{n}}}\ar[r]^-{\res_{C_{2^{n-1}}}^{C_{2^{n}}}}
&
\pi_0\THR(k)^{C_{2^{n-1}}}
}
\]
\[
\xymatrix@C=40pt{
\pi_0\THR(k)^{D_{2^{n}}}\ar[r]^-{R}\ar[d]^-\res_{\cong}
&
\pi_0\THR(k)^{D_{2^{n-1}}}\ar[d]^-\res_{\cong}
\\
\pi_0\THR(k)^{C_{2^{n}}}\ar[r]^-{R}
&
\pi_0\THR(k)^{C_{2^{n-1}}}
}
\]
and the fact that the maps of the bottom row correspond respectively to $V, F$ and $R$ by \cite[Theorem F]{Wittvect}. Note that to show that the first diagram commutes one needs to use the double-coset formula and the fact that the quotient $D_{2^{n-1}}\backslash D_{2^{n}}/C_{2^n}$ is trivial.
\end{proof}

\begin{rem}\label{WittZ}
 The restriction map of Theorem \ref{pi0TRRp2} is not generally an isomorphism. For example for the ring of integers, there is a map of short exact sequences
\[
\xymatrix@C=40pt{
0\ar[r]&\Z\oplus_2 \Z\ar[r]\ar[d]& \pi_0\THR(\Z)^{D_2}\ar[r]^-R\ar[d]^{\res}&\pi_0\THR(\Z)^{\Z/2}\ar[d]^{\res}_\cong\ar[r]&0
\\
0\ar[r]&\Z\ar[r]^-V&W_2(\Z;2)\ar[r]&\Z\ar[r]&0
}
\]
where $\Z\oplus_2 \Z$ is the pushout of the transfer $2\colon \Z\to \Z$ along itself, which is isomorphic to $\Z\times \Z/2$,
and the left-hand map is the identity on each summand. Thus the middle restriction is not an isomorphism, and moreover $\pi_0\THR(\Z)^{D_2}$ has $2$-torsion.

The top row of the diagram comes from the long exact sequence on homotopy groups for the map $R$ of the proof of Theorem \ref{pi0TRRp2}, upon showing that its connective homomorphism $\partial$ is in this case zero. To see this, we map the sequence to the analogous sequence for $\F_2$ via the canonical quotient map $\Z\to \F_2$, and obtain a commutative diagram with exact rows
\[
\xymatrix@C=40pt{
\pi_1\THR(\Z)^{\Z/2}\ar[d]\ar[r]^-{\partial}& \Z\oplus_2 \Z\ar[r]\ar[d]& \pi_0\THR(\Z)^{D_2}\ar[r]^-R\ar[d]&\pi_0\THR(\Z)^{\Z/2}\ar[d]\ar[r]&0
\\
\pi_1\THR(\F_2)^{\Z/2}\ar[r]^-{\partial}&\F_2\oplus \F_2\ar[r]^-V&W_2(\F_2;2)\ar[r]&\F_2\ar[r]&0
}
\]
where the second vertical map from the left is induced by the projection on each summand.
 Thus if we can show that the left vertical map is zero, we will have that the upper $\partial$ maps into the kernel of the projection $ \Z\oplus_2 \Z\to \F_2\oplus\F_2$, which is the subgroup of elements $(2n,0)$, and isomorphic to $\Z$.
The group $\pi_1\THR(\Z)^{\Z/2}$ is however isomorphic to $\Z/2$ by \cite[Proposition 5.22]{THRmodels}, and therefore $\partial$ is $0$. 

We still need to verify that the left vertical map is zero. We look at its effect on the isotropy separation sequences for the $\Z/2$-spectrum $\THR$, and obtain a diagram
\[
\xymatrix@C=40pt{
\pi_1\THR(\Z)^{\Z/2}\ar[d]\ar[r]^{\phi}&\pi_1\THR(\Z)^{\phi\Z/2}\ar[r]\ar[d]^0&\Z\ar[r]^2\ar[d]&\Z\ar[d]
\\
\pi_1\THR(\F_2)^{\Z/2}\ar[r]^{\overline{\phi}}&\pi_1\THR(\F_2)^{\phi\Z/2}\ar[r]&\F_2\ar[r]^0&\F_2\rlap{\.}
}
\]
The map $\pi_1\THR(\Z)^{\phi\Z/2} \to \pi_1\THR(\F_2)^{\phi\Z/2}$ is equal to $0$ by the calculation in the proof of \cite[Theorem 5.20]{THRmodels}.
In the bottom row the lower left map $\overline{\phi}$ is injective. This follows by the last part of the proof of \cite[Theorem 5.15]{THRmodels}, where this map is explicitly identified. Hence we conclude that the left vertical map is zero.

One can in fact show that the connecting homomorphism is zero also for the larger dihedral groups, by calculating the first part of the long exact sequence for $R$ on geometric fixed points using the calculations of section \ref{sec:TCRZ}. One then obtains short exact sequences
\[
\xymatrix@C=40pt{
0\ar[r]&\Z\oplus_2 \Z\ar[r]& \pi_0\THR(\Z)^{D_{2^{n+1}}}\ar[r]^-R& \pi_0\THR(\Z)^{D_{2^{n}}}\ar[r]&0
}
\]
for every $n\geq 1$. We will address this in future work.
\end{rem}

\begin{prop} \label{pi0TRRWittformula}
Let $k$ be a perfect field of characteristic $2$. The tower of abelian groups
\[
\dots\to \pi_1\THR(k)^{D_{2^n}}\stackrel{R}{\longrightarrow} \pi_1\THR(k)^{D_{2^{n-1}}}\stackrel{R}{\longrightarrow}\dots \stackrel{R}{\longrightarrow} \pi_1\THR(k)^{\Z/2}
\]
satisfies the Mittag-Leffler condition, and therefore there is an isomorphism of rings
\[
\pi_0\TRR(k;2)^{\Z/2}\cong W(k;2).
\]
\end{prop}

\begin{proof}
We need to analyse the images in $\pi_1$ of the composite maps $R^j$.  Let $(\nsupseteq C_{2^j})$ be the family of subgroups of $D_{2^{n+j}}$ that do not contain $C_{2^j}$ (it is the family $\mathcal{R}$ when $j=1$). By taking the $D_{2^{n+j}}/C_{2^j}$-fixed points of the isotropy separation sequence for the subgroup $C_{2^j}\subset D_{2^{n+j}}$ we obtain a fibre sequence of spectra
\[
(\THR(k)\otimes E(\nsupseteq C_{2^j})_+)^{D_{2^{n+j}}}\longrightarrow\THR(k)^{D_{2^{n+j}}}\stackrel{}{\longrightarrow}(\THR(k)^{\phi C_{2^j}})^{D_{2^{n+j}}/C_{2^j}},
\]
and after identifying the third term with $\THR(k)^{D_{2^{n}}}$ using the real cyclotomic structure we obtain a fibre sequence
\[
(\THR(k)\otimes E(\nsupseteq C_{2^j})_+)^{D_{2^{n+j}}}\longrightarrow\THR(k)^{D_{2^{n+j}}}\stackrel{R^j}{\longrightarrow}\THR(k)^{D_{2^{n}}}.
\]
The group of components of the fibre can be calculated as the colimit
\[
\pi_0 (\THR(k)\otimes E(\nsupseteq C_{2^j})_+)^{D_{2^{n+j}}}\cong 
\colim_{\mathcal{O}_{(\nsupseteq C_{2^j})}}\underline{\pi}_0\THR(k)
\]
where $\mathcal{O}_{(\nsupseteq C_{2^j})}$ is the full subcategory of the orbit category of $D_{2^{n+j}}$ spanned by the subgroups in $(\nsupseteq C_{2^j})$. This is equivalent to the category
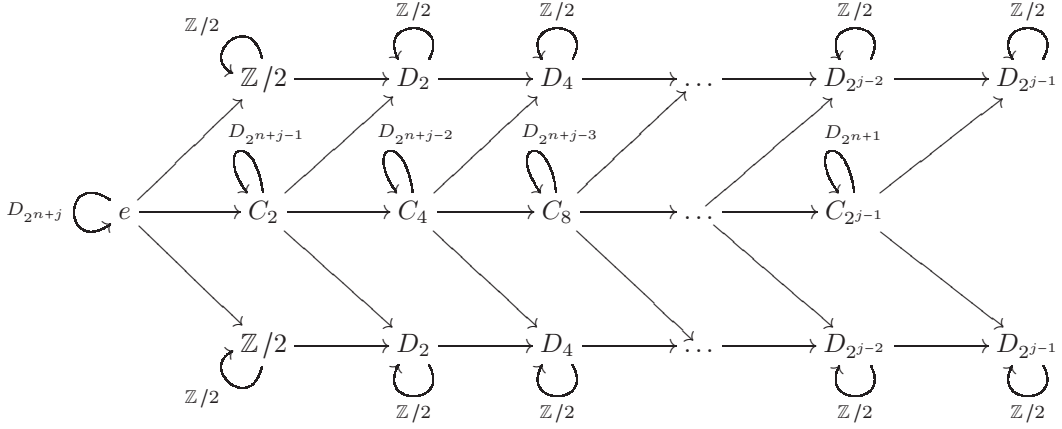
\[
\xymatrix@R=35pt@C=35pt{
&\Z/2\ar@(u,l)[]_{\Z/2}\ar[r]&D_2\ar@(ur,ul)[]_{\Z/2}\ar[r]&D_4\ar@(ur,ul)[]_{\Z/2}\ar[r]&\dots\ar[r]&D_{2^{j-2}}\ar@(ur,ul)[]_{\Z/2}\ar[r]&D_{2^{j-1}}\ar@(ur,ul)[]_{\Z/2}
\\
\ar@(ul,dl)[]_{D_{2^{n+j}}}e\ar[ur]\ar[dr]\ar[r]&C_2\ar@(u,ul)[]_(.4){D_{2^{n+j-1}}}\ar[ur]\ar[dr]\ar[r]&C_4\ar@(u,ul)[]_(.4){D_{2^{n+j-2}}}\ar[ur]\ar[dr]\ar[r]&C_8\ar@(u,ul)[]_(.4){D_{2^{n+j-3}}}\ar[ur]\ar[dr]\ar[r]&\dots\ar[r]\ar[ur]\ar[dr]&C_{2^{j-1}}\ar@(u,ul)[]_(.4){D_{2^{n+1}}}\ar[ur]\ar[dr]
\\
&\Z/2\ar[r]\ar@(d,l)[]^{\Z/2}&D_2\ar@(dr,dl)[]^{\Z/2}\ar[r]&D_4\ar@(dr,dl)[]^{\Z/2}\ar[r]&\dots\ar[r]&D_{2^{j-2}}\ar@(dr,dl)[]^{\Z/2}\ar[r]&D_{2^{j-1}}\ar@(dr,dl)[]^{\Z/2}
}
\]
Since the dihedral actions extend to an action of $O(2)$, the cyclic groups $C_{2^{n-i}}\leq D_{2^{n-i}}$ act trivially on $\pi_0\THR(A)^{C_{2^i}}$ and one can replace the dihedral groups of automorphisms of the middle row by the groups $\Z/2=D_{2^{n+j-i}}/C_{2^{n+j-i}}$. Thus this is the colimit over a product category, and it is isomorphic to
\[
 (\colim\big(
\xymatrix@C=15pt{
\pi_0 \THR(k)^{D_{2^{j-1}}}
&
\ar[l]_-{\tran}\THR(k)^{C_{2^{j-1}}}\ar[r]^-\tran
&
\pi_0 \THR(k)^{D_{2^{j-1}}}
}
\big))_{\Z/2}.
\]
Since the restriction map for the inclusion $C_{2^{j-1}}\subset D_{2^{j-1}}$  is an isomorphism, the Weyl actions on $\pi_0 \THR(k)^{D_{2^{j-1}}}$ are trivial, and by the previous calculation this is
\begin{align*}
\pi_0 (\THR(k)\otimes E(\nsupseteq C_{2^j})_+)^{D_{2^{n+j}}}\cong&  \colim\big(
\xymatrix@C=15pt{
W_{j}(k;2)
&
\ar[l]_-{2}W_{j}(k;2)\ar[r]^-2
&
W_{j}(k;2)
}
\big)\\
&\cong W_{j}(k;2)\oplus (W_{j}(k;2)/2).
\end{align*}
The last isomorphism sends the class of $(x,y)$ to $(x+y,[y])$.
We then obtain a long exact sequence
\[
\xymatrix@C=25pt@R=15pt{
\dots\ar[r]&\pi_1\THR(k)^{D_{2^{n+j}}}\ar[r]^-{R^j}&\pi_1\THR(k)^{D_{2^{n}}}\ar[r]^-\partial &W_{j}(k;2)\oplus W_{j}(k;2)/2\ar[r]^-{(V^{n+1},0)}&\\
\ar[r]^-{(V^{n+1},0)}&W_{n+j+1}(k;2)\ar[r]^-{R^j}&W_{n+1}(k;2)\ar[r]&0\ ,\hspace{4.3cm}
}
\]
where the map $V^{n+1}$ comes from the identification of the Verschiebung with the transfer of Theorem \ref{pi0TRRp2}.

We need to show that after a sufficiently large value of $j$ the image of $R^j$ is constant, that is that the projection map
\[
\pi_1\THR(k)^{D_{2^{n}}}/\im R^{j+l}\longrightarrow \pi_1\THR(k)^{D_{2^{n}}}/\im R^{j}
\]
is an isomorphism. By exactness, the target of this map is isomorphic to
\[
\pi_1\THR(k)^{D_{2^{n}}}/\im R^j=\pi_1\THR(k)^{D_{2^{n}}}/\ker \partial\cong \im \partial=\ker (V^{n+1},0)=W_{j}(k;2)/2,
\]
and similarly for the source. Thus the images stabilise if and only if the map
\[
R^l\colon W_{j+l}(k;2)/2\longrightarrow W_{j}(k;2)/2
\]
is an isomorphism, which is the case since $k$ is perfect as both sides identify with $k$ and $R$ with the identity.
\end{proof}

\subsubsection{TCR of perfect fields of characteristic $2$}

We now combine the results of the previous two sections to prove the following theorem.

\begin{theorem}\label{TRRperfectchar2} For any perfect field $k$ of characteristic $2$, there is an equivalence of $\Z/2$-equivariant 
ring spectra
\[\TRR(k;2) \simeq H\underline{W(k;2)},\]
where $\underline{W(k;2)}$ is the constant Green functor of the abelian group with trivial involution $W(k;2)$.
\end{theorem}

\begin{proof} By Theorem \ref{pi0TRRp2} and Proposition \ref{pi0TRRWittformula}, we understand the Mackey functor 
of components of $\TRR(k;2)$. Let $\TRR(k;2) \to H\underline{W(k;2)}$ be  the zeroth Postnikov section.
The diagram 
\[\xymatrix{\TRR(k;2) \ar[d] \ar[r] & H\underline{W(k;2)} \ar[d] \\  \THR(k;2) \ar[r] & H \underline{k}}\]
commutes, where the right vertical map is induced by the projection $W(k;2) \to k$ which induces an isomorphism $W(k;2)/2 \cong k$. 

The map $\TRR(k;2) \to H\underline{W(k;2)}$ is an underlying equivalence by \cite[Theorem 4.5]{Wittvect}. Hence it suffices to show that it is an equivalence after applying the geometric fixed points. By the calculation of Theorem \ref{geometricofTRnF2} (and in particular using the formula for $R$) we see that the map ${\TRR(k;2)}^{\phi \Z/2} \to {\THR(k;2)}^{\phi \Z/2}$ induces injections on homotopy groups. 
Hence it suffices to show that after applying the geometric fixed points the lower horizontal map induces an injection on the image of the left vertical map. Indeed, this will imply that the upper horizontal map induces an injection on the homotopy groups of the geometric fixed points, and since these are either $0$ or $1$ dimensional $k$-vector spaces (the target has homotopy groups $W(k;2)/2\cong k$ in even non-negative degrees) it must also be surjective. The lower map is, on geometric fixed-points, the multiplication map
\[H \underline{k}^{\phi\Z/2} \otimes_{H \underline{k}} H \underline{k}^{\phi\Z/2} \to H \underline{k}^{\phi\Z/2} .\]
By 
\cite[Proposition 5.19]{THRmodels} the induced map on homotopy groups
\[k[w_1, w_2] \longrightarrow k[v],\]
where $\vert w_1 \vert=1$, $\vert w_2 \vert=1$ and $\vert v \vert=1$, sends both $w_1$ and $w_2$ to $v$. This implies that its restriction
\[k[w_1w_2] \longrightarrow k[v]\]
is injective, and by the description of $R$ of Theorem \ref{geometricofTRnF2} $k[w_1w_2]$ is exactly the image of the left vertical map on homotopy groups after applying geometric fixed points. 
\end{proof}

\begin{cor} \label{TCRperfectchar2} For any perfect field $k$ of characteristic $2$, one has an equivalence of genuine $\Z/2$-spectra
\[ \TCR(k;2) \simeq H \underline{\Z_2} \oplus \Sigma^{-1}H \underline{\coker(1-F)},\]
where $F \colon W(k;2) \to W(k;2)$ is the Witt vector Frobenius. 
\end{cor}

\begin{proof} It follows from Theorem \ref{pi0TRRp2} that $F \colon \TRR(k;2) \to \TRR(k;2)$ corresponds to the Witt vector Frobenius $H\underline{F} \colon H\underline{W(k;2)} \to H\underline{W(k;2)}$ under the equivalence of Theorem \ref{TRRperfectchar2}. It is an easy exercise in Witt vectors to see that $\ker(1-F) \cong W({\mathbb F}_2;2) \cong \Z_2$. Hence we get
\[\underline{\pi}_0\TCR(k;2) \cong \underline{\Z_2} \]
and 
\[\underline{\pi}_{-1}\TCR(k;2) \cong \underline{\coker(1-F)},\]
and all the other homotopy Mackey functors of $\TCR(k;2)$ vanish. Since $\underline{\coker(1-F)}$ is a $\underline{\Z_2}$-module coming from a $\Z_2$-module, its homological dimension over the Green functor $\underline{\Z_2}$ is less than or equal to $1$. The universal coefficient theorem in the category of modules over the Green functor $\underline{\Z_2}$ now implies that in fact $\TCR(k;2)$ splits as claimed (in case $k=\F_2$ this is obvious since $F=\id$). \end{proof}

\section{TCR of the integers and perfect rings}

In this section we will calculate the homotopy type of $\TCR(A;2)^{\phi \Z/2}$ where $A$ is either a perfect $\F_2$-algebra or $2$-torsion free ring with a perfect$\mod 2$ reduction (for example the Witt vectors of a perfect $\F_2$-algebra). We will first calculate $\TCR(\Z;2)^{{\phi \Z/2}}$, and then deduce $\TCR(A;2)^{\phi \Z/2}$ by a base-change formula from $\F_2$ and $\mathbb{Z}$.

\subsection{The geometric fixed-points of \texorpdfstring{$\TCR(\Z;2)$}{TCR(Z;2)}}\label{sec:TCRZ}


Let us denote by $NA:=N_{e}^{C_2}HA$ the $C_2$-equivariant norm of the Eilenberg-MacLane ring spectrum of a commutative ring $A$. We regard $H\underline{\Z}$ (the $C_2$-equivariant Eilenberg-MacLane spectrum for the constant Mackey functor $\underline{\Z}$) as an $N\Z$-module via the multiplication map $N\Z\to H\underline{\Z}$. We then consider $H\underline{\Z}^{{\phi \Z/2}}$ as an $H\Z$-module via the induced map on geometric fixed-points $H\Z \simeq (N\Z)^{{\phi \Z/2}} \to H\underline{\Z}^{{\phi \Z/2}}$, and obtain
a splitting of $H\Z$-modules
\[H\underline{\Z}^{{\phi \Z/2}} \simeq \bigoplus_{n \geq 0} \Sigma^{2n} H\F_2.\]
Again using the description of $\THR(\Z)^{\phi\Z/2}$ as the derived smash product of Lemma \ref{C2THRphi} and the splitting above just as in Lemma \ref{geofixgenuine} , we obtain an equivalence of genuine $C_2$-spectra
\[\THR(\Z)^{{\phi \Z/2}} \simeq \bigoplus_{n\geq 0}\Sigma^{2n\rho} ((N\F_2)\otimes_{N\Z}H\underline{\Z})    \oplus  \bigoplus_{{ \begin{smallmatrix} (n,m) \\ 0 \leq n  < m \end{smallmatrix}}}\Sigma^{2n+2m}(C_2)_+\otimes ((H\F_2 \otimes H\F_2 )\otimes_{H\Z \otimes H\Z}H\Z).\]
In order to apply Theorem \ref{tcrformula1} to compute $\TCR(\Z;2)^{{\phi \Z/2}}$, we need to understand the genuine $C_2$-fixed points of this spectrum.
By the Wirthm\"uller isomorphism, the genuine $C_2$-fixed point spectrum of the induced summands are 
\begin{align*}(\Sigma^{2n+2m}(C_2)_+\otimes ((H\F_2 \otimes H\F_2 )\otimes_{H\Z \otimes H\Z}H\Z))^{C_2}&\simeq \Sigma^{2n+2m} ((H\F_2 \otimes H\F_2 )\otimes_{H\Z \otimes H\Z}H\Z)
\\
&\simeq \Sigma^{2n+2m} H\F_2 \oplus \Sigma^{2n+2m+1} H\F_2.
\end{align*}
The genuine fixed points of the terms $\Sigma^{2n\rho} ((N\F_2)\otimes_{N\Z}H\underline{\Z})$ are more laborious.
\begin{lemma}\label{cofibrenorm} There is a fibre sequence of $C_2$-spectra
\[
H(\Z\oplus \Z/2,w)\xrightarrow{H(2,0)} H\underline{\Z}\longrightarrow (N\F_2)\otimes_{N\Z}H\underline{\Z}
\]
where the left-hand spectrum is the Eilenberg MacLane spectrum of the abelian group $\Z\oplus \Z/2$ with involution $w(a,x)=(a,[a]+x)$. In particular the homotopy Mackey functors of the cofibre are
\[
\underline{\pi}_0((N\F_2)\otimes_{N\Z}H\underline{\Z})\cong
\big(
\xymatrix{\Z/2\ar@<1ex>[r]&\Z/4\ar@<1ex>[l]}
\big)
\]
where the restriction is the canonical projection and the transfer is injective, $\underline{\pi}_1((N\F_2)\otimes_{N\Z}H\underline{\Z})\cong \underline{\Z/2}$ which is the constant Mackey functor of $\Z/2$,  and the other homotopy groups vanish.
\end{lemma}

\begin{proof}
Let us first calculate the fibre of the canonical map $N\Z\to N\F_2$ given by the norm of the reduction modulo $2$. There is a commutative $C_2$-equivariant diagram of spectra (see \cite[Section 1.1]{Gdiags} for details on $G$-diagrams)
\[
\xymatrix{
N\Z\ar[r]^{1\otimes 2}\ar[d]_{2\otimes 1}&H\Z\otimes H\Z\ar[d]^{2\otimes 1}
\\
H\Z\otimes H\Z\ar[r]_{1\otimes 2}&N\Z
}
\]
whose total cofibre is equivalent to $N\F_2$. Let us denote by $P$ the pushout of the punctured square above. There is therefore a fibre sequence of $N\Z$-modules $P\to N\Z\to N\F_2$, giving rise to a fibre sequence of $C_2$-spectra
\[
P\otimes_{N\Z}H\underline{\Z}\longrightarrow (N\Z)\otimes_{N\Z}H\underline{\Z}\longrightarrow (N\F_2)\otimes_{N\Z}H\underline{\Z},
\]
with the middle term equivalent to $H\underline{\Z}$. Let us compute the left-hand map. By applying the functor $(-)\otimes_{N\Z}H\underline{\Z}$ to the $C_2$-equivariant square above, we obtain a $C_2$-equivariant pushout square
\[
\xymatrix{
H\underline{\Z}\ar[r]^{2}\ar[d]_{2}&H\Z\ar[d]
\\
H\Z\ar[r]&P\otimes_{N\Z}H\underline{\Z}.
}
\]
Let us identify this pushout. The underlying spectrum is $H((\Z\oplus \Z)/(2,-2))$ with the involution $\tau$ which flips the two summands. This is equivalent to $\Z\oplus \Z/2$ by the map that sends $[a,b]$ to $(a+b,[b])$, and the involution $w$ sends $(a,x)$ to $(a,[a]+x)$. We claim that $H(\Z\oplus \Z/2,w)$ is in fact the equivariant pushout. 
We need to verify that the map from the equivariant pushout to $H(\Z\oplus \Z/2,w)$ is an equivalence on geometric fixed-points. By the formula of \cite[Proposition 1.6]{HigherEx} the geometric fixed-points of the equivariant pushout are equivalent to the geometric fixed-points of the initial vertex, and therefore we need to show that the composite map from the top left corner of the square to $H(\Z\oplus \Z/2,w)$ is an equivalence on geometric fixed-points. This is the map induced on $C_2$-equivariant Eilenberg MacLane spectra by
\[
(2,0)\colon (\Z,\id)\xrightarrow{(2,0)} ((\Z\oplus \Z)/(2,-2),\tau)\cong (\Z\oplus \Z/2,w).
\]
Let us compute its cofibre, and show that is has trivial geometric fixed-points. The quotient of $(2,0)$ is $\Z/2\oplus\Z/2$ with the involution $w(y,x)=(y,y+x)$. Its fixed-points are $0\oplus \Z/2$, which is also the quotient of the map $(2,0)\colon \Z\to (\Z\oplus \Z/2)^{C_2}\cong (2\Z)\oplus \Z/2$. It follows that there is a fibre sequence of $C_2$-spectra
\[
H\underline{\Z}\xrightarrow{(2,0)} H(\Z\oplus \Z/2,w)\longrightarrow H(\Z/2\oplus\Z/2,w).
\]
The cofibre is equivariantly equivalent to $H(\Z/2\oplus\Z/2,\tau)$ where $\tau$ flips the summands, and therefore its geometric fixed-points vanish.

From the identification of this pushout we obtain a fibre sequence of $C_2$-spectra
\[
H(\Z\oplus \Z/2,w)\longrightarrow H\underline{\Z}\longrightarrow (N\F_2)\otimes_{N\Z}H\underline{\Z}
\]
where the first map is the map of abelian groups with involution that sends $(a,[b])$ to $2a$. The cokernel and kernel of this map identify respectively the zero-th and first homotopy Mackey functor of the cofibre with those of the statement, and the others are zero.
\end{proof}

\begin{lemma} \label{Bredon for HZ}
The $C_2$-equivariant homotopy groups of  $\Sigma^{k\rho}((N\F_2)\otimes_{N\Z}H\underline{\Z})$ for even $k$ are
\[
\pi^{C_2}_\ast\Sigma^{k\rho}((N\F_2)\otimes_{N\Z}H\underline{\Z})=
\left\{\begin{array}{cl}
\Z/2& k\leq \ast\leq 2k-1
\\
\Z/4 &\ast=2k
\\
\Z/2&\ast=2k+1
\\
0& \mbox{otherwise.}
\end{array}\right.
\]
\end{lemma}

\begin{proof}
We calculate the equivariant homotopy groups from the fibre sequence of $C_2$-spectra
\[
\Sigma^{k\rho}H(\Z\oplus \Z/2,w)\xrightarrow{(2,0)} \Sigma^{k\rho}H\underline{\Z}\longrightarrow \Sigma^{k\rho}(N\F_2)\otimes_{N\Z}H\underline{\Z}
\]
from Lemma \ref{cofibrenorm}. We start by calculating the equivariant homotopy groups of the first two spectra. These are respectively the Bredon homology groups of $S^{k\rho}$ with coefficients in the Mackey functors of the abelian groups $\Z$ with the the trivial involution, and $\Z\oplus \Z/2$ with the involution $w(a,x)=(a,[a]+x)$.
These are respectively the homology of the chain complexes
\[
\big(0\to \Z\xrightarrow{0}\Z\xrightarrow{2}
\dots \xrightarrow{0}\Z\xrightarrow{2}\Z\xrightarrow{0}\Z\xrightarrow{2}\Z\to 0
\big)
\]
where the  non-zero groups are sitting between degree $k$ and $2k$ (and $k$ is even), and
\[
\big(0\to \Z\times \Z/2\xrightarrow{1-w}\Z\times \Z/2\xrightarrow{1+w}
\dots \xrightarrow{1-w}\Z\times \Z/2\xrightarrow{1+w}\Z\times \Z/2\xrightarrow{1-w}\Z\times \Z/2\xrightarrow{1+w}(2\Z)\times \Z/2\to 0
\big)
\]
with the non-zero groups sitting in the same degrees. The first complex has homology groups $\Z/2$ in even degrees between $k$ and $2k-2$, a $\Z$ in degree $2k$, and zero everywhere else. The differentials of the second complex are respectively
\[
(1+w)(a,x)=(2a,[a])\ \ \ \ \ \ \ \ \ \ \  (1-w)(a,x)=(0,[a])
\]
for all $(a,x)\in \Z\times \Z/2$. Its homology is then concentrated in even degrees between $k$ and $2k$, with
\[
((2\Z)\times \Z/2)/\langle (2a,[a]) \rangle \cong \Z/2
\]
in even degrees between $k$ and $2k-2$ and $(2\Z)\times\Z/2$ in degree $2k$. The long exact sequence of the above fibre sequence therefore splits and gives rise to exact sequences
\[
0\to  \pi_{2k+1}^{C_2}(\Sigma^{k\rho}(N\F_2)\otimes_{N\Z}H\underline{\Z})\longrightarrow (2\Z)\times\Z/2\xrightarrow{(2,0)}\Z\longrightarrow  \pi_{2k}^{C_2}(\Sigma^{k\rho}(N\F_2)\otimes_{N\Z}H\underline{\Z}) \to 0
\]
and 
\[
0\to  \pi_{2k-1}^{C_2}(\Sigma^{k\rho}(N\F_2)\otimes_{N\Z}H\underline{\Z})\longrightarrow \Z/2\stackrel{0}{\longrightarrow}\Z/2\longrightarrow  \pi_{2k-2}^{C_2}(\Sigma^{k\rho}(N\F_2)\otimes_{N\Z}H\underline{\Z}) \to 0
\]
\[
0\to  \pi_{2k-3}^{C_2}(\Sigma^{k\rho}(N\F_2)\otimes_{N\Z}H\underline{\Z})\longrightarrow \Z/2\stackrel{0}{\longrightarrow}\Z/2\longrightarrow  \pi_{2k-4}^{C_2}(\Sigma^{k\rho}(N\F_2)\otimes_{N\Z}H\underline{\Z}) \to 0
\]
\[\vdots\]
\[
0\to  \pi_{k+1}^{C_2}(\Sigma^{k\rho}(N\F_2)\otimes_{N\Z}H\underline{\Z})\longrightarrow \Z/2\stackrel{0}{\longrightarrow}\Z/2\longrightarrow  \pi_{k}^{C_2}(\Sigma^{k\rho}(N\F_2)\otimes_{N\Z}H\underline{\Z}) \to 0
\]
which give the groups of the statement.
\end{proof}

Since $H\underline{\Z} \to (N\F_2)\otimes_{N\Z}H\underline{\Z}$ is a map of $C_2$-equivariant algebras, the fixed point spectrum $(\Sigma^{k\rho}(N\F_2)\otimes_{N\Z}H\underline{\Z})^{C_2}$ is a module over $H\mathbb{Z}$ and therefore splits as a wedge of Eilenberg-MacLane spectra. As a consequence of Lemma \ref{Bredon for HZ}  we obtain an equivalence
\begin{align*}
(\THR(\Z)^{{\phi \Z/2}})^{C_2}& \simeq 
\bigoplus_{ \begin{smallmatrix} (n,m) \\ n>m\geq 0  \end{smallmatrix}}(\Sigma^{2n+2m}(H\F_2 \oplus \Sigma H\F_2)) \oplus \bigoplus_{n\geq 0}(\Sigma^{4n}(H\Z/4 \oplus \Sigma H\F_2))
\\
&\oplus   \bigoplus_{ \begin{smallmatrix} (n,m) \\ 0 \leq n  < m \end{smallmatrix}}(\Sigma^{2n+2m}(H\F_2 \oplus \Sigma H\F_2)).
\end{align*}
We recall that the underlying non-equivariant spectrum of $\THR(\Z)^{{\phi \Z/2}}$ is equivalent to 
\[\bigoplus_{ \begin{smallmatrix} (n,m) \\ n>m\geq 0 \end{smallmatrix}}(\Sigma^{2n+2m}(H\F_2 \oplus \Sigma H\F_2)) \oplus \bigoplus_{n\geq 0}(\Sigma^{4n}(H\F_2 \oplus \Sigma H\F_2)) \oplus   \bigoplus_{ \begin{smallmatrix} (n,m) \\ 0 \leq n  < m \end{smallmatrix}}(\Sigma^{2n+2m}(H\F_2 \oplus \Sigma H\F_2)),\]
and we now want to identify the maps $r,f \colon (\THR(\Z)^{{\phi \Z/2}})^{C_2}\to \THR(\Z)^{{\phi \Z/2}}$ under these splittings.

\begin{prop} \label{F for Z} Under the above equivalences the map $f \colon (\THR(\Z)^{{\phi \Z/2}})^{C_2} \to \THR(\Z)^{{\phi \Z/2}}$ corresponds to the map
\begin{align*} \bigoplus_{ \begin{smallmatrix} (n,m) \\ n>m\geq 0 \end{smallmatrix}}(\Sigma^{2n+2m}(H\F_2 \oplus \Sigma H\F_2)) \oplus \bigoplus_{n\geq 0}(\Sigma^{4n}(H\Z/4 \oplus \Sigma H\F_2)) \oplus   \bigoplus_{ \begin{smallmatrix} (n,m) \\ 0 \leq n  < m \end{smallmatrix}}(\Sigma^{2n+2m}(H\F_2 \oplus \Sigma H\F_2)) 
&
\to 
\\  \bigoplus_{ \begin{smallmatrix} (n,m) \\ n>m\geq 0\end{smallmatrix}}(\Sigma^{2n+2m}(H\F_2 \oplus \Sigma H\F_2)) \oplus \bigoplus_{n\geq 0}(\Sigma^{4n}(\hspace{.15cm}  H\F_2 \hspace{.2cm} \oplus \Sigma H\F_2))  \oplus   \bigoplus_{ \begin{smallmatrix} (n,m) \\ 0 \leq n  < m \end{smallmatrix}}(\Sigma^{2n+2m}(H\F_2 \oplus \Sigma H\F_2))  \end{align*}
which kills the $(n>m)$-summands, embeds diagonally  the $(n<m)$-summands into the sum of the summands $(n>m)$ and $(n<m)$, and on the remaining summands it  is given by
\[\pr \oplus \id \colon \Sigma^{4n}(H\Z/4 \oplus \Sigma H\F_2) \to \Sigma^{4n}(H\F_2 \oplus \Sigma H\F_2).\]
\end{prop}

\begin{proof} That $f$ sends the $(n<m)$-summands diagonally into the sum of the summands $(n>m)$ and $(n<m)$ follows from the Wirthm\"uller isomorphism. For the remaining summands, we need to understand the restriction map
\[\res^{C_2}_e \colon (\Sigma^{2n\rho}((N\F_2)\otimes_{N\Z}H\underline{\Z}))^{C_2} \to \Sigma^{4n}((H\F_2 \otimes H\F_2)\otimes_{H\Z \otimes H\Z}H\Z).\]
For every fixed $n\geq 0$, the sequence
\[
\Sigma^{2n\rho}H(\Z\oplus \Z/2,w)\xrightarrow{(2,0)} \Sigma^{2n\rho}H\underline{\Z}\longrightarrow \Sigma^{2n\rho}(N\F_2)\otimes_{N\Z}H\underline{\Z}
\]
is a fibre sequence of $H\underline{\Z}$-modules, and it thus induces a commutative diagram of $H\Z$-modules
\[\xymatrix@C=15pt{(\Sigma^{2n\rho}H(\Z\oplus \Z/2,w))^{C_2} \ar[d]_-{\res^{C_2}_e} \ar[r]^-{(2,0)} & (\Sigma^{2n\rho}H\underline{\Z})^{C_2} \ar[d]^{\res^{C_2}_e}  \ar[r] & (\Sigma^{2n\rho}(N\F_2)\otimes_{N\Z}H\underline{\Z})^{C_2} \ar[r] \ar[d]^{\res^{C_2}_e}  & (\Sigma^{2n\rho+1}H(\Z\oplus \Z/2,w))^{C_2} \ar[d]^-{\res^{C_2}_e} \\ \Sigma^{4n}(H\Z \oplus H\Z/2) \ar[r]^-{(2,0)} & \Sigma^{4n}H\Z \ar[r] & \Sigma^{4n} (H\F_2 \oplus \Sigma H \F_2) \ar[r] & \Sigma^{4n+1}(H\Z \oplus H\Z/2).}  \]
Using the Bredon complexes in the proof of Lemma \ref{Bredon for HZ}, we see that the left hand square in the latter diagram is equivalent to the commutative square
\[\xymatrix@R=40pt@C=40pt{ 
\displaystyle
 \Sigma^{4n}(H(2\Z \oplus \Z/2)) \oplus (\bigoplus_{ \begin{smallmatrix} (n,m) \\ n>m\geq 0 \end{smallmatrix}} \Sigma^{2n+2m}H\F_2) \ar[d]_-{(\incl \oplus \id) \oplus 0} \ar[r]^-{(2,0) \oplus 0} & 
 \displaystyle
  \Sigma^{4n}H\Z \oplus (\bigoplus_{ \begin{smallmatrix} (n,m) \\ n>m\geq 0 \end{smallmatrix}}\Sigma^{2n+2m}H\F_2) \ar[d]^-{\id \oplus 0} \\ \Sigma^{4n}(H\Z \oplus H\F_2) \ar[r]^-{(2,0)} & \Sigma^{4n}H\Z.}\]
After taking horizontal cofibres it induces the map
\[ \res^{C_2}_e \colon \Sigma^{4n}H\Z/4 \oplus \Sigma^{4n+1} H\F_2 \oplus  \bigoplus_{ \begin{smallmatrix} (n,m) \\ n>m\geq 0 \end{smallmatrix}}(\Sigma^{2n+2m}(H\F_2 \oplus \Sigma H\F_2)) \to \Sigma^{4n}H\F_2 \oplus \Sigma^{4n+1} H\F_2,\]
which is given by $\pr \oplus \id \oplus 0$.
\end{proof}

The identification of the map $r \colon (\THR(\Z)^{{\phi \Z/2}})^{C_2} \to \THR(\Z)^{\phi \Z/2}$ in terms of the above splittings will contain higher stable cohomology operations, and this complicates the calculation of the equaliser of $r$ and $f$. However, like in the case of fields, it is possible to compute $r$ on homotopy groups and after identifying only a portion of the matrix describing $r$ we will be able to compute $\TCR(\Z;2)^{\phi \Z/2}$ using Theorem \ref{tcrformula1}. 

%

\begin{prop} \label{R for Z} Under the above splittings, the map $r \colon (\THR(\Z)^{{\phi \Z/2}})^{C_2} \to \THR(\Z)^{\phi \Z/2}$ corresponds to the map 
\begin{align*} r\colon& \bigoplus_{ \begin{smallmatrix} (n,m) \\ n>m\geq 0 \end{smallmatrix}}(\Sigma^{2n+2m}(H\F_2 \oplus \Sigma H\F_2)) \oplus \bigoplus_{n\geq 0}(\Sigma^{4n}(H\Z/4 \oplus \Sigma H\F_2)) \oplus   \bigoplus_{ \begin{smallmatrix} (n,m) \\ 0 \leq n  < m \end{smallmatrix}}(\Sigma^{2n+2m}(H\F_2 \oplus \Sigma H\F_2)) &\to 
\\
& \bigoplus_{ \begin{smallmatrix} (n,m) \\ n>m\geq 0 \end{smallmatrix}}(\Sigma^{2n+2m}(H\F_2 \oplus \Sigma H\F_2)) \oplus \bigoplus_{n\geq 0}(\Sigma^{4n}(\hspace{.15cm}H\F_2\hspace{.2cm} \oplus \Sigma H\F_2)) \oplus   \bigoplus_{ \begin{smallmatrix} (n,m) \\ 0 \leq n  < m \end{smallmatrix}}(\Sigma^{2n+2m}(H\F_2 \oplus \Sigma H\F_2)) \end{align*} 
with the following properties. It is zero on the $(n<m)$-summands. On the summands $(n>m)$ it has components
\[\Sigma^{2n+2m}(H\F_2 \oplus \Sigma H\F_2) \to \Sigma^{2n'+2m'}(H\F_2 \oplus \Sigma H\F_2) \]
which are zero if $n\neq n'$ or $m'<m$, and the identity if $n=n'$ and $m=m'$. The entry 
\[\Sigma^{4n}(H\Z/4 \oplus \Sigma H\F_2) \to \Sigma^{4n}(H\F_2 \oplus \Sigma H\F_2) \]
is given by the matrix $\left(\begin{smallmatrix}\pr &  0 \\ \Sigma^{4n}\beta & \id  \end{smallmatrix}\right)$, where $\beta \colon  H\Z/4 \to \Sigma H\F_2$ is the Bockstein associated to the short exact sequence
\[0 \to \Z/2 \to \Z/8 \to \Z/4 \to 0. \]
The remaining entries are zero on homotopy groups, but generally contain higher stable cohomology operations (cf. \cite[Section IV.1]{NS}). 
\end{prop}

\begin{proof} The map $r$ vanishes on the summands $(n<m)$ since it factors through the geometric fixed-points of an induced spectrum which are zero. From the description of $r$ of Example \ref{minicycloTHR} we see that $r$ preserves the wedge decomposition over $n$, and  thus its components vanish for $n\neq n'$. It remains to identify 
\[
r\colon (\Sigma^{2n\rho}(N\F_2)\otimes_{N\Z}H\underline{\Z})^{C_2} \longrightarrow (\Sigma^{2n\rho}(N\F_2)\otimes_{N\Z}H\underline{\Z})^{\phi C_2} 
\]
for every fixed $n\geq 0$. The fibre sequence of Lemma \ref{cofibrenorm} induces a commutative diagram
\[\xymatrix@C=15pt{(\Sigma^{2n\rho}H(\Z\oplus \Z/2,w))^{C_2} \ar[d] \ar[r]^-{(2,0)} & (\Sigma^{2n\rho}H\underline{\Z})^{C_2} \ar[d]  \ar[r] & (\Sigma^{2n\rho}(N\F_2)\otimes_{N\Z}H\underline{\Z})^{C_2} \ar[r] \ar[d]  & (\Sigma^{2n\rho+1}H(\Z\oplus \Z/2,w))^{C_2} \ar[d] \\ \Sigma^{2n} H(\Z\oplus \Z/2,w)^{\phi C_2} \ar[r]^-{(2,0)} & \Sigma^{2n}H\underline{\Z}^{\phi C_2} \ar[r] & \Sigma^{2n} ((N\F_2)\otimes_{N\Z}H\underline{\Z})^{\phi C_2}  \ar[r] & \Sigma^{2n+1} H(\Z\oplus \Z/2,w)^{\phi C_2}.} \]
For any underlying connective $C_2$-spectrum $X$, the canonical map $\pi_*^{C_2}(\Sigma^{l\rho}  X) \to \pi_*((\Sigma^{l\rho}  X)^{\phi C_2})$
 induces an isomorphism in degrees $\ast < 2l$ and a surjection in degree $\ast=2l$, since the homotopy orbits of $\Sigma^{l\rho}  X$ are $(2l-1)$-connected. By applying this fact to the vertical maps of the commutative diagram above we obtain the description of $r$ on the summands $(n>m)$.
Let us finally compute the map 
\[r\colon \Sigma^{4n}(H\Z/4 \oplus \Sigma H\F_2) \to \Sigma^{4n}(H\F_2 \oplus \Sigma H\F_2). \]
Using that the fibre sequence of Lemma \ref{cofibrenorm}  is $H\underline{\Z}$-linear, by considering the relevant summands in the diagram above we get a morphism of exact triangles
\[ \xymatrix{  \Sigma^{4n} (H2\Z \oplus H\F_2) \ar[d]^-{(\pr,\id)} \ar[r]^-{(2,0)} & \Sigma^{4n}H\Z \ar[r]^-{\pr \oplus 0} \ar[d]^-{\pr} & \Sigma^{4n} H\Z/4 \oplus \Sigma^{4n+1} H\F_2 \ar[rr]^{\Sigma^{4n}\beta_{\Z} \oplus \id } \ar[d]^r & & \Sigma^{4n+1} (H2\Z \oplus H\F_2) \ar[d]^-{(\pr,\id)} \\ \Sigma^{4n}H\F_2 \ar[r]^0 & \Sigma^{4n}H\F_2 \ar[r]^-{\id \oplus 0} & \Sigma^{4n}H\F_2 \oplus \Sigma^{4n+1}H\F_2  \ar[rr]^{0 \oplus \id} & & \Sigma^{4n+1}H\F_2, }  \]
where $r$ is the map we are trying to compute, and $\beta_{\Z}$ is the Bockstein of
\[0 \to 2\Z \xrightarrow{2} \Z \to \Z/4 \to 0.\]
Composing the Bockstein $\beta_{\Z}$ with the projection $H2\Z \to H\F_2$  gives the Bockstein for
\[0 \to \Z/2 \to \Z/8 \to \Z/4 \to 0\]
which gives the desired result. \end{proof}

\begin{theorem} \label{geometric of TCRZ computation} There is an equivalence of spectra
\[\TCR(\Z;2)^{{\phi \Z/2}}\simeq \bigoplus_{n \geq 0} \big(\Sigma^{4n-1} H \F_2\oplus \Sigma^{4n} H \Z/8 \oplus \Sigma^{4n+1} H \F_2 \big). \]
\end{theorem}

\begin{proof} For simplicity we use the symbol $\bigoplus_{n>m\geq 0}$ to denote $\bigoplus_{(n,m), n>m\geq 0}(\Sigma^{2n+2m}(H\F_2 \oplus \Sigma H\F_2))$ and a similar symbol for the summands indexed by the pairs $(n,m)$ with $0\leq n<m$. Consider the commutative diagram in the stable homotopy category, where the vertical sequences are fibre sequences:
\[ \xymatrix{ \bigoplus_{n>m\geq 0} \oplus   \bigoplus_{0 \leq n<m} \ar[d]^{\incl} \ar[r]^{\omega}_{\simeq} & \bigoplus_{n>m\geq 0} \oplus \bigoplus_{0 \leq n<m} \ar@<0.5ex>[d]^{\alpha=(\id, \varphi, \id)} 
\\
\displaystyle \bigoplus_{n>m\geq 0}\oplus \big( \bigoplus_{n\geq 0}(\Sigma^{4n}(H\Z/4 \oplus \Sigma H\F_2))\big) \oplus   \bigoplus_{0 \leq n<m} \ar[d]^{\pr}  \ar[r]^-{r-f}
  &
 \displaystyle   \bigoplus_{n>m\geq 0}\oplus \big( \bigoplus_{n\geq 0}\Sigma^{4n}(H\F_2 \oplus \Sigma H\F_2) \big)\oplus   \bigoplus_{0 \leq n<m} \ar@<0.5ex>[u]^{\pr} \ar[d]^{(-\varphi, \id)} 
    \\
     \bigoplus_{n\geq 0}\Sigma^{4n}(H\Z/4 \oplus \Sigma H\F_2) \ar@{-->}[r]^{M} & \bigoplus_{n\geq 0}\Sigma^{4n}(H\F_2 \oplus \Sigma H\F_2)} \]
We explain the maps in the diagram: The top map $\omega$ is the composite $\pr \circ (r-f) \circ \incl$ and is an equivalence since by  Propositions \ref{F for Z} and \ref{R for Z} it is an isomorphism on homotopy groups. The map $\alpha$ is then defined to be $(r-f) \circ \incl \circ \omega^{-1}$ and by construction is of the form $(\id, \varphi, \id)$, for some map
\[\varphi \colon \bigoplus_{n>m\geq 0} \oplus   \bigoplus_{0 \leq n<m} \to \bigoplus_{n\geq 0} \Sigma^{4n}(H\F_2 \oplus \Sigma H\F_2).\]
The lower right vertical map is $-\varphi$ on the outer summands and the identity on the middle summand. The map $M$  is the induced map on the cofibres.
Propositions \ref{F for Z} and \ref{R for Z} imply that the map $\varphi$ is zero on the summand $\bigoplus_{0 \leq n<m}$. On the other hand, 
the restriction of $r-f$ to the summand $\Sigma^{4n}(H\Z/4 \oplus \Sigma H\F_2)$ cannot hit $\bigoplus_{n>m\geq 0}$, since the cohomology operations do not decrease degrees and $r$ preserves the $n$-coordinate. Thus $M$ is given in matrix form by the wedge 
\[\bigoplus_{n \geq 0}\left(\begin{smallmatrix}0 &  0 \\ \Sigma^{4n}\beta & 0  \end{smallmatrix}\right).\]
The fibres of $r-f$ and $M$ are equivalent since $\omega$ is an equivalence. This completes the proof by Theorem \ref{tcrformula1}.
\end{proof}

\subsection{Flat base-change and perfect rings}

We recall that we always regard the geometric $\Z/2$-fixed points of a ring spectrum with anti-involution $A$ as a left $A$-module via the geometric fixed-points of the map of $\Z/2$-spectra $N^{\Z/2}_eA\otimes A\to A$, and similarly as a right $A$-module via $A\otimes N^{\Z/2}_eA\to A$. We call these respectively the left and right Frobenius module structures on $A^{\phi\Z/2}$. We will always denote by $\otimes_A$ the \emph{derived} tensor product of $A$-modules.

\begin{defn} A map $\alpha \colon A \to B$ of ring-spectra with anti-involution is called $\phi$-flat if the map 
\[
B\otimes_AA^{{\phi\Z/2}}\longrightarrow B^{{\phi\Z/2}},
\]
induced by the map of left $A$-modules $\alpha \colon A^{{\phi\Z/2}} \to \alpha^\ast B^{{\phi\Z/2}}$, is an equivalence of spectra. \end{defn}

\begin{example}\label{ex:commflat} Let $\alpha \colon A \to B$ be a map of commutative rings with trivial involution which is flat $2$-locally. Then the induced map on Eilenberg-MacLane spectra is $\phi$-flat precisely if the maps
\[B\otimes_A {}_\varphi A/2 \to {}_\varphi B/2 \ \ \ \ \ \ \ \ \ \  \ \ \mbox{and} \ \ \ \ \ \ \ \ \ \  \ \ \ B\otimes_A {}_\varphi A_2 \to {}_\varphi B_2\]
adjoint to $\alpha/2$
are isomorphisms, where ${}_\varphi(-)$ denotes the module structure $r\cdot x:=r^2x$, and $(-)_2$ the two-torsion. Indeed since $\alpha$ is flat $2$-locally
, $\alpha$ is $\phi$-flat precisely if
\[B\otimes_{A}\pi_\ast (H\underline{A}^{\phi \Z/2}) \to \pi_\ast (H\underline{B}^{\phi \Z/2})\]
is an isomorphism. Since $H\underline{A}^{\phi \Z/2}$ is the connective cover of $H\underline{A}^{t \Z/2}$ which is $2$-periodic, and similarly for $H\underline{B}^{\phi \Z/2}$, this is equivalent to showing that $B\otimes_A {}_\varphi \hat{H}^i(\Z/2,A) \to {}_\varphi\hat{H}^i(\Z/2,B)$ are isomorphisms for $i = 0,1$, and this is exactly the assumption above. In particular:
\begin{enumerate}
\item If $B$ is a perfect $\F_2$-algebra with trivial involution, then the map $\F_2\to B$ is $\phi$-flat. Indeed the maps above are both isomorphic to the Frobenius $(-)^2\colon B\to B$.
\item If $B$ is a commutative ring with trivial involution with no $2$-torsion, and $B/2$ is perfect, then $\Z\to B$ is $\phi$-flat. Indeed the maps above are in this case respectively the Frobenius of $B/2$ and the map $0\to B_2$.
\end{enumerate}
\end{example}

Recall that, as a $C_2$-spectrum, $\THR(A)^{{\phi\Z/2}}$ is equivalent to $B(A;N^{C_2}_eA;N^{C_2}_e(A^{\phi\Z/2}))$, where $A$ is regarded as a $C_2$-spectrum via the identification $C_2\cong \Z/2$ (see Lemma \ref{C2THRphi}). In particular $\THR(A)^{{\phi\Z/2}}$ is canonically a module over $A$ in the category of $C_2$-spectra, by acting on the left copy of $A$ in the bar construction.

\begin{prop}\label{generalphibase-change}
Let $\alpha \colon A \to B$ be a $\phi$-flat map of commutative $\Z/2$-equivariant ring spectra. Then the canonical map
\[
B\otimes_A(\THR(A)^{{\phi\Z/2}})\stackrel{\simeq}{\longrightarrow}\THR(B)^{{\phi\Z/2}}
\]
induced by $\alpha\colon \THR(A)^{{\phi\Z/2}}\to \alpha^\ast \THR(B)^{{\phi\Z/2}}$ is an equivalence of $C_2$-spectra. Here $B$ is considered as a $C_2$-spectrum via the isomorphism $C_2 \cong \Z/2$ (see Lemma \ref{C2THRphi}). 
\end{prop}

\begin{proof}
Let us first show that the map is an equivalence on underlying spectra. This is the map $B\otimes_A(A^{{\phi\Z/2}}\otimes_AA^{\phi\Z/2})\to B^{{\phi\Z/2}}\otimes_BB^{\phi\Z/2}$ induced by the map of left $A$-modules $\alpha\otimes \alpha\colon A^{{\phi\Z/2}}\otimes_AA^{\phi\Z/2}\to \alpha^\ast (B^{{\phi\Z/2}}\otimes_BB^{\phi\Z/2})$, where the left $A$-module structure on the source is the left Frobenius structure on the right $A^{\phi\Z/2}$-factor (or equivalently the right one on the left factor), and similarly for the $B$-module structure on the target. 
Since $A$ is commutative this $A$-module structure agrees with the left Frobenius structure on the first   $A^{\phi\Z/2}$-factor, and therefore the map factors as
\begin{align*}
B\otimes_A(A^{{\phi\Z/2}}\otimes_AA^{\phi\Z/2})&=(B\otimes_AA^{{\phi\Z/2}})\otimes_AA^{\phi\Z/2}\xrightarrow{\simeq} B^{{\phi\Z/2}}\otimes_AA^{\phi\Z/2}
\\&\simeq B^{{\phi\Z/2}}\otimes_B B\otimes_AA^{\phi\Z/2}\xrightarrow{\simeq} B^{{\phi\Z/2}}\otimes_B B^{\phi\Z/2}
\end{align*}
where the two right pointing arrows are equivalences since $\alpha$ is $\phi$-flat.

Let us now verify that this map is an equivalence on $C_2$-geometric fixed-points. From the bar construction we see that, this is the map
\[
(B\otimes_A(\THR(A)^{{\phi\Z/2}}))^{\phi C_2}\simeq B^{\phi\Z/2}\otimes_{A^{\phi\Z/2}} \THR(A)^{{\phi\Z/2}}
\stackrel{}{\longrightarrow}\THR(B)^{{\phi\Z/2}}\simeq (\THR(B)^{{\phi\Z/2}})^{\phi C_2}
\]
induced by the map $\alpha\colon \THR(A)^{{\phi\Z/2}}
\to \alpha^\ast \THR(B)^{{\phi\Z/2}}$, where $\THR(A)^{{\phi\Z/2}}=A^{\phi\Z/2}\otimes_AA^{\phi\Z/2}$ is a left $A^{\phi\Z/2}$-module via left multiplication on the left factor (notice that $A^{\phi\Z/2}$ is a ring spectrum since $A$ is commutative), and similarly for the left $B^{\phi\Z/2}$-module structure on the target.
This map then factors as
\begin{align*}
B^{\phi\Z/2}\otimes_{A^{\phi\Z/2}} \THR(A)^{{\phi\Z/2}}&=
B^{\phi\Z/2}\otimes_{A^{\phi\Z/2}} A^{\phi\Z/2}\otimes_AA^{\phi\Z/2}
\simeq 
B^{\phi\Z/2}\otimes_AA^{\phi\Z/2}
\\
&\simeq  B^{\phi\Z/2}\otimes_BB\otimes_AA^{\phi\Z/2}
\xrightarrow{\simeq} B^{\phi\Z/2}\otimes_BB^{\phi\Z/2}=\THR(B)^{{\phi\Z/2}}
\end{align*}
where for the last equivalence we used that $\alpha$ is $\phi$-flat.
\end{proof}

\begin{prop}
Under the assumptions of \ref{generalphibase-change}, suppose moreover that the restriction maps $A^{\Z/2}\to A$ and $B^{\Z/2}\to B$ are equivalences (for example if $A$ and $B$ are the Eilenberg-MacLane spectra of commutative rings with trivial involutions). Then there is an equivalence
\[
(B\otimes_A(\THR(A)^{{\phi\Z/2}}))^{C_2}\simeq B^{C_2}\otimes_{A^{C_2}}((\THR(A)^{{\phi\Z/2}}))^{C_2},
\]
and the maps $f,r\colon (\THR(B)^{{\phi\Z/2}})^{C_2}\to\THR(B)^{{\phi\Z/2}}$ correspond under the equivalences of Proposition \ref{generalphibase-change} respectively to the tensor of the restriction maps
\[
f\colon B^{C_2}\otimes_{A^{C_2}}((\THR(A)^{{\phi\Z/2}}))^{C_2}\xrightarrow{\res\otimes\res} B\otimes_A(\THR(A)^{{\phi\Z/2}})
\]
and to the tensor of the canonical map to the geometric fixed-points and the map $r$ of $\THR(A)^{{\phi\Z/2}}$
\begin{align*}
r\colon &B^{C_2}\otimes_{A^{C_2}}((\THR(A)^{{\phi\Z/2}}))^{C_2}\xrightarrow{can\otimes r}  B^{{\phi\Z/2}}\otimes_{A^{{\phi\Z/2}}}(\THR(A)^{{\phi\Z/2}})\simeq
\\& B\otimes_AA^{{\phi\Z/2}}\otimes_{A^{{\phi\Z/2}}}(\THR(A)^{{\phi\Z/2}}) \simeq B\otimes_A(\THR(A)^{{\phi\Z/2}}),
\end{align*}
where the first equivalence is from the fact that $\alpha$ is $\phi$-flat, and the second is the canonical one.
\end{prop}

\begin{proof}
The first statement follows from the fact that if the restriction maps of $A$ and $B$ are equivalences, the canonical map
\[
B^{C_2}\otimes_{A^{C_2}}X^{C_2}\longrightarrow(B\otimes_AX)^{C_2}
\]
is an equivalence for every $A$-module $X$ (which is cofibrant under our standing assumption). Indeed since the source and target of this map commute with homotopy colimits in $X$, it is sufficient to check it on the generators $A$ and $A\otimes (C_2)_+$ of the category of $A$-modules. For $A$ this is the canonical equivalence
\[
B^{C_2}\otimes_{A^{C_2}}A^{C_2}\simeq B^{C_2}\simeq (B\otimes_AA)^{C_2}.
\]
For $A\otimes (C_2)_+$ this is the map
\[
B^{C_2}\otimes_{A^{C_2}}(A\otimes (C_2)_+)^{C_2}\simeq B^{C_2}\otimes_{A^{C_2}}A
\longrightarrow
B\simeq  ((B\otimes (C_2)_+))^{C_2}\simeq (B\otimes_A(A\otimes (C_2)_+))^{C_2}
\]
where the arrow is induced by the map of $A^{C_2}$-modules $\alpha\colon A\to \alpha^\ast B$, where $A$ is an $A^{C_2}$-module via the restriction $A^{C_2}\to A$, and similarly for $B$.  This is an equivalence since the restrictions of $A$ and $B$ are.
The identifications of $f$ and $r$ follow by naturality and unravelling the definitions, using Example \ref{minicycloTHR} for the cyclotomic structure.
\end{proof}

\begin{cor}\label{TCRbase-change}
Let $\alpha \colon A \to B$ be a $\phi$-flat map of commutative flat $\Z/2$-equivariant ring spectra, and suppose that the restriction maps $A^{\Z/2}\to A$ and $B^{\Z/2}\to B$ are equivalences. Then there is an equaliser diagram
\[
\xymatrix{
\TCR(B;2)^{{\phi\Z/2}}\ar[r]&B^{C_2}\otimes_{A^{C_2}}(\THR(A)^{{\phi\Z/2}})^{C_2}\ar@<1ex>[rr]^-{res\otimes f}\ar@<-1ex>[rr]_-{(\nu^{-1}can)\otimes r}&& B\otimes_{A}\THR(A)^{{\phi\Z/2}}
}
\]
where $\nu\colon B\otimes_AA^{{\phi\Z/2}}\to B^{\phi\Z/2}$ is the equivalence from the $\phi$-flatness condition, and $can\colon B^{C_2}\to B^{{\phi C_2}}$ is the canonical map.
\end{cor}

\begin{rem}
One cannot conclude from Corollary \ref{TCRbase-change} that $\TCR(B;2)^{{\phi\Z/2}}$ is the base-change of $\TCR(A;2)^{{\phi\Z/2}}$, nor that it is a $B$-module. This is because the maps $f$ and $r$ computing $\TCR(A;2)^{{\phi\Z/2}}$ are $A$-linear with respect to two different $A$-module structures.
\end{rem}

\begin{cor}
Let $B$ be a perfect $\F_2$-algebra with the trivial involution. Then there is an equivalence of spectra
\[
\TCR(B;2)^{\phi\Z/2}\simeq \bigoplus_{n\geq 0}(\Sigma^{2n-1}\coker(\id+(-)^2))\oplus\Sigma^{2n}(\ker(\id+(-)^2))
\]
where $(-)^2\colon B\to B$ is the Frobenius of $B$.
\end{cor}

\begin{proof}
By Example \ref{ex:commflat} we can apply Corollary \ref{TCRbase-change} and find that the maps $r$ and $f$ computing $\TCR(B;2)^{{\phi\Z/2}}$ are, on homotopy groups, given by the same maps
\[
\bigoplus_{\begin{smallmatrix}(n,m)\\
0 \leq n,m\\
n+m=\ast\end{smallmatrix}}B\rightrightarrows\bigoplus_{\begin{smallmatrix}(n,m)\\ 0 \leq n,m\\
n+m=\ast\end{smallmatrix}}B
\]
as in the case of perfect fields of Proposition \ref{lowerFR}. The calculation then proceeds exactly as in Remark \ref{rem:geomTCRk}.
%
\end{proof}

\begin{cor}
Let $B$ be a ring with no $2$-torsion and such that $B/2$ is perfect. Then $\TCR(B;2)^{{\phi \Z/2}}$ is a wedge of Eilenberg-MacLane spectra, with homotopy groups given for all $l\geq 0$ by
\[\pi_n\TCR(B;2)^{{\phi \Z/2}}\cong\left\{
\begin{array}{ll}
B/\langle x+x^2 \vert\ x \in B \rangle & n=4l-1
\\
\ker\big(\pr+\pr^2\colon B/\langle 4(x+x^2) \vert\ x \in B \rangle \to B/2 \big) & n=4l
\\
\ker\big(\id+(-)^2\colon B/2\to B/2\big) & n=4l+1
\\
0& n=4l+2
\end{array}
\right.
\]
and where $\pi_n\TCR(B;2)^{{\phi \Z/2}}=0$ for $n\leq -2$.
\end{cor}

\begin{proof}
By Example \ref{ex:commflat} and Corollary \ref{TCRbase-change} the maps $r$ and $f$ computing $\TCR(B;2)^{{\phi\Z/2}}$ are maps
\begin{align*} r,f\colon& \bigoplus_{ \begin{smallmatrix} (n,m) \\ n>m\geq 0 \end{smallmatrix}}\Sigma^{2n+2m}(HB/2 \oplus \Sigma HB/2) \oplus \bigoplus_{n\geq 0}\Sigma^{4n}(HB/4 \oplus \Sigma HB/2) \oplus   \bigoplus_{ \begin{smallmatrix} (n,m) \\ 0 \leq n  < m \end{smallmatrix}}\Sigma^{2n+2m}(HB/2 \oplus \Sigma HB/2) \to 
\\
& \bigoplus_{ \begin{smallmatrix} (n,m) \\ n>m\geq 0 \end{smallmatrix}}\Sigma^{2n+2m}(HB/2 \oplus \Sigma HB/2) \oplus \bigoplus_{n\geq 0}\Sigma^{4n}(HB/2 \oplus \Sigma HB/2) \oplus   \bigoplus_{ \begin{smallmatrix} (n,m) \\ 0 \leq n  < m \end{smallmatrix}}\Sigma^{2n+2m}(HB/2 \oplus \Sigma HB/2).
\end{align*} 
On homotopy groups they are described by the same projections and diagonals as in the case for $\Z$ of Propositions \ref{F for Z} and \ref{R for Z}, except that $r$ is postcomposed with the root isomorphism of the perfect $\F_2$-algebra $B/2$. The same argument of the proof of \ref{geometric of TCRZ computation} gives a fibre sequence
\[
\TCR(B;2)^{{\phi \Z/2}}\longrightarrow 
\bigoplus_{n\geq 0}\Sigma^{4n}(HB/4 \oplus \Sigma HB/2) 
\xrightarrow{\bigoplus_{n \geq 0}\Sigma^{4n}\left(\begin{smallmatrix}\pr+\sqrt{\pr} &  0 \\\sqrt{\beta} & \id+\sqrt{(-)}  \end{smallmatrix}\right)}
\bigoplus_{n\geq 0}\Sigma^{4n}(HB/2 \oplus \Sigma HB/2),
\]
and in particular $\TCR(B;2)^{{\phi \Z/2}}$ splits as a wedge of Eilenberg-MacLane spectra, since the projection map in the fibre sequence is $H\Z$-linear.
Moreover by composing with the Frobenius of $B/2$, which is an isomorphism, we can trade $\sqrt{\beta}$ for $\beta$, and replace all the other roots by squares.
The homotopy groups non-congruent to $0$ modulo $4$ follow immediately from the long exact sequence on homotopy groups, and $\pi_{4l}$ is isomorphic to $\pi_0$ for all $l\geq 0$. In order to calculate $\pi_0$ we observe that the fibre of a triangular matrix such as the one above can be calculated by the iterated pullback
\[
\xymatrix@C=40pt@R=15pt{
&&\TCR(B;2)^{\phi\Z/2}\ar[dl]\ar[dr]
\\
&\fib(\pr+\pr^2)\ar[dr]^-{\iota}\ar[dl]&&P\ar[dl]_-{a}\ar[dr]^-b
\\
\ast\ar[dr]&&HB/4\ar[dl]^-{\pr+\pr^2}\ar[dr]_-{\beta}&&\Sigma HB/2\ar[dl]^-{\id+(-)^2}
\\
&HB/2&&\Sigma HB/2
}
\]
where the three squares are pullbacks. By the Mayer-Vietoris sequence of the top square we see that there is an isomorphism
\[
\pi_0\TCR(B;2)^{\phi\Z/2}\cong\ker\big( (\ker(\pr+\pr^2)\times \pi_0P\xrightarrow{\iota-a}B/4).
\big)\]
By looking at the long exact sequences induced by $\beta$ and $b$, the right square gives a commutative diagram with exact rows
\[
\xymatrix{
B/2\ar[d]^-{\id+(-)^2}\ar[r]^-{\partial}&\pi_0\fib(b)\ar[d]^{\cong}\ar[r]&\pi_0P\ar[d]^a\ar[r]&0
\\
B/2\ar[r]^4&B/8\ar[r]&B/4\ar[r]&0\rlap{\ .}
}
\]
Thus $\pi_0 P\cong (B/8)/\im (4 \circ (\id+(-)^2))=B/\langle 4(x+x^2) \vert x \in B \rangle$, and the map $a$ is the reduction modulo $4$. Thus $\pi_0\TCR(B;2)^{\phi\Z/2}$ consists of those elements $y$ of $B/\langle 4(x+x^2) \vert x \in B \rangle$ such that $y=y^2$ modulo $2$.
\end{proof}

\begin{rem}
In \S\ref{sec:TCRk2} we have computed the $\Z/2$-equivariant homotopy type of $\TRR(k;2)$ and $\TCR(k;2)$ for perfect fields $k$ of characteristic $2$. We built our proof onto our knowledge of $\THR(k)^{\phi\Z/2}$ and $\TR(k;2)$ without ever needing to know the equivariant homotopy type of $\THR(k)$.
We can in fact use the base-change results of this section to show that as a $\Z/2$-spectrum 
\[\THR(k)\simeq k \otimes_{\F_2} \THR(\F_2)\simeq\bigoplus_{n\geq 0}\Sigma^{n\rho}Hk.\]
Indeed the canonical map
\[
 k \otimes_{\F_2} \THR(\F_2) \longrightarrow \THR(k)
\]
is an equivalence on $\Z/2$-geometric fixed-points by Proposition \ref{generalphibase-change} and its proof. It is also an equivalence on underlying spectra 
by \cite[Corollary 5.5]{Wittvect}.
Finally, the equivariant homotopy type of $\THR(\F_2)$ is computed in \cite{THRmodels}.
\end{rem}

\appendix

\bibliographystyle{amsalpha}
\bibliography{bib}

\end{document}